\documentclass[10pt]{article}

\usepackage{amsmath,amssymb,amscd,amsthm,makeidx,txfonts,graphicx,url,bm}
\usepackage{a4wide,pstricks,xy}

\usepackage{youngtab}
\numberwithin{equation}{subsection}

\xyoption{all}
\input{xypic}

\newtheorem{theorem}{Theorem}[subsection]
\newtheorem{proposition}[theorem]{Proposition}
\newtheorem{corollary}[theorem]{Corollary}
\newtheorem{conjecture}[theorem]{Conjecture}
\newtheorem{lemma}[theorem]{Lemma}

\theoremstyle{remark}
\newtheorem{remark}[theorem]{Remark}
\newtheorem{example}[theorem]{Example}
\newtheorem{notation}[theorem]{Notation}

\theoremstyle{definition}
\newtheorem{definition}[theorem]{Definition}

\def\beq{\begin{eqnarray}}
\def\eeq{\end{eqnarray}}
\def\bes{\begin{eqnarray*}}
\def\ees{\end{eqnarray*}}

\def\omhat{{\bm\omega}}
\def\tauhat{{\bm \tau}}

\def\muhat{{\bm \mu}}
\def\ttauhat{{\tilde{\tauhat}}}
\def\betahat{{\bm \beta}}
\def\thetahat{{\bm \theta}}
\def\zetahat{{\bm\zeta}}

\def\xihat{{\bm \xi}}
\def\lambdahat{{\bm \lambda}}

\DeclareMathOperator{\Tr}{Tr}

\DeclareMathOperator*{\rightoverleft}{\parbox{2em}{\centerline{$\longrightarrow$}\vskip -6pt\centerline{$\longleftarrow$}}}

\def\A{{\bf A}}

\def\Proj{{\rm Proj}\,}

\def\C{\mathbb{C}}
\def\M{{\mathcal{M}}}
\def\calQ{{\mathcal{Q}}}
\def\calA{{\mathcal{A}}}

\def\calX{{\mathcal{X}}}

\def\calE{{\mathcal{E}}}
\def\calV{{\mathcal{V}}}

\def\calF{{\mathcal{F}}}

\def\calP{\mathcal{P}}

\def\x{\mathbf{x}}

\def\v{\mathbf{v}}
\def\u{\mathbf{u}}
\def\e{\mathbf{e}}
\def\w{\mathbf{w}}
\def\IC{\mathcal{IC}^{\bullet}_}
\def\pIC{\underline{\mathcal{IC}}^{\bullet}_}
\def\P{\mathcal{P}}
\def\H{\mathbb{H}}

\def\tnu{{\tilde{\nu}}}

\def\tmu{{\tilde{\mu}}}

\def\tomega{{\tilde{\omega}}}
\def\tomhat{{\tilde{\omhat}}}
\def\ttau{{\tilde{\tau}}}
\def\ttauhat{{\tilde{\tauhat}}}
\def\N{\mathbb{Z}_{\geq 0}}

\def\F{\mathbb{F}}
\def\Q{\mathbb{Q}}

\def\Tr{\rm Tr}

\def\calC{{\mathcal C}}
\def\calO{{\mathcal O}}
\def\bfO{{\bf O}}

\def\Z{\mathbb{Z}}

\def\K{\mathbb{K}}

\def\gl{{\mathfrak g\mathfrak l}}

\def\g{{\mathfrak{g}}}

\newcommand{\nc}{\newcommand}

\def\bP{{\bf P}}

\nc{\op}[1]{\mathop{\mathchoice{\mbox{\rm #1}}{\mbox{\rm #1}}
{\mbox{\rm \scriptsize #1}}{\mbox{\rm \tiny #1}}}\nolimits}
\nc{\al}{\alpha}

\nc{\ep}{\varepsilon} \nc{\ga}{\gamma} \nc{\Ga}{\Gamma}
\nc{\la}{\lambda} \nc{\La}{\Lambda} \nc{\si}{\sigma}
\nc{\Sig}{{\Gamma}} \nc{\Om}{\Omega} \nc{\om}{\omega}

\nc{\SL}{{\rm SL}} \nc{\GL}{{\rm GL}} \nc{\PGL}{{\rm PGL}}
\nc{\G}{{\rm G}}

\nc{\cpt}{{\op{cpt}}} \nc{\Dol}{{\op{Dol}}} \nc{\DR}{{\op{DR}}}
\nc{\B}{{\op{B}}} \nc{\Triv}{\op{Triv}} \nc{\Hod}{{\op{Hod}}}
\nc{\Log}{{\op{Log}}} \nc{\Exp}{{\op{Exp}}} \nc{\Est}{E_{\op{st}}}
\nc{\Hst}{H_{\op{st}}} \nc{\Left}[1]{\hbox{$\left#1\vbox to
  10.5pt{}\right.\nulldelimiterspace=0pt \mathsurround=0pt$}}
\nc{\Right}[1]{\hbox{$\left.\vbox to
  10.5pt{}\right#1\nulldelimiterspace=0pt \mathsurround=0pt$}}
\nc{\LEFT}[1]{\hbox{$\left#1\vbox to
  15.5pt{}\right.\nulldelimiterspace=0pt \mathsurround=0pt$}}
\nc{\RIGHT}[1]{\hbox{$\left.\vbox to
 15.5pt{}\right#1\nulldelimiterspace=0pt \mathsurround=0pt$}}

\nc{\bee}{{\bf E}} \nc{\bphi}{{\bf \Phi}}

\begin{document}

\title{Quiver varieties and the character ring of general linear groups over finite fields}

\author{Emmanuel Letellier \\ {\it Universit\'e de Caen} \\ {\tt letellier.emmanuel@math.unicaen.fr}}

\pagestyle{myheadings}

\maketitle

\begin{center}\emph{To Gus Lehrer and Jean Michel on the occasion of their 63th and 60th birthday}\end{center}

\begin{abstract} Given a tuple $(\calX_1,\dots,\calX_k)$ of irreducible characters of $\GL_n(\F_q)$ we define a star-shaped quiver $\Gamma$ together with a dimension vector $\v$. Assume that $(\calX_1,\dots,\calX_k)$ is \emph{generic}. Our first result is a formula which expresses the multiplicity of the trivial character in the tensor product $\calX_1\otimes\cdots\otimes\calX_k$ as the trace of the action of some  Weyl group on the intersection cohomology of some (non-affine) quiver varieties associated to $(\Gamma,\v)$. The existence of such a quiver variety is subject to some condition. Assuming that this condition is satisfied, we prove our second result: The multiplicity $\langle \calX_1\otimes\cdots\otimes\calX_k,1\rangle$ is non-zero if and only if $\v$ is a root of the Kac-Moody algebra associated with $\Gamma$.  This is somehow similar to the connection between Horn's problem and the representation theory of $\GL_n(\C)$ \cite[Section 8]{knutson}.

\end{abstract}

\newpage
\tableofcontents
\newpage

\section{Introduction} 

\subsection{Decomposing tensor products of irreducible characters}\label{motivation}

The motivation of this paper is the study of  the decomposition $$\calX_1\otimes\calX_2=\sum_\calX\langle \calX_1\otimes\calX_2,\calX\rangle\calX$$of the tensor product  $\calX_1\otimes\calX_2$ of two irreducible complex characters of $\GL_n(\F_q)$ as a sum of irreducible characters. This is equivalent to the study of the multiplicities $\langle \calX_1\otimes\calX_2\otimes\calX_3,1\rangle$ of the trivial character $1$ in $\calX_1\otimes\calX_2\otimes\calX_3$.

Although the character table of $\GL_n(\F_q)$ is known since 1955 by the work of Green \cite{green}, the computation of these multiplicities remains an open problem which does not seem to have been studied much in the literature.

When $\calX_1,\calX_2,\calX_3$ are unipotent characters, the multiplicities $\langle \calX_1\otimes\calX_2\otimes\calX_3,1\rangle$ are computed by Hiss and L\"ubeck \cite{Hiss} using CHEVIE for $n\leq 8$ and appeared to be polynomials in $q$ with positive coefficients.

Let $\chi:\GL_n(\F_q)\rightarrow\C$ be the character of the conjugation action of $\GL_n(\F_q)$ on the group algebra $\C[\gl_n(\F_q)]$. Fix a non-negative integer $g$ and put $\Lambda:=\chi^{\otimes g}$ (with $\Lambda=1$ if $g=0$). 

In this paper we describe the multiplicities $\langle \Lambda\otimes\calX_1\otimes\cdots\otimes\calX_k,1\rangle$ for \emph{generic} tuples $(\calX_1,\dots,\calX_k)$ of irreducible characters of $\GL_n(\F_q)$ in terms of representations of a certain quiver $\Gamma$ (see \S \ref{gen} for the definition of generic tuple). Although the occurence of $\Lambda$ does not seem to be very interesting from the perspective of the representation theory of $\GL_n(\F_q)$ it will appear to be more interesting for the theory of quiver representations.
 
Let us now explain how to construct the quiver  together with a dimension vector  from any tuple of irreducible characters (not necessarily generic).

We first define a type $A$ quiver together with a dimension vector from a single irreducible character $\calX$.

Consider a total ordering $\geq$ on the set $\calP$ of partitions and define a total ordering denoted again by $\geq$ on the set $\Z_{>0}\times\left(\calP-\{0\}\right)$ as follows.  If $\mu\neq\lambda$ then $(d,\mu)\geq (d',\lambda)$ if $\mu\geq\lambda$, and $(d,\lambda)\geq (d',\lambda)$ if $d\geq d'$. Denote by ${\bf T}_n$ the set of non-increasing sequences  $\omega=(d_1,\omega^1)\cdots(d_r,\omega^r)$ such that $\sum_id_i|\omega^i|=n$. 

In \S \ref{gen}, we associate to the irreducible character $\calX$ an element $\omega=(d_1,\omega^1)\cdots(d_r,\omega^r)\in{\bf T}_n$ called the type of $\calX$. The $d_i$'s are called the degrees of $\calX$. If the degrees $d_i$'s are all equal to $1$ we say that $\calX$ is \emph{split}. Let us now draw the Young diagrams of these partitions $\omega_1,\dots,\omega_r$  from the left to the right with diagram of $\omega^i$ repeated $d_i$ times (partitions being represented by the rows of the Young diagram). Let $l$ be the total number of columns and let $n_i$ be the length of the $i$-th column. We obtain a striclty decreasing sequence $\u_\omega:=(v_0=n> v_1> v_2>\dots> v_{l-1})$ by putting $v_1:=n-n_1$, $v_i:=v_{i-1}-n_i$. We obtain then a type $A_l$-quiver with dimension vector $\u_\omega$. For instance if $\calX=1$, then $\omega=(1,(1,1,\dots,1))$ and so $A_l=A_1$ and $\u_\omega=n$. If $\calX$ is the Steinberg character then $\omega=(1,(n))$ and so $A_l=A_n$ and $\u_\omega=(n,n-1,n-2,\dots,1)$. If $\calX$ is of type $(1,1)(1,1)\cdots(1,1)$, then we still have $A_l=A_n$ and $\u_\omega=(n,n-1,n-2,\dots,1)$.

Given $\omhat=(\omega_1,\dots,\omega_k)\in\left({\bf T}_n\right)^k$, we obtain (as just explained) $k$ type $A$ quivers equipped with dimension vectors $\u_{\omega_1},\dots,\u_{\omega_k}$. Gluing together the vertices labelled by $0$ of these $k$ quivers   and adding $g$ loops at the central vertex of this new quiver we get a so-called comet-shaped quiver $\Gamma_\omhat$ with $k$ legs (see picture in \S \ref{star}) together with a dimension vector $\v_\omhat$ which is determined in the obvious way by $\u_{\omega_1},\dots,\u_{\omega_k}$. 

Let $\Phi(\Gamma_\omhat)$ be the root system associated with $\Gamma_\omhat$ (see Kac \cite{kac}). Let $C_{\Gamma_\omhat}$ be the Cartan matrix of $\Gamma_\omhat$ and put $d_\omhat=2-{^t}\v_\omhat C_{\Gamma_\omhat}\v_\omhat$.

In \S \ref{multi695} we show that for every multi-type $\omhat\in\left({\bf T}_n\right)^k$, there exists a polynomial $\H_\omhat(T)\in\Q[T]$ such that for any finite field $\F_q$ and any generic tuple $(\calX_1,\dots,\calX_k)$ of irreducible characters of $\GL_n(\F_q)$ of type $\omhat$, we have 

$$
\langle\Lambda\otimes\calX_1\otimes\cdots\otimes\calX_k,1\rangle=\H_\omhat(q).
$$

In \S \ref{Q} (see above Theorem \ref{TH3}) we define the notion of \emph{admissible} multi-type. This notion arises naturally in the theory of quiver varieties.

In this paper we use the geometry of quiver varieties to prove the following theorem (see next section for more details).

\begin{theorem}Assume that $\omhat\in\left({\bf T}_n\right)^k$ is admissible.

\noindent (a) $\H_\omhat(T)\neq 0$ if and only if $\v_\omhat\in\Phi(\Gamma_\omhat)$. Moreover $\H_\omhat(T)=1$ if and only if $\v_\omhat$ is a real root.

\noindent (b) If non-zero, $\H_\omhat(T)$ is a monic polynomial of degree $d_\omhat/2$ with integer coefficients. If moreover $\omhat$ is split, then the coefficients of $\H_\omhat(T)$ are non-negative.

\label{masterconj}\end{theorem}

We will prove (see Proposition \ref{imaginary}) that if $g\geq 1$, then $\v_\omhat$ is always an imaginary root and so the second part of the assertion of (a) is relevant only when $g=0$.

The discussion and conjecture in \S \ref{CV}  together with the results of Crawley-Boevey \cite{crawleyind} imply that the  assertions (a) and (b) of the above theorem remain true in all types (not necessarily admissible).

In a future publication, we will investigate this assertion (a) by analyzing combinatorially the polynomial $\H_\omhat(T)$ which is defined in terms of Hall-Littlewood symmetric functions (see \S \ref{cauchy}).

\begin{example} We give examples  of generic tuples of irreducible characters which are not of admissible types and which satisfy (a) and (b) of the above theorem.

Assume that $g=0$ and $n=k=3$. 

For a partition $\lambda$, we denote by $R_\lambda$ the associated unipotent character of $\GL_3$. Recall that according to our parameterization (see beginning of this section), the trivial character corresponds to the partition $(1,1,1)$ and the Steinberg character to the partition $(3)$. 

For a linear character $\alpha:\F_q^{\times}\rightarrow\C^\times$ we put $R_\lambda^\alpha:=\left(\alpha\circ{\rm det}\right)\cdot R_\lambda$. This is again an irreducible character of type $(1,\lambda)$.
 
The triple $(R_\lambda^\alpha,R_\mu^\beta,R_\nu^\gamma)$ is generic if  the subgroup $\langle\alpha\beta\gamma\rangle$ of ${\rm Hom}(\F_q^\times,\C^\times)$ is of size $3$.

Assume now that $(R_\lambda^\alpha,R_\mu^\beta,R_\nu^\gamma)$ is generic (it is not admissible, see (iii) below Theorem \ref{TH5}). As mentioned earlier, the  multiplicity  $\left\langle 
R_\lambda^\alpha\otimes R_\mu^\beta\otimes R_\nu^\gamma,1\right\rangle$ depends only on $\lambda,\mu,\nu$ and not on $\alpha,\beta,\gamma$.

Put

$$R_{\lambda,\mu,\nu}:=R_\lambda^\alpha\otimes R_\mu^\beta\otimes R_\nu^\gamma.$$

We can easily verify that  the only non zero multiplicities (with unipotent type characters) are \begin{equation}\left\langle R_{(3),(3),(3)},1\right\rangle=q,\label{E6affine}\end{equation} \begin{equation}\left\langle R_{(2,1),(3),(3)},1\right\rangle=\left\langle R_{(3),(2,1),(3)},1\right\rangle=\left\langle R_{(3),(3),(2,1)},1\right\rangle=1.\end{equation}In the first case the underlying graph of $\Gamma_\omhat$ is $\tilde{E}_6$ and $\v_\omhat$ is the indivisible positive imaginary root. In the second case the underlying graph of $\Gamma_\omhat$ is the Dynkin diagram $E_6$ and $\v_\omhat$ is the positive real root $\alpha_1+\alpha_2+2\alpha_3+3\alpha_4+2\alpha_5+\alpha_6$ in the notation of \cite[PLANCHE V]{bourbaki}. Finally we can verify that there is no other pair $(\Gamma_\omhat,\v_\omhat)$ arising from $\omhat=((1,\lambda),(1,\mu),(1,\nu))$ with $\v_\omhat\in\Phi(\Gamma_\omhat)$.

\end{example}

\subsection{Quiver varieties}\label{Q}

We now introduce the quiver varieties which provide  a geometrical interpretation of  $\left\langle \Lambda\otimes\calX_1\otimes\cdots\otimes\calX_k,1\right\rangle$ for generic tuples $(\calX_1,\dots,\calX_k)$ of admissible type. 

Let $P$ be a parabolic subgroup of $\GL_n(\C)$, $L$ a Levi factor of $P$ and let $\Sigma=\sigma+C$ where $C$ is a nilpotent orbit of the Lie algebra $\mathfrak{l}$ of $L$ and where $\sigma$ is an element of the center $z_\mathfrak{l}$ of $\mathfrak{l}$. Put $$\mathbb{X}_{L,P,\Sigma}:=\{(X,gP)\in\gl_n\times (\GL_n/P)\,|\, g^{-1}Xg\in \overline{\Sigma}+\mathfrak{u}_P\}$$where  $\mathfrak{u}_P$ is the Lie algebra of the unipotent radical of $P$. We then denote by $\mathbb{X}_{L,P,\Sigma}^o$ the open subset of pairs $(X,gP)$ which verify $g^{-1}Xg\in\Sigma+\mathfrak{u}_P$.

It is known (cf. \S \ref{adjoint} for more details) that the image of the projection $\mathbb{X}_{L,P,\Sigma}\rightarrow\gl_n$ on the first coordinate is the Zariski closure $\overline{\calO}$ of an adjoint orbit $\calO$. 

We assume without loss of generality that $L$ is of the form $\prod_j\GL_{n_j}$ and that $P$ is the unique parabolic subgroup of $\GL_n$ containing the upper triangular matrices and having $L$ as a Levi factor (such a choice is only for convenience).

When $\calO$ is nilpotent regular, the varieties $\mathbb{X}_{L,P,\Sigma}$ appears in  Borho and MacPherson \cite{BM}. These varieties were also considered by Lusztig in the framework of his generalization of Springer correspondence \cite{LuIC}.

Consider triples $\{(L_i,P_i,\Sigma_i)\}_{i=1,\dots,k}$, with $\Sigma_i=\sigma_i+C_i$, as above and put ${\bf L}:=L_1\times\cdots\times L_k$, ${\bf P}:=P_1\times\cdots\times P_k$, ${\bf \Sigma}:=\Sigma_1\times\cdots\times\Sigma_k$ and ${\bf C}:=C_1\times\cdots\times C_k$.

Let $(\calO_1,\dots,\calO_k)$ be the tuple of adjoint orbits of $\gl_n(\C)$ such that the image  of $\mathbb{X}_{L_i,P_i,\Sigma_i}\rightarrow\gl_n$ is $\overline{\calO}_i$. 

We say that the pair $({\bf L,\Sigma})$ is \emph{generic} if the tuple  $(\calO_1,\dots,\calO_k)$ is \emph{generic}. The existence of  generic tuples of adjoint orbits with prescribed multiplicities of eigenvalues is subject to some restriction (cf. \S \ref{generic} for more details).

We assume now that $({\bf L,\Sigma})$ is generic.

Fix a non-negative integer $g$, put $\mathbb{O}_{\bf L,P,\Sigma}=(\gl_n)^{2g}\times\mathbb{X}_{\bf L,P,\Sigma}$, $\mathbb{O}_{\bf L,P,\Sigma}^o=(\gl_n)^{2g}\times\mathbb{X}_{\bf L,P,\Sigma}^o$ and define 

$$\mathbb{V}_{\bf L,P,\Sigma}:=\left\{\left(\left.A_1,B_1,\dots,A_g,B_g,(X_1,\dots,X_k,g_1P_1,\dots,g_kP_k)\right)\in\mathbb{O}_{\bf L,P,\Sigma}\,\right|\, \sum_j[A_j,B_j]+\sum_iX_i=0\right\}.$$Put ${\bf O}:=(\gl_n)^{2g}\times\overline{\calO}_1\times\cdots\times\overline{\calO}_k$, ${\bf O}^o:=(\gl_n)^{2g}\times\calO_1\times\cdots\times\calO_k$ and define $$\calV_{\bf O}:=\left\{\left(\left.A_1,B_1,\dots,A_g,B_g,X_1,\dots,X_k\right)\in {\bf O}\,\right|\, \sum_j[A_j,B_j]+\sum_i X_i=0\right\}.$$Let $\rho:\mathbb{V}_{\bf L,P,\Sigma}\rightarrow \calV_{\bf O}$ be the projection on the first $2g+k$ coordinates.

The group $\GL_n$ acts on  $\mathbb{V}_{\bf L,P,\Sigma}$ (resp. on $\calV_{\bf O}$) diagonally by conjugating the first $2g+k$ coordinates and by left multiplication of the last $k$-coordinates (resp. diagonally by conjugating the $2g+k$ coordinates). Since the tuple $(\calO_1,\dots,\calO_k)$ is generic, this action induces a set-theoritically free action of $\PGL_n$ on both $\mathbb{V}_{\bf L,P,\Sigma}$ and $\calV_{\bf O}$. The $\PGL_n$-orbits of these two spaces are then all closed. Consider the affine GIT quotient

$$\calQ_{\bf O}:=\calV_{\bf O}/\PGL_n={\rm Spec}\left(\C[\calV_{\bf O}]^{\PGL_n}\right).$$The quotient map $\calV_\bfO\rightarrow\calQ_\bfO$ is actually a principal $\PGL_n$-bundle in the \'etale topology. Since $\mathbb{V}_{\bf L,P,\Sigma}$ is projective over $\calV_\bfO$, by a result of Mumford \cite{mumford} the categorical quotient $\mathbb{Q}_{\bf L,P,\Sigma}$ of $\mathbb{V}_{\bf L,P,\Sigma}$ by $\PGL_n$ exists and the quotient map $\mathbb{V}_{\bf L,P,\Sigma}\rightarrow\mathbb{Q}_{\bf L,P,\Sigma}$ is also a principal $\PGL_n$-bundle. 

We will see that we can identify $\calQ_\bfO$ and $\mathbb{Q}_{\bf L,P,\Sigma}$ with quiver varieties $\mathfrak{M}_\xi(\v_\bfO)$ and $\mathfrak{M}_{\xi,\theta}(\v_{\bf L,P,\Sigma})$ made out of the same comet-shaped quiver $\Gamma_{\bf L,P,\Sigma}=\Gamma_{\bf O}$ equipped with (possibly different) dimension vector $\v_\bfO$ and $\v_{\bf L,P,\Sigma}$ (here we use Nakajima's notation, cf. \S \ref{uquiver}). The variety $\calQ_{\bf O}$ is also isomorphic to  the image $\pi\left(\mathfrak{M}_{\xihat,\thetahat}(\v_{\bf L,P,\Sigma})\right)$ of $\pi:\mathfrak{M}_{\xihat,\thetahat}(\v_{\bf L,P,\Sigma})\rightarrow\mathfrak{M}_\xihat(\v_{\bf L,P,\Sigma})$. 

The identification of $\calQ_\bfO$ with the quiver variety $\mathfrak{M}_\xi(\v_\bfO)$ is due to Crawley-Boevey \cite{crawley-mat} and is also available in the non-generic case (see \S \ref{star}). Although it may not be in the literature, the identification of $\mathbb{Q}_{\bf L,P,\Sigma}$ with $\mathfrak{M}_{\xi,\theta}(\v_{\bf L,P,\Sigma})$ is then quite natural to consider.

Under the identification $\calQ_\bfO\simeq\mathfrak{M}_\xi(\v_\bfO)$, the open subset $\calQ_{\bf O}^{\,o}\subset\calQ_\bfO$ defined as the image of $$\calV_{\bf O}^o:=\calV_{\bf O}\cap{\bf O}^o$$in $\calQ_{\bf O}$  corresponds to the subset $\mathfrak{M}_\xihat^s(\v_\bfO)\subset\mathfrak{M}_\xihat(\v_\bfO)$ of simple representations. The image $\mathbb{Q}_{\bf L,P,\Sigma}^o$ of $$\mathbb{V}_{\bf L,P,\Sigma}^o:=\mathbb{V}_{\bf L,P,\Sigma}\cap \mathbb{O}_{\bf L,P,\Sigma}^o$$in $\mathbb{Q}_{\bf L,P,\Sigma}$  corresponds to the subset $\mathfrak{M}_{\xihat,\thetahat}^s(\v_{\bf L,P,\Sigma})\subset\mathfrak{M}_{\xihat,\thetahat}(\v_{\bf L,P,\Sigma})$ of $\thetahat$-stable representations.

The generic quiver variety $\mathbb{Q}_{\bf L,P,\Sigma}$ (which does not seem to have been considered in the literature before) and $\calQ_{\bf O}$ will be one of the main focus of this paper. 

If $\calV_{\bf O}\neq\emptyset$, the varieties $\mathbb{Q}_{\bf L,P,\Sigma}^o$ and $\calQ_{\bf O}^{\,o}$ are both non-empty irreducible nonsingular dense open subsets of $\mathbb{Q}_{\bf L,P,\Sigma}$ and $\calQ_{\bf O}$ respectively. The irreducibility of $\calQ_\bfO$ follows from a more general result due to Crawley-Boevey (see Theorem \ref{irrquiver}). The irreducibility of $\mathbb{Q}_{\bf L,P,\Sigma}$ (see Theorem \ref{strat}) seems to be new and our proof uses Theorem \ref{HLRpure} and Crawley-Boevey's result in Theorem \ref{irrquiver}. The equivalence between the non-emptyness of $\calQ_\bfO$ and that of $\calQ_\bfO^o$ is not stated explicitely in Crawley-Boevey's paper but our proof follows very closely various arguments which are due to him. More precisely we have the following result which is important for this paper.

\begin{theorem} The following assertions are equivalent.
 
(i) The variety $\calQ_\bfO^o$ is not empty.

(ii) The variety $\calQ_\bfO$ is not empty.

(iii) $\v_\bfO\in\Phi(\Gamma)$.
\label{TH6}\end{theorem}

Let us discuss this theorem. Say that an element $X$ in $\calV_\bfO^o$ is \emph{irreducible} if there is no  non-zero proper subspace of $\C^n$ which is preserved by all the coordinates of $X$. The existence of irreducible elements in $\calV_\bfO^o$ was  studied by Kostov \cite{Kostov} who calls it the (additive) Deligne-Simpson problem (in \cite{Kostov} the tuple $(\calO_1,\dots,\calO_k)$ is not necessarily generic). 
Later on, Crawley-Boevey \cite{crawley-mat} reformulated Kostov's answer to the Deligne-Simpson problem in terms of roots of $\Gamma$. This reformulation involves general results of Crawley-Boevey on quiver varieties (see \S \ref{uquiver} for more details) and his identification of $\calQ_\bfO$ with $\mathfrak{M}_\xihat(\v_\bfO)$. Our proof of Theorem \ref{TH6} consists of working out in the generic case Crawley-Boevey's results on the Deligne-Simpson problem.

For a pair $(L,\Sigma)$ as  above, we put $$W(L,\Sigma):=\{n\in N_{\GL_n}(L)\,|\, n\Sigma n^{-1}=\Sigma\}/L.$$The group $W(L,\Sigma)$ acts on the complex $p_*(\pIC {\mathbb{X}_{L,P,\Sigma}})$ where $p:\mathbb{X}_{L,P,\Sigma}\rightarrow \gl_n$ is the projection on the first coordinate, and $\pIC {\mathbb{X}_{L,P,\Sigma}}$ is the simple perverse sheaf with coefficient in the constant local system $\C$.

From this, we find an action of  $$W({\bf L,\Sigma}):=W(L_1,\Sigma_1)\times\cdots\times W(L_k,\Sigma_k)$$on the complex $\left(\rho/_{\PGL_n}\right)_*\left(\pIC {\mathbb{Q}_{\bf L,P,\Sigma}}\right)$ and so on the hypercohomology $\mathbb{H}_c^i\left(\mathbb{Q}_{\bf L,P,\Sigma},\IC {\mathbb{Q}_{\bf L,P,\Sigma}}\right)$ which we take as a definition for the compactly supported intersection cohomology $IH_c^i\left(\mathbb{Q}_{\bf L,P,\Sigma},\C\right)$.

From the theory of quiver varieties, we have $IH_c^i\left(\mathbb{Q}_{\bf L,P,\Sigma},\C\right)=0$ for odd $i$. Let us then consider the polynomials 

$$P_c^{\bf w}\left(\mathbb{Q}_{\bf L,P,\Sigma},q\right):=\sum_i{\rm Tr}\,\left({\bf w}\,\left|\, IH_c^{2i}\left(\mathbb{Q}_{\bf L,P,\Sigma},\C\right)\right.\right) q^i,$$with ${\bf w}\in W({\bf L,\Sigma})$.

As explained in \S \ref{adjoint} to each pair $(L,C)$ with $L=\prod_{i=1}^r\GL_{n_i}\subset \GL_n$ and $C$ a nilpotent orbit of $\bigoplus_{i=1}^r\gl_{n_i}$ corresponds a unique sequence of partitions  $$\tomega=\underbrace{\omega^1\cdots\omega^1}_{a_1}\underbrace{\omega^2\cdots\omega^2}_{a_2}\cdots\underbrace{\omega^l\cdots\omega^l}_{a_l}$$with $\omega^1\geq\omega^2\geq\cdots\geq\omega^l$ and $\omega^j\neq\omega^s$ if $j\neq s$. 

The group $W(L,C)$ is then isomorphic to $\prod_{j=1}^lS_{a_j}$ where $S_d$ denotes the symmetric group in $d$ letters.

The decomposition of the coordinates of an element  $w\in W(L,C)\simeq\prod_{j=1}^lS_{a_j}$ as a product of disjoint cycles provides a partition $(d_j^1,d_j^2,\dots,d_j^{r_j})$ of $a_j$ for each $j$, and so defines a unique type $$\omega=(d_1^1,\omega^1)\cdots(d_1^{r_1},\omega^1)(d_2^1,\omega^2)\cdots(d_2^{r_2},\omega^2)\cdots(d_l^1,\omega^l)\cdots(d_l^{r_l},\omega^l)\in{\bf T}_n.$$We thus have a surjective map from the set of triples $(L,C,w)$ with $w\in W(L,C)$ to the set ${\bf T}_n$.

Note that $W({\bf L},{\bf \Sigma})\subset W({\bf L},{\bf C})$.

Let $\w\in W({\bf L},{\bf \Sigma})$. The datum $({\bf L},{\bf C},\w)$ defines thus a multi-type $\omhat=(\omega_1,\dots,\omega_k)\in\left({\bf T}_n\right)^k$. We call \emph{admissible} the multi-types arising in this way from generic pairs $({\bf L,\Sigma})$.

Let $(\calX_1,\dots,\calX_k)$ be a generic tuple of irreducible characters of $\GL_n(\F_q)$ of type $\omhat$ (generic tuples of irreducible characters of a given type always exist assuming that the characteristic of $\F_q$ and $q$ are large enough).  The pair $(\Gamma_\omhat,\v_\omhat)$ defined in \S \ref{motivation} is the same as the pair $(\Gamma_{\bf L,P,\Sigma},\v_{\bf L,P,\Sigma})$ defined from  $({\bf L,P,\Sigma})$, moreover the integer $d_\omhat$  equals ${\rm dim}\,\mathbb{Q}_{\bf L,P,\Sigma}$. 

\begin{theorem} We have:

$$P_c^{\bf w}\left(\mathbb{Q}_{\bf L,P,\Sigma},q\right)=q^{\frac{1}{2}{\rm dim}\,\mathbb{Q}_{\bf L,P,\Sigma}}\langle \Lambda\otimes\calX_1\otimes\cdots\otimes\calX_k,1\rangle.$$

\label{TH3}\end{theorem}

If ${\bf w}=1$ and if the adjoint orbits $\calO_1,\dots,\calO_k$ are semisimple in which case $\mathbb{Q}_{\bf L,P,\Sigma}\simeq\calQ_{\bf O}$, the theorem is proved in \cite{hausel-letellier-villegas}.

One of the consequence of Theorem \ref{TH3} is an explicit formula for $P_c^{\bf w}\left(\mathbb{Q}_{\bf L,P,\Sigma},q\right)$ in terms of Hall-Littlewood symmetric functions (cf. \S \ref{MSF}).

Note that if for each $i=1,\dots,k$, we have $C_{\GL_n}(\sigma_i)=L_i$, then the projection $\mathbb{X}_{L_i,P_i,\Sigma_i}\rightarrow\overline{\calO}_i$ is an isomorphism and so is the map $\rho/_{\PGL_n}:\mathbb{Q}_{\bf L,P,\Sigma}\rightarrow\calQ_{\bf O}$. Hence our main results will give in particular explicit formulas for the Poincar\'e polynomial $P_c\left(\calQ_{\bf O},q\right)$ where we write $P_c$ instead of $P_c^{\bf w}$ when ${\bf w}=1$.

Let $\calA_{({\bf L,C})}$ be the set of $\sigma=(\sigma_1,\dots,\sigma_k)\in z_{\mathfrak{l}_1}\times\cdots\times z_{\mathfrak{l}_k}$ such that the pair $({\bf L},\sigma+{\bf C})$ is generic. 

It follows from Theorem \ref{TH3} that $P_c\left(\mathbb{Q}_{\bf L,P,\Sigma},q\right)$ depends only on $({\bf L,C})$ and not on $\sigma\in\calA_{({\bf L,C})}$. 

We say that a generic tuple $(\calX_1,\dots,\calX_k)$ of irreducible characters is \emph{admissible} if it is of admissible type.

From Theorem \ref{TH3} and Theorem \ref{TH6}, we prove Theorem \ref{masterconj}, namely:

\begin{theorem} Let $(\calX_1,\dots,\calX_k)$ be an admissible generic tuple of irreducible characters of $\GL_n(\F_q)$ of type $\omhat$.

\noindent (a) We have $\left\langle \Lambda\otimes\mathcal{X}_1\otimes\cdots\otimes\mathcal{X}_k,1\right\rangle\neq 0$ if and only if $\v_\omhat\in\Phi(\Gamma_\omhat)$. Moreover $\left\langle \Lambda\otimes\mathcal{X}_1\otimes\cdots\otimes\mathcal{X}_k,1\right\rangle=1$ if and only if $\v_\omhat$ is real.

\noindent (b) If $\v_\omhat\in\Phi(\Gamma_\omhat)$, the multiplicity $\left\langle \Lambda\otimes\mathcal{X}_1\otimes\cdots\otimes\mathcal{X}_k,1\right\rangle$ is a monic polynomial in $q$ of degree $d_\omhat/2$ with integer coefficients. If moreover $\w=1$, then it has positive coefficients.

\label{TH5}\end{theorem}

Now let us see some examples of generic tuples $(\calX_1,\dots,\calX_k)$ of irreducible characters which are not admissible. This is equivalent of giving examples of triples $({\bf L,C,w})$ for which there is no $\sigma\in\calA_{({\bf L,C})}$ such that $\w\in W({\bf L},\sigma+{\bf C})$.

The condition for the existence of such a $\sigma$  is subject to some restrictions which can be worked out explicitely using \S \ref{generic}. Let us see the explicit situations (i), (ii) and (iii) below.

(i) Assume that ${\bf L}$ is a maximal torus (in which case ${\bf C}$ is the trivial nilpotent orbit) and that the coordinates of ${\bf w}$ are all $n$-cycles. Then ${\bf w}$ belongs to a subgroup $W({\bf L},\sigma+{\bf C})$ of $W({\bf L,C})=W({\bf L})$ if and only if the coordinates of ${\bf \sigma}=(\sigma_1,\dots,\sigma_k)$ are all scalar matrices. But such a $\sigma$ does not belong to $\calA_{({\bf L,C})}$.

(ii)  When the dimension vector $\v$ of the comet-shaped quiver $\Gamma$ is divisible (i.e., the gcd of its coordinates is greater than $1$), then $\calA_{({\bf L,C})}=\emptyset$.  

(iii) If ${\bf L}=(\GL_n)^k$, then we also have $\calA_{({\bf L,C})}=\emptyset$.

When ${\bf C}=\{0\}$, then $\calA_{({\bf L,C})}\neq\emptyset$ if and only if $\v_\omhat$ is indivisible. This implies that a generic tuple of split semisimple irreducible characters is admissible if and only if $\v_\omhat$ is indivisible.

\subsection{Character varieties: A conjecture}\label{CV}

Now we propose a conjectural geometrical interpretation of $\left\langle \Lambda\otimes\mathcal{X}_1\otimes\cdots\otimes\mathcal{X}_k,1\right\rangle$ for any generic tuple $(\calX_1,\dots,\calX_k)$.

Let $P$ be a parabolic subgroup of $\GL_n(\C)$, $L$ a Levi factor of $P$ and let $\Sigma=\sigma C$ where $C$ is a unipotent conjugacy class of $L$ and where $\sigma$ is an element of the center $Z_L$ of $L$. Put $$\mathbb{Y}_{L,P,\Sigma}:=\left\{(x,gP)\in\GL_n\times (\GL_n/P)\,\left|\, g^{-1}xg\in \overline{\Sigma}.U_P\right.\right\}$$where $U_P$ is the unipotent radical of $P$. The variety $\mathbb{Y}_{L,P,\Sigma}$ is the multiplicative analogue of $\mathbb{X}_{L,P,\Sigma}$. 

We choose a tuple $(\mathfrak{O}_1,\dots,\mathfrak{O}_k)$ of conjugacy classes of $\GL_n(\C)$ and for each $i=1,\dots,k$ we let $\tilde{\mathfrak{O}}_i$ be the conjugacy class of the semisimple part of an element in $\mathfrak{O}_i$. We say that the tuple $(\mathfrak{O}_1,\dots,\mathfrak{O}_k)$ is \emph{generic} if $\prod_{i=1}^k{\rm det}\, (\mathfrak{O}_i)=1$ and if $V$ is a subspace of $\C^n$ which is stable by some $x_i\in\tilde{\mathfrak{O}}_i$ (for each $i$) such that $$\prod_{i=1}^k{\rm det}\,(x_i|_V)=1$$then either $V=0$ or $V=\C^n$. 
 Unlike the additive case, generic tuples of conjugacy classes always exist (the multiplicities of the eigenvalues being prescribed). For instance, while we can not form generic tuples of adjoint orbits of nilpotent type, we can always form generic tuples of conjugacy classes of unipotent type as follows. Let $\zeta$ be a primitive $n$-th root of unity, and $\mathfrak{O}_1=\zeta C_1$, $\mathfrak{O}_2=C_2,\dots,\mathfrak{O}_k=C_k$ where $C_1,\dots, C_k$ are unipotent conjugacy classes, then $(\mathfrak{O}_1,\dots,\mathfrak{O}_k)$ is generic. 

For each $i=1,\dots,k$, let $(L_i,P_i,\Sigma_i)$ be such that the image of the projection $\mathbb{Y}_{L_i,P_i,\Sigma_i}\rightarrow\gl_n$ is $\overline{\mathfrak{O}}_i$. As in \S \ref{Q}, we define ${\bf L,P,\Sigma,C}$ and we say that $({\bf L,\Sigma})$ is \emph{generic} if the tuple $(\mathfrak{O}_1,\dots,\mathfrak{O}_k)$ is generic which we now assume.  We define the multiplicative analogue of $\mathbb{V}_{\bf L,P,\Sigma}$ as
\vspace{.2cm}

$\mathbb{U}_{\bf L,P,\Sigma}:=$
$$\left\{\left(\left.a_1,b_1,\dots,a_g,b_g,(x_1,\dots,x_k,g_1P_1,\dots,g_kP_k)\right)\in(\GL_n)^{2g}\times\mathbb{Y}_{\bf L,P,\Sigma}\,\right|\, (a_1,b_1)\cdots(a_g,b_g)x_1\cdots x_k=1\right\}$$where $(a,b)$ denotes the commutator $aba^{-1}b^{-1}$. As in the quiver case,  the genericity condition ensures that the group $\PGL_n$ acts freely on $\mathbb{U}_{\bf L,P,\Sigma}$. Then consider the quotient $\mathbb{M}_{\bf L,P,\Sigma}=\mathbb{U}_{\bf L,P,\Sigma}/\PGL_n$. The projection $\mathbb{U}_{\bf L,P,\Sigma}\rightarrow \left(\GL_n\right)^{2g+k}$ on the $2g+k$ first coordinates induces a morphism from $\mathbb{M}_{\bf L,P,\Sigma}$ onto the affine GIT quotient

$$\M_{\bf \mathfrak{O}}:=\left\{(a_1,b_1,\dots,a_g,b_g,x_1,\dots,x_k)\in\left(\GL_n\right)^{2g}\times\overline{\mathfrak{O}}_1\times\cdots\times \overline{\mathfrak{O}}_k\,\left|\, \prod(a_i,b_i)\prod x_j=1\left.\right.\right\}\right/\PGL_n.$$

\begin{remark} If $S_g$ is a compact Riemann surface of genus $g$ with punctures $p=\{p_1,\dots,p_k\}\subset S_g$, then $\M_{\bf \mathfrak{O}}$ can be identified (hence the name of character varieties) with the affine GIT quotient

$$\left\{\rho\in{\rm Hom}\,\left(\pi_1(S_g\backslash p),\GL_n\right)\,\left|\, \rho(\gamma_i)\in \overline{\mathfrak{O}}_i\left.\right.\right\}\right/\PGL_n,$$where $\gamma_i$ is the class of a simple loop around $p_i$ with orientation compatible with that of $S_g$.

\end{remark}

Unlike quiver varieties, the mixed Hodge structure on $IH_c^k\left(\mathbb{M}_{\bf L,P,\Sigma},\C\right)$ is not pure (see for instance \cite{hausel-letellier-villegas} in the case where the conjugacy classes $\mathfrak{O}_i$ are semisimple).

We let $W_\bullet$ be the weight filtration on $IH_c^k\left(\mathbb{M}_{\bf L,P,\Sigma},\C\right)$ and put $$H^{i,k}\left(\mathbb{M}_{\bf L,P,\Sigma}\right):=W_iIH_c^k\left(\mathbb{M}_{\bf L,P,\Sigma},\C\right)/W_{i-1}IH_c^k\left(\mathbb{M}_{\bf L,P,\Sigma},\C\right).$$The action of $W({\bf L,\Sigma})$ preserves the weight filtration and so, for ${\bf w}\in W({\bf L,\Sigma})$, we may consider the mixed Poincar\'e polynomial

$$H_c^{\bf w}\left(\mathbb{M}_{\bf L,P,\Sigma};q,t\right):=\sum_{i,k}{\rm Tr}\,\left({\bf w}\,\left|\, H^{i,k}\left(\mathbb{M}_{\bf L,P,\Sigma}\right)\right.\right)q^it^k$$and its \emph{pure part}

$$PH_c^{\bf w}\left(\mathbb{M}_{\bf L,P,\Sigma},t\right):=\sum_i{\rm Tr}\,\left({\bf w}\,\left|\, H^{i,i}\left(\mathbb{M}_{\bf L,P,\Sigma}\right)\right.\right)t^i.$$

Recall that ${\bf \Sigma}={\bf \sigma}{\bf C}$ with ${\bf C}$ a unipotent conjugacy class of ${\bf L}$ and  $\sigma\in Z_{\bf L}$. 

Let $\w\in W({\bf L,\Sigma})$. As above Theorem \ref{TH3}, we can define a type $\omhat\in({\bf T}_n)^k$ from $({\bf L,C,w})$. Let $(\calX_1,\dots,\calX_k)$ be a generic tuple of irreducible characters of $\GL_n(\F_q)$  of type $\omhat$.

\begin{conjecture} We have \begin{equation}PH_c^{\bf w}\left(\mathbb{M}_{\bf L,P,\Sigma},\sqrt{q}\right)=q^{\frac{1}{2}{\rm dim}\,\mathbb{M}_{\bf L,P,\Sigma}}\left\langle\Lambda\otimes\calX_1\otimes\cdots\otimes\calX_k,1\right\rangle.\label{conjfor}\end{equation}
\label{conjecture}\end{conjecture}

If ${\bf w}=1$ and if the conjugacy classes $\mathfrak{O}_i$ are semisimple, in which case $\mathbb{M}_{\bf L,P,\Sigma}\simeq \M_{\mathfrak{O}}$, then this conjecture is already in \cite{hausel-letellier-villegas}.

Now put $\xi:=(\zeta \cdot 1,1,\dots,1)\in \left(Z_{\GL_n}\right)^k$ where $\zeta$ is a primitive $n$-th root of unity. Then for any triple $({\bf L,C,w})$ with $\w\in W({\bf L,C})$ the pair $({\bf L,\xi C})$ is always generic and $\w\in W({\bf L,\xi C})=W({\bf L,C})$. Hence Conjecture  \ref{conjecture} implies that for any generic tuple $(\calX_1,\dots,\calX_k)$ of irreducible characters there exists a triple $({\bf L,C,w})$ with $\w\in W({\bf L,C})$ such that if we put ${\bf \Sigma}:=\xi {\bf C}$, then Formula (\ref{conjfor}) holds.

Put ${\bf C}':={\bf C}-1$ and assume that there exists $\sigma'\in\calA_{({\bf L,C}')}$ such that $C_{\GL_n}(\sigma)=C_{\GL_n}(\sigma')$.  Then Conjecture \ref{conjecture} together with Theorem \ref{TH3} implies the following conjecture.

\begin{conjecture}We have $$PH_c^{\bf w}\left(\mathbb{M}_{\bf L,P,\Sigma},\sqrt{q}\right)=P_c^{\bf w}\left(\mathbb{Q}_{\bf L,P,\Sigma'},q\right).$$ \end{conjecture}

In the case where the adjoint orbits $\calO_1,\dots,\calO_k$ and the conjugacy classes $\mathfrak{O},\dots, \mathfrak{O}_k$ are semisimple and ${\bf w}=1$, then this conjecture is due to T. Hausel. If $g=0$, he actually conjectured that the identity between the two polynomials is realized by the Riemann-Hilbert monodromy map $\calQ_{\bf O}\rightarrow\M_{\mathfrak{O}}$.

In \cite{hausel-letellier-villegas} we gave a conjectural formula for the mixed Poincar\'e polynomial of $\M_\mathfrak{O}$ in terms of Macdonald polynomials when $\mathfrak{O}_1,\dots,\mathfrak{O}_k$ are semisimple. We will discuss the generalization of this conjecture for the twisted mixed Poincar\'e polynomial $H_c^{\bf w}\left(\mathbb{M}_{\bf L,P,\Sigma};q,t\right)$ in a forthcoming paper.

\paragraph{Acknowledgements.} I am very grateful to P. Satg\'e for helpful discussions on some parts of this paper and to the referee for his very careful reading and his suggestions to improve the writing of this paper. This work started during the special semester entitled ``Algebraic Lie Theory'' at the Newton Institute (Cambridge, 2009). I would like to thank the organisers for the invitation and the institute's staff for their kindness. This work is supported by ANR-09-JCJC-0102-01.

\section{Preliminaries on geometric invariant theory}\label{preliminaries}

In this section,  $\K$ is an algebraically closed field of arbitrary characteristic. 

%By an \emph{algebraic variety} over $\K$ we shall mean a separated reduced scheme of finite type over $\K$.

In the following the letter $G$ denotes a connected reductive algebraic groups over $\K$.

We review the construction by Mumford \cite{mumford} of GIT quotients.

\subsection{GIT quotients}\label{GIT}
For an algebraic variety $X$ over $\K$ we denote by $\K[X]:=H^0(X,\calO_X)$ the $\K$-algebra of regular functions on $X$. Let $G$ acts on $X$ and let $\sigma:G\times X\rightarrow X$, $pr_2:G\times X\rightarrow X$ denote respectively the $G$-action and the projection, then a $G$-linearization of a line bundle $L$ over $X$ is an isomorphism $\Phi:\sigma^*(L)\simeq pr_2^*(L)$ satisfying a certain co-cycle condition (see Mumford \cite{mumford}). The isomorphism $\Phi$ defines a linear action of $G$ on the space of sections $H^0(X,L)$ as $(g\cdot s)(x)=g\cdot s(g^{-1}\cdot x)$. We denote by $H^0(X,L)^G$ the space of $G$-invariant sections.

Fix a $G$-linearization $\Phi$ of $L$ and for an integer $n$, put $L(n):=L^{\otimes\,n}$. A point $x\in X$ is \emph{semistable} (with respect to $\Phi$) if there exists $m>0$ and $s\in H^0(X,L(m))^G$ such that $X_s:=\{y\in X\,|\, s(y)\neq 0\}$ is affine and contains $x$. If moreoever the $G$-orbits of $X_s$ are closed in $X_s$ and the stabilizer $C_G(x)$ of $x$ in $G$ is finite, then $x$ is said to be \emph{stable}.

We denote by $X^{ss}(\Phi)$ (resp. $X^s(\Phi)$) the open $G$-invariant subset of semistable (resp. stable) points of $X$. 

Let $q:X^{ss}(\Phi)\rightarrow X/\!/_\Phi G$ denote the GIT quotient map defined by Mumford \cite[Theorem 1.10]{mumford}. It is defined by glueing together the affine quotient maps $X_s\rightarrow X_s/\!/G:={\rm Spec}\,\left(\K[X_s]^G\right)$ where $s$ runs over the set of sections $H^0(X,L(m))^G$, with $m>0$, such that $X_s$ is affine.

We will use the following well-known properties of $q$.

\begin{theorem} 
 
\noindent (1) The quotient $q$ is a categorical quotient (in the category of algebraic varieties).

\noindent (2) If $x,y\in X^{ss}(\Phi)$, we have $q(x)=q(y)$ if and only if $\overline{G\cdot x}\cap \overline{G\cdot y}\neq \emptyset$.

\noindent (3) If $U$ is an $q$-saturated (i.e. $q^{-1}q(U)=U$) $G$-stable open subset of $X^{ss}(\Phi)$, then $q(U)$ is an open subset of $X/\!/_\Phi G$ and the restriction $U\rightarrow q(U)$ is a categorical quotient.

\noindent (4) Let $F$ be a closed $G$-stable subset of $X^{ss}(\Phi)$ then $q\left(F\right)$ is closed in $X/\!/_\Phi G$. 

\noindent (5) There is an ample line bundle $M$ on $X/\!/_\Phi G$ such that $q^*(M)\simeq L(n)$ for some $n$.
\label{goodquotient}\end{theorem}

The theorem can be found for instance in Mumford \cite{mumford} or in Dolgachev \cite[Theorem 8.1, 6.5]{Dolg}.

Since the Zariski closure of a $G$-orbit contains always a closed orbit, the assertion (2) shows that $X/\!/_\Phi G$ parameterizes the closed orbits of $X^{ss}(\Phi)$. If we identify $X/\!/_\Phi G$ with the set of closed orbits of $X^{ss}(\Phi)$, the map $q$ sends an orbit $\calO$ of $X^{ss}(\Phi)$ to the unique closed orbit contained in $\overline{\calO}$.

Let $G'$ be another connected reductive algebraic group over $\K$ acting on $X$. Assume that the two actions of $G$ and $G'$ on $X$ commutes. Put $G''=G\times G'$ and assume that there is a $G''$-linearization $\Phi''$ of $L$ extending $\Phi$. Denote by $\Phi'$ the  $G'$-linearization on $L$ obtained by restricting $\Phi''$ to $G'\times X$.  Let $\pi'':X^{ss}(\Phi'')\rightarrow X/\!/_{\Phi''}G''$ and $\pi':X^{ss}(\Phi')\rightarrow X/\!/_{\Phi'}G'$ be the quotient maps. Since the actions of $G$ and $G'$ commute, the group $G$ acts on the spaces $H^0(X,L(n))^{G'}$ and so the quotient map $\pi'$ is $G$-equivariant. Also the ample line bundle $M$ on  $X/\!/_{\Phi'}G'$ constructed in \cite[Proof of Theorem 8.1]{Dolg} such that $(\pi')^*(M)\simeq L(n)$ is $G$-equivariant and there is a $G$-linearization $\Psi$ of $M$ such that $(\pi')^*(\Psi)=\Phi(n)$.

\begin{proposition} Assume that the inclusion $X^{ss}(\Phi'')\subset X^{ss}(\Phi')$ is an equality and put $Z=X/\!/_{\Phi'}G'$. Then there is a canonical isomorphism $X/\!/_{\Phi''}G''\simeq Z/\!/_{\Psi}G$.\label{quotientransi}\end{proposition}

\begin{proof} If $X$ is affine clearly $X/\!/G''={\rm Spec}\,\K[X]^{G''}\simeq {\rm Spec}\,\left(\K[X]^{G'}\right)^G=(X/\!/G')/\!/G$. Hence the proposition follows  from the construction of GIT quotients by glueing afffine quotients.\end{proof}

Let $\psi: G\times X\rightarrow X\times X$, $(g,x)\mapsto (g\cdot x,x)$. According to Mumford (see \cite[Definition 0.6]{mumford} or \cite[\S 6]{Dolg}) we say that a morphism $\phi:X\rightarrow Y$ of algebraic varieties is a \emph{geometric quotient} (of $X$ by $G$) if the following conditions are satisfied:

(i) $\phi$ is surjective and constant on $G$-orbits,

(ii) the image of $\psi$ is $X\times_YX$,

(iii) $U\subset Y$ is open if and only if $\phi^{-1}(U)$ is open,

(iv) for any open subset $U$ of $Y$, the natural homomorphism $H^0(U,\calO_Y)\rightarrow H^0(\phi^{-1}(U),\calO_X)$ is an isomorphism onto the subring $H^0(\phi^{-1}(U),\calO_X)^G$ of $G$-invariant sections.

A geometric quotient is a categorical quotient, hence if it exists it is unique. The condition (ii) says that $Y$ parameterizes the $G$-orbits of $X$ and so we will sometimes use the notation $X/G$ to denote the geometric quotient of $X$ by $G$.

Recall that the restriction $X^s(\Phi)\rightarrow q(X^s(\Phi))$ of $q$ is a geometric quotient $X^s(\Phi)\rightarrow X^s(\Phi)/G$.

Unless specified, the principal $G$-bundles we will consider in throughout this paper will be with respect to the \'etale topology.

\begin{lemma}A geometric quotient $\pi:X\rightarrow Y$ is a principal $G$-bundle  if and only if $\pi$ is flat and $\psi:G\times X\rightarrow X\times_Y X, (g,x)\mapsto (g\cdot x,x)$ is an isomorphism.
\end{lemma}

\begin{proposition}If $X\rightarrow P$ is a principal $G$-bundle with $P$ quasi-projective, then there exists a line bundle $L$ on $X$ together with a $G$-linearization $\Phi$ of $L$ such that $X^s(\Phi)=X$. In particular $P\simeq X/\!/_\Phi G$.\end{proposition} 

\begin{proof} Follows from Mumford \cite[\S 4, Converse 1.12]{mumford} and the fact that the morphism $X\rightarrow P$ is an affine morphism (as $G$ is affine).\end{proof}

We say that the action of $G$ on $X$ is \emph{free} if $\psi:G\times X\rightarrow X\times X$ is a closed immersion.  Recall that a geometric quotient $X\rightarrow X/G$ by a free action of $G$ on $X$ is a principal $G$-bundle \cite[Proposition 0.9]{mumford}. In the case where $X$ is affine then the quotient map $X\rightarrow X/\!/G$ is a principal $G$-bundle if and only if the stabilizers $C_G(x)$, with $x\in X$, are all trivial and the $G$-orbits of $X$ are all separable (see Bardsley and Richardson \cite[Proposition 8.2]{BR}).

We have the following proposition (see Mumford \cite[Proposition 7.1]{mumford}).

\begin{proposition} Let $G$ act on the algebraic varieties $X$ and $Y$ and let $f:X\rightarrow Y$ be a $G$-equivariant morphism. Assume that $Y\rightarrow Z$ is a principal $G$-bundle with $Z$ quasi-projective. Assume also that there exists a $G$-equivariant line bundle $L$ over $X$ which is relatively ample for $f$. Then there exists a quasi-projective variety $P$ and a principal $G$-bundle  $X\rightarrow P$.  Moreover the commutative diagram

$$\xymatrix{X\ar[rr]^f\ar[d]&&Y\ar[d]\\P\ar[rr]^{f/G}&&Z}$$is Cartesian.

If $\K=\overline{\F}_q$ and if all our data are defined over $\F_q$ then $P$, $X\rightarrow P$ and $X\simeq P\times_Z Y$ are also defined over $\F_q$. \end{proposition}

Assume that $A$ is a finitely generated $\K$-algebra. The projective $r$-space over $A$ is the algebraic variety $\bP^r_A:=\Proj A[x_0,\dots,x_r]= {\rm Spec}\, A\times\bP^r_\K$. We denote by $\calO_A(1)$ the twisting sheaf on $\bP^r_A$. 

We now assume that $G$ acts on the algebraic varieties ${\rm Spec}\, A$ and $\bP^r_\K$ and so on $\bP^r_A$. The ample line bundle $\calO_A(n)$ admits a $G$-linearization for some $n$ sufficiently large (as the twisting sheaf $\calO(1)$ on $\bP^r_\K$ does by Dolgachev \cite[Corollary 7.2]{Dolg}). For such an $n$, the restriction $L$ of $\calO_A(n)$ to a closed $G$-stable subvariety $X$ of $\bP_A^r$ admits then a $G$-linearization $\Phi$. In this case, the $X_s$, with $s\in H^0(X,L(n))^G$, are always affine.

\begin{corollary}Let $f:X\rightarrow Y$ be a projective $G$-equivariant morphism with $Y$ affine. Assume moreover that $C_G(y)=1$ for all $y\in Y$ and that the $G$-orbits of $Y$ are all separable. Then the geometric quotients $Y\rightarrow Y/G$ and $X\rightarrow X/G$ exists (and are principal $G$-bundles) and $X\simeq X/G\times_{Y/G}Y$. If $\K=\overline{\F}_q$ and if $X$,$Y$, $G$ and $f$ are defined over $\F_q$, then $Y\rightarrow Y/G$, $X\rightarrow X/G$ and $X\simeq X/G\times_{Y/G}Y$ are also defined over $\F_q$. \label{goodquotient1}
\end{corollary}

\subsection{Particular case: Affine varieties}

Assume now that $X$ is an affine algebraic variety. Let $\chi:G\rightarrow \K^{\times}$ be a linear character of $G$. Then the action of $G$ on $L^o=X\times\A^1$ given by $g\cdot(x,t)\mapsto (g\cdot x,\chi(g)^{-1}t)$ defines a $G$-linearization $\Phi$ of $L^o$. The space $H^o(X,L^o(n))^G$  with $n\geq 0$ can be then identified with the space $\K[X]^{G,\chi^n}$ of functions $f\in\K[X]$ which satisfy $f(g\cdot x)=\chi^n(g)f(x)$ for all $g\in G$ and $x\in X$. Such a function $f\in\K[X]$ is called a $\chi^n$-\emph{semi-invariant function}.

A polynomial $f=\sum_{i=0}^rf_i\cdot z^i\in\K[X][z]\simeq\K[X\times\A^1]$ is $G$-invariant if and only if for each $i$, the function $f_i$ is a $\chi^i$-semi-invariant, that is $$\K[X\times\A^1]^G=\bigoplus_{n\geq 0}\K[X]^{G,\chi^n}$$and so $$X/\!/_\Phi G=\Proj\left(\K[X\times\A^1]^G\right).$$The canonical projective morphism 
\begin{equation}\pi_X:X/\!/_\Phi G\rightarrow X/\!/G:={\rm Spec}\,\left(\K[X]^G\right).\label{pro}\end{equation}is induced by the inclusion of algebras $\K[X]^G\subset \K[X\times\A^1]^G$. Of course if $\Phi$ is trivial then $\pi_X$ is an isomorphism.

We will use the following important property. Let $q:X^{ss}(\Phi)\rightarrow X/\!/_\Phi G$ be the quotient with respect to $(L^o,\Phi)$.

\begin{proposition} If $F$ is closed subvariety of $X$, then $F^{ss}(\Phi)=X^{ss}(\Phi)\cap F$ and the canonical morphism $F/\!/_\Phi G\rightarrow q(F^{ss}(\Phi))$ is bijective. If $\K=\C$, it is an isomorphism.
\label{affinegoodquotient}\end{proposition}

\begin{remark} Note that for any $G$-equivariant morphism $\phi:X\rightarrow Y$ of affine algebraic varieties, then the co-morphism $\phi^\sharp:\K[Y]\rightarrow\K[X]$ preserves $\chi$-semi-invariants, hence we always have $\phi^{-1}(Y^{ss}(\Phi))\subset X^{ss}(\Phi)$. If moreover $\phi$ is a finite morphism then $\phi\left(X^{ss}(\Phi)\right)\subset Y^{ss}(\Phi)$  and so we recover the first assertion of the proposition. 
\end{remark}

\section{Intersection cohomology}\label{IC}

\subsection{Generalities and notation}Let $X$ be an algebraic variety over the algebraically closed field $\K$. Let $\ell$ be a prime which does not divide the characteristic of $\K$. The letter $\kappa$ denotes the field $\overline{\Q}_\ell$.

We denote by $\mathcal{D}_c^b(X)$ the bounded ``derived category'' of $\kappa$-(constructible) sheaves on $X$. For $K\in\mathcal{D}_c^b(X)$
 we denote by $\mathcal{H}^iK$ the $i$-th cohomology sheaf of $K$. If $m$ is an integer, then we denote by $K[m]$ the $m$-th shift of $K$\hspace{.05cm};\hspace{.05cm} we have $\mathcal{H}^iK[m]=\mathcal{H}^{i+m}K$.
For a morphism $f:X\rightarrow Y$, we have the usual functors $f_*, f_!:\mathcal{D}_c^b(X)\rightarrow\mathcal{D}_c^b(Y)$ and $f^*, f^!:\mathcal{D}_c^b(Y)\rightarrow\mathcal{D}_c^b(X)$. If $i:Y\hookrightarrow X$ is a closed immersion, the restriction $i^*K$ of $K\in\mathcal{D}_c^b(X)$ is denoted by $K|_Y$. We denote by $D_X:\mathcal{D}_c^b(X)\rightarrow\mathcal{D}_c^b(X)$ the Verdier dual operator.

Recall (see Beilinson-Bernstein-Deligne \cite{BBD}) that a perverse sheaf on $X$ is an object $K$ in $\mathcal{D}_c^b(X)$ which satisfies the following two conditions:
\vspace{.2cm}

${\rm dim}\hspace{.05cm}\big({\rm Supp}(\mathcal{H}^iK)\big)\leq-i,$

${\rm dim}\hspace{.05cm}\big({\rm Supp}(\mathcal{H}^iD_XK)\big)\leq-i\,\text{ for all }\,i\in\Z$.

\vspace{.2cm}

The full subcategory of $\mathcal{D}_c^b(X)$ of 
perverse sheaves on $X$ forms an abelian category (see BBD \cite[Th\'eor\`eme 1.3.6]{BBD}) and its objects are all of finite length (see BBD \cite[Th\'eor\`eme 4.3.1 (i)]{BBD}). 

 Let now $Y$ be an irreducible open nonsingular subset of $X$ such that $\overline{Y}=X$. Then for a local system $\xi$ on $Y$, we let $\mathcal{IC}_{X,\xi}^{\bullet}\in\mathcal{D}_c^b(X)$ be the intersection cohomology complex defined by Goresky-McPherson and Deligne. The perverse sheaf $K=\underline{\mathcal{IC}}_{X,\xi}^{\bullet}:=\mathcal{IC}_{X,\xi}^{\bullet}[{\rm dim}\hspace{.05cm}X]$ is characterized by the following properties:
\vspace{.2cm}

$\mathcal{H}^iK=0\,\text{ if }\,i<-{\rm dim}\hspace{.05cm}X,$

$\mathcal{H}^{-{\rm dim}\hspace{.05cm}X}K|_Y=\xi,$

${\rm dim}\hspace{.05cm}\big({\rm Supp}(\mathcal{H}^iK)\big)<-i\,\text{ if }\,i>-{\rm dim}\hspace{.05cm}X,$

${\rm dim}\hspace{.05cm}\big({\rm Supp}(\mathcal{H}^iD_XK)\big)<-i\,\text{ if }\,i>-{\rm dim}\hspace{.05cm}X.$
\vspace{.2cm}

\noindent If $U$ is another open nonsingular subset of $X$ and if $\zeta$ is any local system on $U$ such that $\zeta|_{U\cap Y}=\xi|_{U\cap Y}$, then $\IC {X,\xi}=\IC {X,\zeta}$. This is why we omitt the open set $Y$ from the notation $\IC {X,\xi}$. We will simply denote by $\IC X$ the complex $\IC {X,\overline{\Q}_{\ell}}$.

\begin{remark}Note that if $U$ is a locally closed subvariety of $X$ such that $\overline{U}\subsetneq X$ then $\mathcal{H}^{-{\rm dim}\,U}K|_U=0$.
\label{remlocal}\end{remark}

We have the following description of simple perverse sheaves due to Beilinson, Bernstein and Deligne. If $Z$ is an irreducible closed subvariety of $X$ and $\xi$ an irreducible local system on some open subset of 
$Z$ then the extension by zero of $\pIC {Z,\xi}$ on $X-Z$ is a simple perverse sheaf on $X$ and any simple perverse sheaf on $X$ arises in this way from some pair $(Z,\xi)$ (see BBD \cite[4.3.1]{BBD}). 

It will be convenient to continue to denote by $\IC {Z,\xi}$ and $\pIC {Z,\xi}$ their extension by zero on $X-Z$.

Note that if $X$ is nonsingular then $\pIC {X,\xi}=:\underline{\xi}$ is the complex $K^\bullet$ concentrated in degree $-{\rm dim}\,X$ with $K^{-{\rm dim}\, X}=\xi$.

We define the compactly supported $i$-th intersection cohomology groups $IH_c^i(X,\xi)$ with coefficient in the local system $\xi$ as the compactly supported $i$-th $\ell$-adic hypercohomology group $\H_c^i\big(X,\IC {X,\xi}\big)$.  If $f$ is the unique morphism $X\rightarrow \{pt\}$, then $IH^i_c(X,\xi)=\mathcal{H}^i\big(f_!\IC {X,\xi}\big)$.

If $X$ is nonsingular, then $\IC X$ is the constant sheaf $\kappa$ concentrated in degree $0$ and so $IH^i_c(X,\kappa)=H_c^i(X,\kappa)$.
\vspace{.2cm}

We will need the following decomposition theorem of Beilinson, Bernstein, Deligne and Gabber.

\begin{theorem}\label{resolution} Suppose that  $\varphi:X\rightarrow X'$ is a proper map with $X$ irreducible. Then $$\varphi_*(\IC X)\simeq \bigoplus_{Z,\xi,r}V_{Z,\xi,r}\otimes\IC {Z,\xi}[r]$$where $\xi$ is an irreducible local system on some open subset of a closed irreducible subvariety $Z$ of $X'$. 
If moreover $\varphi_*(\pIC X)$ is a perverse sheaf, then \begin{equation}\varphi_*(\pIC X)\simeq \bigoplus_{Z,\xi}V_{Z,\xi}\otimes\pIC {Z,\xi}\label{resol}\end{equation}
\label{BBDG}\end{theorem}

The theorem remains true if we replace $\IC X$ by a semisimple object of ``geometrical origin'' \cite[6.2.4]{BBD}.

\begin{remark} Let $Y$ be a closed irreducible subvariety of $X'$ and let $U$ be a non-empty nonsingular open subset of $Y$. Note that

$$
\mathcal{H}^{-{\rm dim}\,Y}\left(\left.\bigoplus_{Z,\xi}V_{Z,\xi}\otimes\pIC {Z,\xi}\right)\right|_U\simeq\bigoplus_\xi V_{Y,\xi}\otimes\xi
 $$
where the direct sum on the right hand side is over the irreducible local systems on $Y$.

\label{finrem}\end{remark}

\begin{definition} A proper surjective morphism $f:Z\rightarrow X$ is \emph{semi-small} if and only if one of the following equivalent conditions is satisfied:

(i) ${\rm dim}\hspace{.05cm}\{x\in X\,|\,{\rm dim}\,f^{-1}(x)\geq i\}\leq {\rm dim}\,X-2i$ for all $i\in\N$.

(ii) There exists a filtration $X:=F_0\supset F_1\supset\cdots\supset F_r=\emptyset$ of $X$ by closed subsets such that, for all $i\in\{0,\dots,r-1\}$ and $x\in F_i-F_{i+1}$, we have $2\,{\rm dim}\,f^{-1}(x)\leq {\rm dim}\,X-{\rm dim}\,F_i$.
 
\label{semi-small}\end{definition}

We will use the following easy fact.

\begin{lemma} Let $f:Z\rightarrow X$ be a proper surjective map and let $X:=F_0\supset F_1\supset\cdots\supset F_r=\emptyset$ be a filtration of $X$ by closed subsets. Let $h:X'\rightarrow X$ be a surjective map and put   $F'_i=h^{-1}(F_i)$. Assume that ${\rm dim}\,X-{\rm dim}\,F_i={\rm dim}\,X'-{\rm dim}\,F_i'$. Then the projection on the second coordinate $Z\times_XX'\rightarrow X'$ is semi-small with respect to the filtration $X':=F_0'\supset F_1'\supset\cdots\supset F_r'=\emptyset$ if and only if the map $f$ is semi-small with respect to $X:=F_0\supset F_1\supset\cdots\supset F_r=\emptyset$.
\label{easyfact}\end{lemma}

\begin{definition} Let $X$ be an algebraic variety over $\K$. We say that $X=\coprod_{\alpha\in I}X_\alpha$ is a \emph{stratification} of $X$ if the set $\{\alpha\in I\,|\, X_\alpha\neq\emptyset\}$ is finite, for each $\alpha\in I$ such that $X_\alpha\neq\emptyset$, the subset $X_\alpha$ is a locally closed nonsingular equidimensional subvariety of $X$, and for each $\alpha,\beta\in I$, if  $X_\alpha\cap\overline{X}_\beta\neq\emptyset$, then $X_\alpha\subset \overline{X}_\beta$.
 \label{stratification}\end{definition}

It is well-known that if $f:Z\rightarrow X$ is a semi-small map with $Z$ nonsingular and irreducible, then the complex $f_*(\underline{\xi})$ is a perverse sheaf for any local system $\xi$ on $Z$.

We can actually generalize this result as follows.

\begin{proposition} Let $f:Z\rightarrow X$ be a proper surjective map with $Z$ irreducible and let $Z=\coprod_{\alpha\in I} Z_\alpha$ be a stratification of $Z$. For $x\in X$, put $f^{-1}(x)_\alpha:=f^{-1}(x)\cap Z_\alpha$. Assume that 

$${\rm dim}\left\{x\in X\,\left|\, {\rm dim}\, f^{-1}(x)_\alpha\geq \frac{i}{2}-\frac{1}{2}\,{\rm codim}_Z(Z_\alpha)\right.\right\}\leq{\rm dim}\, X-i$$holds for all $\alpha\in I$ and all $i\in\N$ where ${\rm codim}_Z(Z_\alpha):={\rm dim}\, Z-{\rm dim}\, Z_\alpha$. Then for any perverse sheaf $K$ on $Z$ with respect to the above stratification, the complex $f_*K$ is a perverse sheaf on $X$. 
\label{Lu1}\end{proposition}

This proposition is used and proved (without being stated explicitly) in Lusztig's generalisation of Springer correspondence \cite[proof of Proposition 4.5]{LuIC}. 

\begin{proof}  We need to prove that

\vspace{.2cm}

(i) ${\rm dim}\hspace{.05cm}\big({\rm Supp}(\mathcal{H}^if_*K)\big)\leq-i,$

(ii) ${\rm dim}\hspace{.05cm}\big({\rm Supp}(\mathcal{H}^iD_Xf_*K)\big)\leq-i\,\text{ for all }\,i\in\Z$.

\vspace{.2cm}

Since $f$ is proper, the Verdier dual commutes with $f_*$ and so  we only prove (i) as the proof of (ii) will be similar. The stalk $\mathcal{H}^i_xf_*K$ is the hypercohomology $\H^i\left(f^{-}(x),K|_{f^{-1}(x)}\right)$. If for $x\in X$ we have $\H^i\left(f^{-}(x),K|_{f^{-1}(x)}\right)\neq 0$ this means that there exists $\alpha\in I$ such that the compactly supported hypercohomology $\H_c^i\left(f^{-}(x)_\alpha,K|_{f^{-1}(x)_\alpha}\right)$ does not vanish. Hence to prove (i) we are reduced to see that for all $\alpha\in I$ and $i$, 

\begin{equation}{\rm dim}\, \left\{x\in X\,\left|\, \H_c^i\left(f^{-1}(x)_\alpha,K|_{f^{-1}(x)_\alpha}\right)\neq 0\right.\right\}\leq -i.\label{L1}\end{equation}

If $\H_c^i\left(f^{-1}(x)_\alpha,K|_{f^{-1}(x)_\alpha}\right)\neq 0$ then from the hypercohomology spectral sequence we may write $i$ as $i_1+i_2$ with $i_1\leq 2\, {\rm dim}\, f^{-1}(x)_\alpha$ and $\mathcal{H}^{i_2}\left(K|_{f^{-1}(x)_\alpha}\right)\neq 0$. The last condition implies that $\mathcal{H}^{i_2}K|_{Z_\alpha}\neq 0$. Since $K$ is a perverse sheaf, we must have $i_2+{\rm dim}\, X\leq{\rm codim}_Z(Z_\alpha)$.
We thus have $i+{\rm dim}\, X\leq 2\,{\rm dim}\,f^{-1}(x)_\alpha+{\rm codim}_Z(Z_\alpha)$. Hence the inequality (\ref{L1}) is a consequence of the following one 

$${\rm dim}\, \left\{x\in X\,\left|\, i+{\rm dim}\, X\leq 2\,{\rm dim}\,f^{-1}(x)_\alpha+{\rm codim}_Z(Z_\alpha) \right.\right\}\leq -i.$$Hence we are reduced to see that 

$${\rm dim}\, \left\{x\in X\,\left|\, {\rm dim}\,f^{-1}(x)_\alpha\geq \frac{1}{2}\left(i-{\rm codim}_Z(Z_\alpha)\right) \right.\right\}\leq {\rm dim}\, X-i$$for all $i$.
\end{proof}

\begin{corollary} Let $\varphi:X\rightarrow X'$ be a birational morphism which satisfies the condition in Proposition \ref{Lu1}, then  (\ref{resol}) becomes \begin{equation}\varphi_*(\pIC X)\simeq \pIC {X'}\oplus\left(\bigoplus_{Z,\xi}V_ {Z,\xi}\otimes\pIC {Z,\xi}\right)\label{resol3}\end{equation}with $Z\subsetneq X'$. In particular  \begin{equation}IH_c^i\big(X,\kappa\big)\simeq IH_c^i(X',\kappa)\oplus\left(\bigoplus_{Z,\xi}V_{Z,\xi}\otimes IH_c^{i+d_Z-d_X}(Z,\xi)\right).\label{resol2}\end{equation}where $d_Z$ the dimension of $Z$.
\label{coro3.1.8}\end{corollary}

The isomorphism (\ref{resol2}) is obtained from (\ref{resol3}) by applying the functor $f_!$ with $f:X'\rightarrow\{pt\}$.

\begin{corollary} Assume that $\varphi:X\rightarrow X'$ satisfies the condition in Proposition \ref{coro3.1.8}. If $X'=\bigcup_{\alpha\in I} X'_\alpha$ where $I$ is a finite set and where the $X'_\alpha$ are locally closed irreducible subvarieties of $X'$ such that the restriction of $\mathcal{H}^i\left(\varphi_*\left(\pIC X\right)\right)$ to $X_\alpha'$ is a locally constant sheaf for all $i$ and all $\alpha\in I$, then $$\varphi_*\left(\pIC X\right)\simeq \pIC {X'}\oplus\left(\bigoplus_{\alpha,\xi_\alpha} V_{\alpha,\xi_\alpha} \otimes\pIC {\overline{X}{'_\alpha},\xi_\alpha}\right)$$where the direct sum is over the $\alpha$ such that $\overline{X}{'_\alpha}\subsetneq Y$.
\label{remdecomp}\end{corollary}

\begin{proof} Let $Z$ be an irreducible closed subvariety of $X'$ such that $\pIC {Z,\xi}$ is a direct summand of $\varphi_*\left(\pIC X\right)$. We have $Z=\bigcup_\alpha (X_\alpha'\cap Z)$. Since $Z$ is irreducible, there exists an $\alpha$ such that $X_\alpha'\cap Z$ is dense in $Z$. We have $\mathcal{H}^{-{\rm dim}\, Z}\left.\varphi_*\left(\pIC X\right)\right|_{X_\alpha'\cap Z}\neq 0$. Since $\mathcal{H}^{-{\rm dim}\, Z}\left.\varphi_*\left(\pIC X\right)\right|_{X_\alpha'}$ is locally constant and non-zero, we have $X_\alpha'\subset {\rm Supp}\,\left(\mathcal{H}^{-{\rm dim}\, Z}\varphi_*\left(\pIC X\right)\right)$. Hence

$${\rm dim}\, X_\alpha'\leq {\rm dim}\,\left({\rm Supp}\,\left(\mathcal{H}^{-{\rm dim}\, Z}\varphi_*\left(\pIC X\right)\right)\right)\leq {\rm dim}\, Z.$$The right inequality holds because $\varphi_*\left(\pIC X\right)$ is a perverse sheaf. Since ${\rm dim}\, (X_\alpha'\cap Z)={\rm dim}\, Z$, we deduce that the inclusion $X_\alpha'\cap Z\subset X_\alpha'$ is an equality, i.e., $X_\alpha'\subset Z$, and so that $\overline{X}{_\alpha'}=Z$.
\end{proof}

Assume that $\K$ is an algebraic closure of a finite field $\F_q$ and that $X$ is an irreducible algebraic variety defined over $\F_q$. We denote by $F:X\rightarrow X$ the corresponding Frobenius endomorphism. We will use either the notation $X^F$ or $X(\F_q)$ to denote the fixed points of $X$ by $F$. Let $K\in\mathcal{D}_c^b(X)$ and assume that there exists an isomorphism $\varphi: F^*(K)\simeq K$. The \emph{characteristic function} $\mathbf{X}_{K,\varphi}:X^F\rightarrow\kappa$ of $(K,\varphi)$ is defined by $$\mathbf{X}_{K,\varphi}(x)=\sum_i(-1)^i{\rm Trace}\hspace{.05cm}\big(\varphi_x^i,\mathcal{H}^i_xK\big).$$If $r\in\Z$, we denote by $K(r)$ the $r$-th Tate twist of $K$. Then $\mathbf{X}_{K(r),\,\varphi(r)}=q^{-r}\,\mathbf{X}_{K,\,\varphi}$. 

Let $Y$ be an open nonsingular $F$-stable subset of $X$. We will simply denote by ${\bf X}_{\IC X}$ the function ${\bf X}_{\IC X,\varphi}$ where $\varphi:F^*\left(\IC X\right)\rightarrow \IC X$ is the unique isomorphism which induces the identity on $\mathcal{H}^0_x\left(\IC X\right)$ for all $x\in Y^F$.

\subsection{Restriction}

Assume that $X$ is irreducible. Let $Z$ be an irreducible closed subvariety of $X$ et let $i:Z\hookrightarrow X$ denotes the inclusion. We give a condition for $i^*(\IC X)=\IC Z$ to be true.

\begin{proposition} Assume that there is a decomposition $X=\bigcup_{\alpha\in I} X_\alpha$ of $X$ where $I$ is a finite set and where the $X_\alpha$ are locally closed irreducible subvarieties such that

(i) if $Z_\alpha:=X_\alpha\cap Z$ is not empty, then it is equidimensional and  ${\rm codim}_X\,X_\alpha={\rm codim}_Z\, Z_\alpha$.

\noindent Assume moreover that there exists a Cartesian diagram 

$$\xymatrix{\tilde{X}\ar^f[rr]&&X\\
\tilde{Z}\ar[rr]^g\ar^{\tilde{i}}[u]&&Z\ar^i[u]}$$such that the conditions (ii) and (iii) below are satisfied.

(ii) $f$ and $g$ are semi-small resolutions of singularities.

(iii) The restriction of the sheaf $\mathcal{H}^i(f_*(\kappa))$ to $X_\alpha$ is a locally constant sheaf for all $i$.

\noindent Then $i^*(\IC X)=\IC Z$.
\label{proprest}\end{proposition}

\begin{proof} If $Y$ is a variety, let $d_Y$ denote its dimension. Let $\alpha_o\in I$ be such that $X_{\alpha_o}$ is the open stratum of $X$. To avoid any confusion we will use the notation $\IC Z[d_Z]$ instead of $\pIC Z$. By Corollary \ref{remdecomp}, we have 

\begin{equation}f_*(\kappa[d_X])=\IC X[d_X]\oplus\left(\bigoplus_{\alpha\neq\alpha_o,\xi_\alpha}V_{\alpha,\xi_\alpha}\otimes\IC {\overline{X}_\alpha,\xi_\alpha}[d_{X_\alpha}]\right).\label{decompX}\end{equation}By (iii) and  $i^*f_*(\kappa)= g_*(\kappa)$ we see that the restriction  of $\mathcal{H}^i(g_*(\kappa))$ to $Z_\alpha$ is locally constant. Hence by Corollary \ref{remdecomp}, we have \begin{equation}g_*(\kappa[d_Z])=\IC Z[d_Z]\oplus\left(\bigoplus_{\alpha\neq\alpha_o,\beta\in I_\alpha, \zeta_{\alpha,\beta}}W_{(\alpha,\beta),\zeta_{\alpha,\beta}}\otimes\IC {\overline{Z}_{(\alpha,\beta),\zeta_{\alpha,\beta}}}[d_{Z_\alpha}]\right)\label{decompZ}\end{equation}where $\{Z_{(\alpha,\beta)}\}_{\beta\in I_\alpha}$ is the set of irreducible components of $Z_\alpha$. Using again $i^*f_*(\kappa)=g_*(\kappa)$ we see from (\ref{decompX}) and (\ref{decompZ}) that the complex $i^*(\IC X)[d_Z]$ is a direct summand of the semisimple perverse sheaf $g_*(\kappa[d_Z])$. It is therefore a semisimple perverse subsheaf of $g_*(\kappa[d_Z])$. Since the open stratum $Z_{\alpha_o}$ of $Z$ is contained in the open stratum of $X$, the restriction of $i^*(\IC X)[d_Z]$ to $Z_{\alpha_o}$ is the constant sheaf $\kappa[d_Z]$. Hence $i^*(\IC X)[d_Z]$ contains $\IC Z[d_Z]$ as a direct summand, i.e., 

$$i^*(\IC X)[d_Z]=\IC Z[d_Z]\oplus\left(\bigoplus_{\alpha\neq\alpha_o,\beta\in I_\alpha, \zeta_{\alpha,\beta}}W_{(\alpha,\beta),\zeta_{\alpha,\beta}}'\otimes\IC {\overline{Z}_{(\alpha,\beta),\zeta_{\alpha,\beta}}}[d_{Z_\alpha}]\right)$$for some subspaces $W_{(\alpha,\beta),\zeta_{\alpha,\beta}}'\subset W_{(\alpha,\beta),\zeta_{\alpha,\beta}}$. It remains to see that $W_{(\alpha,\beta),\zeta_{\alpha,\beta}}'=0$ for all $\alpha\neq\alpha_o$.
 
Put $K:=i^*(\IC X)[d_Z]$. Then for $\alpha\neq\alpha_o$ we have

\begin{align*} \mathcal{H}^{-d_{Z_\alpha}}K|_{Z_\alpha}&=\mathcal{H}^{d_Z-d_{Z_\alpha}}\IC X|_{Z_\alpha}\\
   &=\mathcal{H}^{d_X-d_{X_\alpha}}\IC X|_{Z_\alpha}\\
&=\mathcal{H}^{-d_{X_\alpha}}\IC X[d_X]|_{Z_\alpha}\\
&=0.\end{align*}The last equality follows from Remark \ref{remlocal}. Hence $W_{(\alpha,\beta),\zeta_{\alpha,\beta}}'=0$ by Remark \ref{finrem} and we proved the proposition.\end{proof}

\subsection{E-polynomial}\label{epolynome}

Recall that a mixed Hodge structure on a rational vector space $H$ consist of a finite increasing filtration $W_{\bullet}$ (the weight filtration) on $H$, and a finite decreasing filtration $F^{\bullet}$ (the Hodge filtration) on the complexification $H_\C$, which induces a
pure Hodge structure of weight $k$ on the complexified graded pieces ${\rm Gr}_k^WH_\C=(W_kH/W_{k-1}H)_\C$, i.e., 

$${\rm Gr}_k^WH_\C=\bigoplus_{p+q=k}\left({\rm Gr}_k^WH_\C\right)^{p,q}$$with 

$$\left({\rm Gr}_k^WH_\C\right)^{p,q}=F^p{\rm Gr}_k^WH_\C\cap \overline{F^q{\rm Gr}_k^WH_\C}.$$We call the integers $\left\{h^{p,q}:={\rm dim}\, \left({\rm Gr}_{p+q}^WH_\C\right)^{p,q}\right\}_{p,q}$ the \emph{mixed Hodge numbers}. 

Recall (Saito \cite{saito}, see also \cite[Chapter 14]{Peters-etal}) that for any complex algebraic variety $X$, the intersection cohomology group $IH_c^k(X,\C)$ is endowed with a mixed Hodge structure. If $X$ is nonsingular, it coincides with Deligne's mixed Hodge structure on $H_c(X,\C)$ which is defined in \cite{Del1}. 

We then denote by $\{ih_c^{p,q;k}(X)\}_{p,q}$ the mixed Hodge numbers of $IH_c^k(X,\C)$ and we define the mixed Hodge polynomial  of $X$ as $$IH_c(X;x,y,z)=\sum_{p,q,k}ih_c^{p,q;k}(X)x^py^qz^k.$$The compactly supported Poincar\'e polynomial of $X$ is then $IH_c(X;1,1,t)$. 

In this paper we will say that $X$ is \emph{pure} if the mixed Hodge structure on $IH_c^k(X,\C)$ is pure for all $k$, i.e., if $ih_c^{p,q;k}(X)=0$ when $p+q\neq k$.

The $E$-polynomial of $X$ is defined as $$E^{ic}(X;x,y):=IH_c(X;x,y,-1)=\sum_{p,q}\left(\sum_k(-1)^kih_c^{p,q;k}(X)\right)x^py^q.$$

Let $R$ be a subring of $\C$ which is finitely generated as a $\Z$-algebra and let $\calX$ be a separated $R$-scheme of finite type. According to \cite[Appendix]{hausel-villegas}, we say that $\calX$ is \emph{strongly polynomial count} if there exists a polynomial $P(T)\in\C[T]$ such that for any finite field $\F_q$ and any ring homomorphism $\varphi: R\rightarrow\F_q$, the $\F_q$-scheme $\calX^{\varphi}$ obtained from $\calX$ by base change is polynomial count with counting polynomial $P$, i.e., for every finite extension $\F_{q^n}/\F_q$, we have $$\sharp\{\calX^{\varphi}(\F_{q^n})\}=P(q^n).$$

According to Katz terminology (cf. [Appendix]\cite{hausel-villegas}), we call a separated $R$-scheme $\calX$ which gives back $X$ after extension of scalars from $R$ to $\C$ a \emph{spreading out} of $X$.

The complex variety $X$ is said to be \emph{polynomial count} if there exists a spreading out of $X$ which is strongly polynomial count.

Let us now denote by $\{h_c^{i,j;k}(X)\}_{i,j}$ the mixed Hodge numbers of $H_c^k(X,\C)$ and put $$E(X;x,y):=\sum_{i,j}\left(\sum_k(-1)^kh_c^{i,j;k}(X)\right)x^iy^j.$$
We recall the result of Katz in the appendix of \cite{hausel-villegas} (see also Kisin and Lehrer \cite{LK} for closely related results).

\begin{theorem}Assume that $X$ is  polynomial count with counting polynomial $P\in\C[T]$. Then 
$$E(X;x,y)=P(xy).$$
\label{Katz}\end{theorem}

Let $X=\coprod_{\alpha\in I} X_\alpha$ be a stratification and let $X_{\alpha_o}$ be the open stratum, i.e., $X=\overline{X}_{\alpha_o}$. Put $\alpha\leq\beta$ if $X_\alpha\subset\overline{X}_\beta$, and  $r_\alpha:=({\rm dim}\, X_\alpha-{\rm dim}\, X)/2$. 

We say that $X$ satisfies the property $(E)$ with respect to this stratification and the ring $R$ if there exists a spreading out $\calX$ of $X$, a stratification $\calX=\coprod_\alpha\calX_\alpha$,  and a morphism  $\mathcal{r}:\tilde{\calX}\rightarrow\calX$ of $R$-schemes such that:

\noindent (1) $\tilde{\calX}$ and the closed strata $\calX_\alpha$ are strongly polynomial count,

\noindent (2) for each $\alpha$, the stratum $\calX_\alpha$ is a spreading out of $X_\alpha$, the morphism $r:\tilde{X}\rightarrow X$ obtained from $\mathcal{r}$ after extension of scalars from $R$ to $\C$ yields an isomorphism of mixed Hodge structures

\begin{equation}H_c^i(\tilde{X},\Q)\simeq IH_c^i(X,\Q)\oplus\left(\bigoplus_{\alpha\neq\alpha_o} W_\alpha\otimes\left(IH_c^{i+2r_\alpha}(\overline{X}_\alpha,\Q)\otimes\Q(r_\alpha)\right)\right),\label{isomhs}\end{equation}where $\Q(-d)$ is the pure mixed Hodge structure on $\Q$ of weight $2d$ and with Hodge filtration $F^d=\C$ and $F^{d+1}=0$.

\noindent (3) for any ring homomorphism $\varphi: R\rightarrow \F_q$, the morphism $\mathcal{r}^{\varphi}:\tilde{\calX}^{ \varphi}\rightarrow\calX^{\varphi}$ obtained from $\mathcal{r}$ by base change yields an isomorphism   \begin{equation}\left(\mathcal{r}^{\varphi}\right)_*(\underline{\kappa})\simeq \pIC {\calX^\varphi}\oplus\left(\bigoplus_{\alpha\neq\alpha_o} W_\alpha\otimes \pIC {\overline{\calX}_\alpha^\varphi}(r_\alpha)\right)\label{isops}\end{equation}of perverse sheaves.

Assume now that  all complex varieties $\overline{X}_\alpha$ (in particular $X$) satisfy the property $(E)$ with respect to the stratification $\overline{X}_\alpha=\coprod_{\beta\leq\alpha}X_\beta$ and the ring $R_\alpha$. Since there is only a finite number of strata we may assume without loss of generalities that the rings $R_\alpha$ are all equal to the same ring $R$.

\begin{theorem} With the above assumption, there exists a polynomial $P(T)\in\Z[T]$ such that for any ring homomorphism $\varphi: R\rightarrow\F_q$, we have

\begin{equation}\sum_{x\in \calX^\varphi(\F_q)}{\bf X}_{\IC {\calX^\varphi(\overline{\F}_q)}}(x)=P(q)\label{equakatz}\end{equation}and $$E^{ic}(X;x,y)=P(xy).$$
\label{Katz2}\end{theorem}

\begin{proof} If there is only one stratum, i.e., if $X$ is nonsingular, then the theorem is true by Theorem \ref{Katz}. The theorem is now easy to prove by induction on $\alpha<\beta$. Assume that the theorem is true for all $\alpha<\alpha_o$. By Formula (\ref{isomhs}), we have $$E(\tilde{X};x,y)=E^{ic}(X;x,y)+\sum_{\alpha<\alpha_o}\left({\rm dim}\, W_\alpha\right) x^{-r_\alpha}y^{-r_\alpha}E^{ic}\left(\overline{X}_\alpha;x,y\right).$$By induction hypothesis and since $\tilde{X}$ is polynomial count, this formula shows that $E^{ic}(X;x,y)$ depends only on the product $xy$, i.e., that there exists a unique polynomial $P$ such that $E^{ic}(X;x,y)=P(xy)$, more precisely $P$ is defined as $P=\tilde{P}-\sum_{\alpha<\alpha_o}\left({\rm dim}\, W_\alpha\right)x^{-r_\alpha}y^{-r_\alpha} P_\alpha(xy)$ where $\tilde{P}$ is the counting polynomial of $X$ and $P_\alpha$ (with $\alpha\neq\alpha_o$) is the polynomial which satifies the theorem for $X=\overline{X}_\alpha$. It remains to see that $P$ satisfies Formula (\ref{equakatz}).

By Formula (\ref{isops}), we have 

\begin{equation}{\bf X}_{\left(\mathcal{r}^\varphi\right)_*(\kappa)}={\bf X}_{\IC {\calX^\varphi}}+\sum_{\alpha<\alpha_o}\left({\rm dim}\, W_\alpha\right)q^{-r_\alpha}{\bf X}_{\IC {\overline{\calX}{^\varphi_\alpha}}}.\label{equaps1}\end{equation}By Grothendieck trace formula we have $$\sum_{x\in\calX^{\varphi}(\F_q)}{\bf X}_{\left(\mathcal{r}^\varphi\right)_*(\kappa)}(x)=\sharp\{\tilde{\calX}^\varphi(\F_q)\}=\tilde{P}(q).$$Now integrating Formula (\ref{equaps1}) over $\calX^\varphi(\F_q)$ proves Formula (\ref{equakatz}).
\end{proof}

\begin{proposition}Assume that $X$ satifies the assumptions of Theorem \ref{Katz2} and that $X$ is pure. Then for any ring homomorphism $\varphi:R\rightarrow\F_q$ we have
 
$$\sum_{x\in \calX^\varphi(\F_q)}{\bf X}_{\IC {\calX^\varphi(\overline{\F}_q)}}(x)=P_c(X;q)$$where $P_c(X;t):=\sum_i\left({\rm dim}\,IH_c^{2i}(X,\C)\right)t^i$.
\label{rempure}\end{proposition}

\begin{proof} Since $X$ is pure we have $E^{ic}(X;x,y)=\sum_{p,q}(-1)^{p+q}ih_c^{p,q;p+q}(X)x^py^q$. By Theorem \ref{Katz2}, the polynomial $E^{ic}(X;x,y)$  depends only on the product $xy$, hence $ih_c^{p,q;p+q}(X)=0$ if $p\neq q$. The mixed Hodge numbers of $X$ are thus all of the form $ih_c^{p,p;2p}(X)$ and so $E^{ic}\left(X;x,y\right)=P_c(X;xy)$.
\end{proof}

\section{Preliminaries on quiver varieties}\label{quiver}

 We introduce the so-called quiver varieties $\mathfrak{M}_{\xi,\thetahat}(\v)$ and $\mathfrak{M}_{\xi,\thetahat}(\v,\w)$ over $\K$ which were considered by many authors including Kronheimer, Lusztig, Nakajima and Crawley-Boevey. The second one, due to Nakajima and called \emph{framed} quiver varieties, can be realized as the first one by an observation due to Crawley-Boevey \cite[introduction]{crawley-quiver}. For our application we found more convenient to introduce them separatly. Here we recall the basic results we need.

In this section we will only consider quotients of affine varieties by  (finite) direct products of $\GL_n$'s. If $G=\GL_{n_1}\times\cdots\times\GL_{n_r}$ is such a group and if $\chi:G\rightarrow\K^{\times}$, $(g_i)\mapsto\prod_i({\rm det}\,g_i)^{-\theta_i}$ is the character given by $\thetahat\in\Z^{\{1,\dots,r\}}$, then we will use the notation  $X/\!/_\thetahat G$ instead of $X/\!/_\Phi G$ and we will use often $X^{ss}$ instead of $X^{ss}(\Phi)$ when the context is clear.

\subsection{Generalities on quiver varieties}\label{uquiver}

Let $\Gamma$ be a quiver and let $I$ denote the set of its vertices. We assume that $I$ is finite. A
\emph{dimension vector} for $\Gamma$ is a collection of non-negative
integers $\mathbf{v}=\{v_i\}_{i\in I}\in \N^I$ and a representation
of $\Gamma$ of dimension $\v$ over $\K$ is a collection of
$\K$-linear maps $\varphi_{i,j}:\K^{v_i}\rightarrow \K^{v_j}$, for each
arrow $i\rightarrow j$ of $\Gamma$, that we identify with matrices
(using the canonical basis of $\K^r$). We define a morphism between two representations (possibly of different dimension) in the obvious way. A \emph{subrepresentation} of $\varphi$ is a representation $\varphi'$ together with an injective morphism $\varphi'\rightarrow\varphi$. Let $\Omega$ be a set
indexing the edges of $\Gamma$. For $\gamma\in\Omega$, let
$h(\gamma),t(\gamma)\in I$ denote respectively the head and the tail
of $\gamma$. The algebraic group $\GL_\v:=\prod_{i\in I}\GL_{v_i}(\K)$ acts
on the space $${\bf
M}(\Gamma,\v):=\bigoplus_{\gamma\in\Omega}\text{Mat}_{v_{h(\gamma)},v_{t(\gamma)}}(\K)$$
of representations of dimension $\mathbf{v}$ in the obvious way, i.e., for $g=(g_i)_{i\in I}\in \GL_\v$ and $B=(x_\gamma)_{\gamma\in\Omega}$, we have $g\cdot B:=(g_{v_{h(\gamma)}}x_\gamma g_{v_{t(\gamma)}}^{-1})$. As
the diagonal center $Z=\{(\lambda .{\rm Id}_{v_i})_{i\in I}\,|\,\lambda\in\K^\times\}\subset \GL_\v$ acts trivially, the action of $\GL_\v$ induces an
action of
$$\G_\mathbf{v}:=\GL_\v/Z.$$
Clearly two elements of ${\bf M}(\Gamma,\mathbf{v})$ are
isomorphic if and only if they are $\G_\mathbf{v}$-conjugate.

We define a bilinear form on $\K^I$ by
$\mathbf{a}\centerdot\mathbf{b}=\sum_ia_ib_i$. Let $\thetahat\in\Z^I$ be such that $\thetahat\centerdot\mathbf{v}=0$. This defines a character $\chi:\G_\mathbf{v}\rightarrow\K^{\times}$ given by $(g_i)_i\mapsto \prod_i{\rm det}\,(g_i)^{-\theta_i}$. 

\begin{theorem}\cite{king} A point $B\in{\bf M}\big(\Gamma,\mathbf{v}\big)$ is $\chi$-semistable if and only if  $$\thetahat\centerdot{\rm dim}\,B'\leq 0$$for every subrepresentation $B'$ of $B$. It is $\chi$-stable if and only if it is semistable and the inequality is strict unless $B'=0$ or $B'\simeq B$.
\label{king}\end{theorem}

We will use the terminology  ``$\thetahat$-semistable'' instead of ``$\chi$-semistable''. We denote respectively by ${\bf M}^{\rm ss}_{\thetahat}\big(\Gamma,\mathbf{v}\big)$ and ${\bf M}^{\rm s}_{\thetahat}\big(\Gamma,\mathbf{v}\big)$ the $\thetahat$-semistable and $\thetahat$-stable representations.

Let $\overline{\Gamma}$ be the \emph{double quiver} of $\Gamma$ i.e.
$\overline{\Gamma}$ has the same vertices as $\Gamma$ but the edges
are given by
$\overline{\Omega}:=\{\gamma,\gamma^*|\hspace{.05cm}\gamma\in\Omega\}$
where $h(\gamma^*)=t(\gamma)$ and $t(\gamma^*)=h(\gamma)$. Then via
the trace pairing we may identify  ${\bf
M}\big(\overline{\Gamma},\mathbf{v}\big)$ with the cotangent
bundle ${\rm T^*}{\bf M}(\Gamma,\mathbf{v})$. Put $\gl_\v={\rm Lie}\,(\GL_\v)=\bigoplus_i\gl_{v_i}(\K)$ and $\g_\v:={\rm Lie}\,(\G_\v)$. Define the
\emph{moment map} \begin{equation} \mu_{\mathbf{v}}:{\bf M}\big(\overline{\Gamma},\mathbf{v}\big)\rightarrow
{\rm M}(\mathbf{v})^0\end{equation}$$
(x_{\gamma})_{\gamma\in\overline{\Omega}}\mapsto\sum_{\gamma\in\Omega}[x_{\gamma},x_{\gamma^*}],
$$ where $${\rm M}(\mathbf{v})^0:=\left\{\left(f_i\right)_{i\in
I}\in\gl_\v\,\left|\,\sum_{i\in I}{\rm
Tr}\right.\,(f_i)=0\right\}.$$Note that we can identify ${\rm M}(\mathbf{v})^0$ with $(\g_\v)^*$ via the trace pairing. The moment map $\mu_\v$ is $\G_\mathbf{v}$-equivariant.

Let
$\xihat=(\xi_i)_i\in \K^I$ be such that $\xihat\centerdot
\mathbf{v}=0$. Then $$(\xi_i.{\rm Id})_i\in\gl_\v$$ is in fact in ${\rm M}(\mathbf{v})^0$. By abuse of notation we denote by $\xihat$ the element $(\xi_i.{\rm Id}_{v_i})_i\in{\rm M}(\mathbf{v})^0$. The affine variety $\mu_{\mathbf{v}}^{-1}(\xihat)$
is $\G_\mathbf{v}$-stable.

Define $$\mathfrak{M}_{\xihat,\thetahat}(\mathbf{v}):=\mu_\v^{-1}(\xihat)/\!/_\thetahat\G_\v.$$
 We define $\mathfrak{M}_{\xihat,\thetahat}^s(\mathbf{v})$ as the image of $\mu_\v^{-1}(\xihat)^s$ in $\mathfrak{M}_{\xihat,\thetahat}(\mathbf{v})$. By Theorem \ref{goodquotient}, it is an open subset of $\mathfrak{M}_{\xihat,\thetahat}(\mathbf{v})$.

Since stabilizers in $\G_\v$ of quiver representations are connected, the action of ${\rm G}_{\mathbf{v}}$ on the space ${\bf M}_{\thetahat}^s(\overline{\Gamma},{\bf v})$ is set-theoritically free and so the restriction $\mu_\v^{-1}(\xi)^s\rightarrow\mathfrak{M}_{\xihat,\thetahat}^s(\mathbf{v})$ of $\varphi$  is the set theoritical quotient $\mu_{\mathbf{v}}^{-1}(\xihat)^s\rightarrow\mu_{\mathbf{v}}^{-1}(\xihat)^s/\G_\mathbf{v}$. By \cite[Lemma 6.5]{reineke}, the map $\mu_{\mathbf{v}}^{-1}(\xihat)^s\rightarrow\mu_{\mathbf{v}}^{-1}(\xihat)^s/\G_\mathbf{v}$ is actually a principal $\G_\v$-bundle (in the \'etale topology).

We put $\mathfrak{M}_{\xi}(\mathbf{v}):=\mathfrak{M}_{\xihat,0}(\mathbf{v})$. It is the affine GIT quotient $
\mu_{\mathbf{v}}^{-1}(\xihat)/\!/\G_\mathbf{v}={\rm Spec}\big(\K[\mu_{\mathbf{v}}^{-1}(\xihat)]^{{\rm G}_{\mathbf{v}}}\big)
$. The set $\mathfrak{M}_{\xihat}(\mathbf{v})$ parameterizes the set of conjugacy classes of the semisimple representations of $
\mu_{\mathbf{v}}^{-1}(\xihat)$. Under this parameterization, the open subset $\mathfrak{M}_{\xihat}^s(\mathbf{v})$ of $0$-stable points coincides with the set of conjugacy classes of simple representations.

The natural projective morphism $\pi:\mathfrak{M}_{\xihat,\thetahat}(\mathbf{v})\rightarrow\mathfrak{M}_{\xi}(\mathbf{v})$  takes a representation to its semi-simplification. 

Let $C_\Gamma=(c_{ij})_{i,j}$ be the Cartan matrix of the quiver $\Gamma$, namely

$$c_{ij}=\begin{cases} 2-2(\text{the number of edges joining $i$ to itself})\hspace{.2cm}\text{if }i=j\\
          - (\text{the number of edges joining $i$ to $j$})\hspace{1.2cm}\text{ otherwise}.
         \end{cases}$$

We say that a variety $X$ is of \emph{pure dimension} $d$ if its irreducible components are all of same dimension $d$. We have the following well-known theorem (the irreducibility is an important result of Crawley-Boevey \cite{crawley-quiver}).

\begin{theorem} Let $\thetahat\in\Z^I$ be such that $\thetahat\centerdot\v=0$. If $\mathfrak{M}_{\xihat,\thetahat}^s(\mathbf{v})\neq\emptyset$, then  it is nonsingular of pure dimension  $2-{^t}\v C_\Gamma\v$. If $\mathfrak{M}_\xihat^s(\mathbf{v})$ is not empty, then $\mathfrak{M}_{\xihat,\thetahat}^s(\mathbf{v})$ is also not empty and $\mathfrak{M}_{\xihat,\thetahat}(\mathbf{v})$ is irreducible.
\label{irrquiver}\end{theorem}

\begin{proof} First a simple representation is necessarily $\thetahat$-stable, hence   $\mathfrak{M}_\xihat^s(\mathbf{v})\neq\emptyset$ implies $\mathfrak{M}_{\xihat,\thetahat}^s(\mathbf{v})\neq\emptyset$. It is a result of Crawley-Boevey \cite[Theorem 1.2]{crawley-quiver} that the existence of simple representations in $\mu_\v^{-1}(\xihat)$ implies the irreducibility of  $\mu_\v^{-1}(\xihat)$ and so the irreducibility of  $\mathfrak{M}_{\xihat,\thetahat}^s(\mathbf{v})$ and $\mathfrak{M}_{\xihat,\thetahat}(\mathbf{v})$. 
Note that a point $\alpha\in\mu_\v^{-1}(\xihat)$ is nonsingular if $\mu_\v$ is smooth at $\alpha$, that is if the stabilizer  of $\alpha$ in $\G_\v$ is trivial. From this we deduce that the space $\mu_\v^{-1}(\xihat)^s$ of $\thetahat$-stable representations is a nonsingular space of dimension ${\rm dim}\,{\bf M}\big(\overline{\Gamma},\mathbf{v}\big)-{\rm dim}\,\G_v$, and so that $\mathfrak{M}_{\xihat,\thetahat}^s(\mathbf{v})$ is nonsingular  of dimension $$2-{^t}\v C_\Gamma\v={\rm dim}\,{\bf M}\big(\overline{\Gamma},\mathbf{v}\big)-2{\rm dim}\,\G_v.$$\end{proof}

We put an order on $\Z^I$ as follows: we say that $\w\leq \v$ if we have $w_i\leq v_i$ for each $i\in I$. We denote by $\calE(\v)$ the set of $\w$ such that $0<\w<\v$, $\xihat\centerdot\w=0$ and $\mu_\w^{-1}(\xihat)\neq\emptyset$.

For $\w\in \N^I$, we denote by $H_\w$ the hyperplane $\{\alpha\in\Q^I\,|\,\alpha\centerdot\w=0\}$ of $\Q^I$. Put $H_{\v\w}:=H_\v\cap H_\w$ and $$D_\v:=H_\v-\bigcup_{\w\in\calE(\v)}H_{\v\w}.$$We say that $\v$ is \emph{indivisible} if the gcd of $\{v_i\}_{i\in I}$ is $1$. Note that $D_\v$ is not empty if and only if $\v$ is \emph{indivisible}.

When $\v$ is indivisible, the spaces $H_{\v\w}$ are hyperplanes of $H_\v$ and so defines a system of \emph{faces}      \cite[Chapter 1,\S 1]{bourbaki}.

\begin{definition}We say that $\thetahat$ is \emph{generic} with respect to $\v$ if $\thetahat\in D_\v$.\label{gen6}\end{definition}

If $\thetahat$ is generic then $\thetahat$-semistability coincides with $\thetahat$-stability, and so $$\mathfrak{M}_{\xihat,\thetahat}^s(\mathbf{v})=\mathfrak{M}_{\xihat,\thetahat}(\mathbf{v}).$$The variety $\mathfrak{M}_{\xihat,\thetahat}(\mathbf{v})$ is thus nonsingular for generic $\thetahat$.

We have \cite{ALE}\cite[\S 2.5]{nakajima-quiver4}:

\begin{proposition}  Assume that $\thetahat$ is generic and that $\mathfrak{M}_\xihat^s(\mathbf{v})\neq\emptyset$. Then the map $\pi:\mathfrak{M}_{\xihat,\thetahat}(\mathbf{v})\rightarrow\mathfrak{M}_\xihat(\mathbf{v})$ is a resolution of singularities.
\label{resolquiver}\end{proposition}

The following proposition is proved in \cite[Proof of Proposition 2.2.6]{hausel-letellier-villegas}.

\begin{theorem}Assume that $\K=\C$ and that $\thetahat$ is generic. Then for any parameter $\xihat$, the varieties $\mathfrak{M}_{\xihat,\thetahat}(\v)$ and $\mathfrak{M}_{0,\thetahat}(\v)$ have isomorphic cohomology supporting pure mixed Hodge structure. 
\label{HLRpure}\end{theorem}

We also have the following result of Nakajima \cite[Appendix B]{crawley-boevey-etal}.

\begin{theorem}Assume that $\K=\overline{\F}_q$ and that $\thetahat$ is generic. Then there exists $r_0\in\N$ such that for all $r\geq r_O$ the varieties $\mathfrak{M}_{\xihat,\thetahat}(\v)$ and $\mathfrak{M}_{0,\thetahat}(\v)$  have the same number of points over $\F_{q^r}$.\label{Nak}
\end{theorem}

We now give a criterion due to Crawley-Boevey for the non-emptyness of $\mathfrak{M}_{\xihat}^s(\mathbf{v})$.  For $i\in I$ let $\e_i\in\Z^I$ be the vector with $1$ at the vertex $i$ and zero elsewhere and let $\Phi(\Gamma)\subset\Z^I$ be the root system associated to $\Gamma$ defined as in \cite{kac}. We denote by $\Phi^+(\Gamma)$ the set of positive roots. Let $(\,,\,)$ be the symmetric bilinear form on the root lattice $\Z^I$ given by $(\e_i,\e_j)=c_{ij}$. Note that vertices of $\Gamma$ may support loops.

 For $\alpha\in\Z^I$, we put $p(\alpha)=1-\frac{1}{2}(\alpha,\alpha)$. If $\alpha$ is a real root we have $p(\alpha)=0$ and if $\alpha$ is an imaginary root then $p(\alpha)>0$. 

The following theorem is due to Crawley-Boevey \cite[Theorem 1.2]{crawley-quiver}.

\begin{theorem} (i) The space $\mathfrak{M}_{\xi}(\mathbf{v})$ is non-empty if and only if  $\mathbf{v}=\beta_1+\beta_2+\dots$ with $\beta_i\in\Phi^+(\Gamma)$ and $\beta_i\centerdot\xihat=0$ for all $i$.

\noindent (ii) The space $\mathfrak{M}_{\xi}^s(\mathbf{v})$ is non-empty if and only if $\mathbf{v}\in\Phi^+(\Gamma)$ and $p(\mathbf{v})>p(\beta_1)+p(\beta_2)+\dots$ for any nontrivial decomposition of $\mathbf{v}$ as a sum $\mathbf{v}=\beta_1+\beta_2+\dots$ with $\beta_i\in\Phi^+(\Gamma)$ and $\beta_i\centerdot\xihat=0$ for all $i$.
 \label{non-emptycond}\end{theorem}

\subsection{Nakajima's framed quiver varieties}\label{fquiver}

The construction of the so-called framed quiver varieties follows  the above one's except that we have an additional graded vector space $W$.

Let $\Gamma$ and $\v$ be as in \S\ref{uquiver}. Let $\mathbf{w}\in\N^I$ be an other dimension vector. Put $L_{\v,\w}=\bigoplus_{i\in I}{\rm Mat}_{w_i,v_i}(\K)\simeq \bigoplus_{i\in I}{\rm Hom}(\K^{v_i},\K^{w_i})$, $L_{\w,\v}=\bigoplus_{i\in I}{\rm Mat}_{v_i,w_i}(\K)$, and $$\mathbf{M}(\overline{\Gamma},\mathbf{v},\mathbf{w}):=\mathbf{M}(\overline{\Gamma},\mathbf{v})\oplus L_{\v,\w}\oplus L_{\v,\w}.$$An element of $\mathbf{M}(\overline{\Gamma},\mathbf{v},\mathbf{w})$ is then denoted by $(B,a,b)$ with $B\in \mathbf{M}(\overline{\Gamma},\mathbf{v})$, $a\in L_{\v,\w}$ and $b\in L_{\w,\v}$.  The group $\GL_\v$ acts on $\mathbf{M}(\overline{\Gamma},\mathbf{v},\mathbf{w})$ by \begin{equation}g\cdot(B,a,b)=(g\cdot B,a\cdot g^{-1},g\cdot b)\label{action}\end{equation}where $g\cdot B$ is the action defined in \S \ref{uquiver}.

Consider the moment map $$\mu_{\mathbf{v},\mathbf{w}}:\mathbf{M}(\overline{\Gamma},\mathbf{v},\mathbf{w})\rightarrow \gl_\v\simeq (\gl_\v)^*$$that maps $(B,a,b)$ to $-ba+\mu_{\mathbf{v}}(B)$. For $\xihat\in\Z^I$ we denote by $\mathfrak{M}_{\xihat}(\mathbf{v},\mathbf{w})$ the affine framed quiver variety $\mu_{\mathbf{v},\mathbf{w}}^{-1}(\xihat)/\!/\GL_\v$ as in \cite{nakajima-quiver2}. Note that unlike in \S\ref{uquiver}, we do not assume that  $\xihat\centerdot\v=0$.

\begin{definition} Let $\thetahat\in\Z^I$. A point $(B,a,b)\in \mathbf{M}(\overline{\Gamma},\mathbf{v},\mathbf{w})$ is $\thetahat$-\emph{semistable} if the two following conditions are satisfied: 

(i) For any $B$-invariant subspace $S$ of $V$ such that $S_i$ is contained in ${\rm Ker}\,(a_i)$ for all $i\in I$, the inequality 

$\thetahat\centerdot{\rm dim}\hspace{.05cm}S\leq 0$ holds.

(ii) For any $B$-invariant subspace $T$ of $V$ such that $T_i$ contains ${\rm Im}\,(b_i)$ for all $i\in I$, the inequality 

$\thetahat\centerdot{\rm dim}\hspace{.05cm}T\leq\thetahat\centerdot{\bf v}$ holds.

\noindent The point $(B,a,b)$ is called $\thetahat$-\emph{stable} if  strict inequalities hold in (i), (ii) unless $S=0$, $T=V$ respectively. \label{stab}\end{definition}

We denote respectively by $\mathbf{M}_{\thetahat}^{ss}(\overline{\Gamma},\mathbf{v},\mathbf{w})$ and $\mathbf{M}_{\thetahat}^s(\overline{\Gamma},\mathbf{v},\mathbf{w})$ the set of $\thetahat$-semistable and $\thetahat$-stable points. Then $\mathbf{M}_{\thetahat}^s(\overline{\Gamma},\mathbf{v},\mathbf{w})$ is an open subset of $\mathbf{M}_{\thetahat}^{ss}(\overline{\Gamma},\mathbf{v},\mathbf{w})$ on which the group $\GL_\v$ acts set-theoritically freely.

\begin{remark} (i) If $\theta_i\geq 0$ for all $i\in I$, then the condition (ii) of Definition \ref{stab} is always satisfied and so a representation is $\thetahat$-semistable if and only if the condition (i) is satisfied.

\noindent (ii) Let $\thetahat,\thetahat'\in\N^I$ and let $J_\thetahat:=\{i\in I\,|\, \theta_i=0\}$ and $J_{\thetahat'}:=\{i\in I\,|\, \theta_i'=0\}$. If $J_\thetahat\subset J_{\thetahat'}$, then $\mathbf{M}_{\thetahat}^{ss}(\overline{\Gamma},\mathbf{v},\mathbf{w})\subset\mathbf{M}_{\thetahat'}^{ss}(\overline{\Gamma},\mathbf{v},\mathbf{w})$.
\label{zeros}\end{remark}

Let $\chi:\GL_\v\rightarrow\K^{\times}$, $(g_i)\mapsto \prod_i{\rm det}\,(g_i)^{-\theta_i}$ be the character associated to $\thetahat$. Then a representation in $\mathbf{M}(\overline{\Gamma},\mathbf{v},\mathbf{w})$ is $\chi$-semistable if and only if it is $\thetahat$-semistable. The framed quiver variety $\mathfrak{M}_{\xihat,\thetahat}(\mathbf{v},\mathbf{w})$ is defined as $$\mathfrak{M}_{\xihat,\thetahat}(\mathbf{v},\mathbf{w}):=\mu_{\v,\w}^{-1}(\xihat)/\!/_\thetahat \GL_\v.$$Define also $\mathfrak{M}_{\xihat,\thetahat}^s(\mathbf{v},\mathbf{w})$ as the image of $\mu_{\v,\w}^{-1}(\xihat)^s$ in $\mathfrak{M}_{\xihat,\thetahat}(\mathbf{v},\mathbf{w})$. If  not empty, the variety $\mathfrak{M}_{\xihat,\thetahat}^s(\mathbf{v},\mathbf{w})$ is a nonsingular open  subset of $\mathfrak{M}_{\xihat,\thetahat}(\mathbf{v},\mathbf{w})$.

Note that $\mathfrak{M}_{\xihat,0}(\mathbf{v},\mathbf{w})$ is the affine framed quiver variety $\mathfrak{M}_{\xihat}(\mathbf{v},\mathbf{w})$ as all points of $\mathbf{M}(\overline{\Gamma},\mathbf{v},\mathbf{w})$ are $0$-semistable. We thus have a natural projective morphism $\pi:\mathfrak{M}_{\xihat,\thetahat}(\mathbf{v},\mathbf{w})\rightarrow\mathfrak{M}_{\xihat}(\mathbf{v},\mathbf{w})$. 

It was observed by Crawley-Boevey \cite[Introduction]{crawley-quiver} that any framed quiver variety can be in fact realized as an ``unframed'' quiver variety of  \S\ref{uquiver}. This is done as follows.

From $\Gamma$ and  $W$ we construct a new quiver $\Gamma^*$ by adding to $\Gamma$ a new vertex $\infty$ and and for each vertex $i$ of $\Gamma$, we add $w_i$ arrows starting at $\infty$ toward $i$. Put $I^*=I\cup\{\infty\}$. We then define $\big(\bf v^*,\thetahat^*\big)\in\N^{I^*}\times\Z^{I^*}$ as follows. We put 

(i) $v_i^*=v_i$ if $i\in I$ and $v_\infty^*=1$,

(ii) $\theta^*_i=\theta_i$ if $i\in I$ and $\theta^*_\infty=-\thetahat\cdot\mathbf{v}$.

We have a natural group embedding $\GL_\v\hookrightarrow \GL_{\v^*}$ that sends an element $g=(g_i)_{i\in I}$ to the element $g^*=(g_i^*)_{i\in I^*}$ with $g_i^*:=g_i$ if $i\in I$ and $g_\infty^*:=1$. This induces an isomorphism $\GL_\v\simeq {\rm G}_{\v^*}=\GL_{\v^*}/\K^{\times}$. We have a $\GL_\v$-equivariant isomorphism $\mathbf{M}(\overline{\Gamma}{^*},\mathbf{v}^*)\rightarrow \mathbf{M}(\overline{\Gamma},\mathbf{v},\mathbf{w})$. Under this isomorphism, the $\thetahat$-semistability (resp. stability) of Definition \ref{stab} coincides with the $\thetahat^*$-semistability (resp. stability) in \S\ref{uquiver}. 

In the context of framed quiver, we say that $\thetahat$ is \emph{generic} if $\thetahat^*$ is generic with respect to $\bf v^*$ in the sense of Definition \ref{gen6}. In this case we have 
$$
\mathbf{M}_{\thetahat}^{ss}(\overline{\Gamma},\mathbf{v},\mathbf{w})=\mathbf{M}_{\thetahat}^s(\overline{\Gamma},\mathbf{v},\mathbf{w})
$$
We have (see Nakajima \cite{ALE}):

\begin{proposition} Assume that $\thetahat$ is generic and that $\mathfrak{M}_{\xihat}^s(\mathbf{v},\mathbf{w})\neq\emptyset$. Then  $\mathfrak{M}_{\xihat,\thetahat}(\mathbf{v},\mathbf{w})=\mathfrak{M}_{\xihat,\thetahat}^s(\mathbf{v},\mathbf{w})$ and the map $\pi:\mathfrak{M}_{\xihat,\thetahat}(\mathbf{v},\mathbf{w})\rightarrow\mathfrak{M}_{\xihat}(\mathbf{v},\mathbf{w})$ is a resolution of singularities.
\label{resolquiver2}\end{proposition}

\begin{remark} If $\theta_i>0$ for all $i$, then $\thetahat^*$ is always generic with respect to $\v^*$.
\label{remgeneric} \end{remark}

\subsection{Quiver varieties of type $A$}\label{type}

We review known results by Kraft-Procesi \cite{KP}, Nakajima \cite{nakajima-quiver2} \cite{nakajima-quiver3}, Crawley-Boevey  \cite{crawley-mat} \cite{crawleyind} and Shmelkin \cite{shmelkin} and give a slight generalization of some of them.

\subsubsection{Partitions and types}\label{partype}

We denote by $\calP$ the set of all partitions including the unique
partition $0$ of $0$, by $\calP^*$ the set of non-zero partitions
and by $\calP_n$ the set of partitions of $n$. Partitions $\lambda$
are denoted by $\lambda=(\lambda_1,\lambda_2,\dots)$, where
$\lambda_1\geq\lambda_2\geq\cdots\geq 0$, or by $(1^{n_1},2^{n_2},\dots)$ where $n_i$ denotes the number of parts of $\lambda_i$ equal to $i$. We put
$|\lambda|:=\sum_i\lambda_i$ for the size of $\lambda$. The length of
$\lambda$ is the maximum $i$ with $\lambda_i>0$ and we denote by $\lambda'$ the dual partition of $\lambda$. For two partitions $\lambda=(\lambda_1,\dots,\lambda_r)$ and $\mu=(\mu_1,\dots,\mu_s)$ we define the partition $\lambda+\mu$ as $(\lambda_1+\mu_1,\lambda_2+\mu_2,\dots)$, and for $\lambda=(1^{n_1},2^{n_2},\dots)$, $\mu=(1^{m_1},2^{m_2},\dots)$, we define the union $\lambda\cup\mu$ as $(1^{n_1+m_1},2^{n_2+m_2},\dots)$. For a partition $\lambda=(\lambda_1,\dots,\lambda_s)$ and a positive integer $d$, we denote by $d\cdot\lambda$ the partition $(d\lambda_1,\dots,d\lambda_s)$. Recall that $(\lambda+\mu)'=\lambda'\cup\mu'$.

Given a total ordering $\leq_t$ on $\calP$, we denote by $\tilde{\bf T}^t$ the set of non-increasing sequences $\tilde{\omega}=\omega^1\omega^2\cdots\omega^r$ with $\omega^i\in\calP$ and let $\tilde{\bf T}_n^t$ be the subset of sequences $\tilde{\omega}$ such that $\sum_i|\omega^i|=n$. We will see in \S \ref{adjoint} that the set $\tilde{\bf T}_n^t$ parameterizes the types of the adjoint orbits in $\gl_n(\K)$. Although the choice of a particular total ordering will be sometimes convenient it will not be essential for the results of this paper. We will actually often use the notation $\tilde{\bf T}$ and $\tilde{\bf T}_n$ instead of $\tilde{\bf T}^t$ and $\tilde{\bf T}_n^t$ when the reference to the ordering $\leq_t$ is not necessary. 

We extend the ordering $\leq_t$ to a  total ordering on the set $\{(d,\lambda)\,|\, d\in\N^*, \lambda\in\calP^*\}$ which we continue to denote by $\leq_t$ as follows. If $\mu\neq\lambda$, we say that $(d,\mu)\leq_t(d',\lambda)$ if  $\mu\leq_t\lambda$, and we say that $(d,\lambda)\leq_t(d',\lambda)$ if $d'\leq d$. We denote by $\mathbf{T}^t$ the set of all non-increasing sequences $\omega=(d_1,\lambda^1)(d_2,\lambda^2)\cdots(d_r,\lambda^r)$ and by ${\bf T}_n^t$ the subset of ${\bf T}^t$ of these sequences which satisfy $|\omega|:=\sum_id_i|\lambda^i|=n$. The first coordinate of a pair $(d,\lambda)$ is called the \emph{degree}. We will see in \S \ref{gen} that ${\bf T}_n^t$ parametrizes both the types of the adjoint orbits in $\gl_n(\F_q)$ and the types of the irreducible characters of $\GL_n(\F_q)$. As for $\tilde{\bf T}$ and $\tilde{\bf T}_n$ we will often use the notation ${\bf T}$ and ${\bf T}_n$ instead of ${\bf T}^t$ and ${\bf T}_n^t$.

Since the terminology ``type'' has two meanings in this paper, we use the letters $\{\omega,\tau,\dots\}$ to denote the  elements of ${\bf T}$ and the symbols $\{\tomega,\ttau,\dots\}$ for the elements of $\tilde{\bf T}$.

Given a type $\omega=(d_1,\omega^1)\cdots (d_r,\omega^r)\in {\bf T}^t$, we assign the type 

$$\tilde{\omega}=\overbrace{\omega^1\cdots\omega^1}^{d_1}\overbrace{\omega^2\cdots\omega^2}^{d_2}\cdots\overbrace{\omega^r\cdots\omega^r}^{d_r}$$of $\tilde{\bf T}^t$. We thus have  a surjective map $\mathfrak{H}:{\bf T}^t\rightarrow\tilde{\bf T}^t$, $\omega\mapsto\tilde{\omega}$. 

Let $$\tomega=\overbrace{\omega^1\cdots\omega^1}^{a_1}\overbrace{\omega^2\cdots\omega^2}^{a_2}\cdots\overbrace{\omega^r\cdots\omega^r}^{a_r}\in\tilde{{\bf T}}^t$$ with $\omega^i\neq\omega^j$ if $i\neq j$ and put $$W_{\tilde{\omega}}:=\prod_{i=1}^rS_{a_i}.$$Note that the elements in the fiber $\mathfrak{H}^{-1}(\tomega)$ are parametrized by $\calP_{a_1}\times\cdots\times\calP_{a_r}$ and so by the conjugacy classes of $W_{\tomega}$.

\subsubsection{Zariski closure of adjoint orbits as quiver varieties}\label{adjoint}

Let $A\in\gl_n(\K)$ with semisimple part $A_s$ and nilpotent part $A_n$. We assume for simplicity that $A_s$ is a  diagonal matrix so that its centralizer $L$ in $\GL_n$ is exactly a product of $\GL_{m_i}$'s. We have $A=A_s+A_n$ with
$[A_s,A_n]=0$ where $[x,y]=xy-yx$. We put
$C_{\gl_n}(A):=\{X\in\gl_n|\,[A,X]=0\}={\rm Lie}(L)$. Let $C$ be the $L$-orbit of $A_n$. Then  the $\GL_n$-conjugacy class of the pair $(L,C)$ is called the \emph{type} of the $\GL_n$-orbit $\calO$ of $A$. 

Fix a total ordering $\leq_t$ on $\calP$. The types of the adjoint orbits of $\gl_n$ are parameterized by the set $\tilde{\bf T}_n^t$ as follows.

Let $m_1,\dots,m_r$ be the multiplicities of the $r$ distinct
eigenvalues $\alpha_1,\dots,\alpha_r$ of $A$. We may assume that $A_s$ is the diagonal matrix $$\left(\overbrace{\alpha_1,\dots,\alpha_1}^{m_1},\overbrace{\alpha_2,\dots\alpha_2}^{m_2},\dots,\overbrace{\alpha_r,\dots,\alpha_r}^{m_r}\right).$$The Jordan form of the element $A_n\in
C_{\gl_n}(\sigma)=\gl_{m_1}\oplus\gl_{m_2}\oplus\cdots\oplus\gl_{m_r}$ defines
a unique partition $\omega^i$ of $m_i$ for each
$i\in\{1,2,\dots,r\}$. Re-indexing if necessary we may assume that
$\omega^r\leq_t\omega^{r-1}\leq_t\cdots\leq_t\omega^1$ in which case we have
$\tomega=\omega^1\cdots\omega^r\in\tilde{\bf T}_n^t$. Conversely, any element of $\tilde{\bf T}_n^t$ arises as the type of some adjoint orbit of $\gl_n$. Types of semisimple orbits are of the form
$(1^{n_1})\cdots(1^{n_r})$ and types of nilpotent orbits are just partitions of $n$.

\begin{lemma}The dimension of $\calO$ is \begin{equation}n^2-\sum_{j=1}^r\langle
\omega^j,\omega^j\rangle\label{dimorb}\end{equation} where for a partition
$\lambda=(\lambda_1,\lambda_2,\dots)$, we put $\langle
\lambda,\lambda\rangle=2n(\lambda)+|\lambda|$ with
$n(\lambda)=\sum_{i\geq 1}(i-1)\lambda_i$.\label{dimO}\end{lemma}\vspace{.2cm}

We now explain how to construct a quiver $\Gamma_\calO$ and a pair $(\xihat_\calO,\v_\calO)$ from $\calO$ such that $\mathfrak{M}_{\xihat_\calO}(\v_\calO,\w)\simeq\overline{\calO}$. While the quiver $\Gamma_\calO$ and $\w$ will be independent from the choice of $\leq_t$, the parameters $\xihat_\calO,\v_\calO$ will depend on the choice of $\leq_t$.

\noindent We draw the Young diagrams respectively of $\omega^1,\dots,\omega^r$ from the left to the right and we label the columns from the left to the right (with the convention that partitions are represented by the rows of the Young diagrams).  Let $d$ be the total number of columns and let $n_i$ be the length of the $i$-th-column with respect to this labeling. We define the dimension vector $\v_\calO=(v_1,\dots,v_{d-1})$ by $v_1:=n-n_1$ and $v_i:=v_{i-1}-n_i$ for $i>1$ and the parameter $\zetahat_\calO=(\zeta_1,\dots,\zeta_d)$ as follows. If the $i$-th column belongs to the Young diagram of $\omega^j$ then we put $\zeta_i=\alpha_j$.

We then have

$$(A-\zeta_1{\rm Id})\cdots(A-\zeta_d{\rm Id})=0.$$

\begin{example} Take the lexicographic ordering  for $\leq_t$ and assume that $\mathcal{O}$ is of type $(2,2)(1,1)$ with eigenvalues $\alpha_1$ and $\alpha_2$ respectively of multiplicity $4$ and $2$. The corrresponding Young diagrams are $$\overset{1\hspace{.3cm}2}{\yng(2,2)} \quad \overset{3}{\yng(1,1)}
$$Then the vector dimension is $\v_\calO=(4,2)$ and $\zetahat_\calO=(\alpha_1,\alpha_1,\alpha_2)$.

\end{example}

We have\begin{lemma} For $i>0$, the integer $v_i$ is the rank of the partial product $$(A-\zeta_1{\rm Id})\cdots(A-\zeta_i{\rm Id}).$$\label{partrk}\end{lemma}

The following result is due to Crawley-Boevey \cite{crawleyind} (in characteristic zero with $\calO$ nilpotent  it is due to Kraft and Procesi \cite{KP}).

\begin{theorem} Let $B\in\gl_n$. The following assertions are equivalent.

\noindent (1) $B\in\overline{\calO}$.

\noindent (2) There is a flag of subspaces $\K^n=V_0\supset V_1\supset V_2\supset\cdots\supset V_{d-1}\supset V_d=0$ with ${\rm dim}\, V_i=v_i$ and such that $(B-\zeta_i{\rm Id})(V_{i-1})\subset V_i$ for all $1\leq i\leq d$.

\noindent (3) There are vector spaces $V_j$ and linear maps $a$, $b$, $\phi_j$, $\phi_j^*$,
\[
V=V_0 \rightoverleft^{b}_{a} V_{1} \rightoverleft^{\phi_{1}^*}_{\phi_{1}} V_{2}
\rightoverleft^{\phi_{2}^*}_{\phi_{2}} \dots
\rightoverleft^{\phi_{d-1}^*}_{\phi_{d-1}} V_{d}=0
\]
where $V_j$ has dimension $v_j$, and satisfying
\[
\begin{split}
B &=ab + \zeta_{1} {\rm Id},
\\
\phi_{j}\phi_{j}^*-\phi_{j-1}^* \phi_{j-1}
&= (\zeta_j-\zeta_{j+1}) {\rm Id},
\qquad (1\le j < d).
\end{split}
\] 
where $\phi_0^*=b$ and $\phi_0=a$.

\label{CB5}\end{theorem}

\begin{remark} We obtain (3) from (2) by putting $\phi_i^*:=(B-\zeta_{i+1}{\rm Id})|_{V_i}$ and by letting $\phi_i$ be the inclusion $V_{i+1}\subset V_i$.
 
\label{remCB}
\end{remark}

Let $\Gamma_\calO$ be the quiver
$$\xymatrix{\bullet^1&\bullet^2\ar[l]&\cdots\ar[l]&\bullet^{d-1}\ar[l]}$$whose underlying graph is the Dynkin diagram of type $A_{d-1}$ and put $I:=\{1,\dots,d-1\}$. Put $\w:=(n,0,\dots,0)$ and define $\xihat_\calO=(\xi_1,\dots,\xi_{d-1})$ by $\xi_j:=\zeta_j-\zeta_{j+1}$.

\begin{theorem} The map $q:\mu_{\v_\calO,\w}^{-1}(\xihat_\calO)\rightarrow\overline{\calO}$ given by $(B,a,b)\mapsto ab+\zeta_1{\rm Id}$ is well-defined and surjective. It induces a bijective morphism $\tilde{q}:\mathfrak{M}_{\xihat_\calO}(\v_\calO,\w)\longrightarrow\overline{\calO}$. If $\K=\C$, then $q$ is a categorical quotient by $\GL_\v$, i.e., the map $\mathfrak{M}_{\xihat_\calO}(\v_\calO,\w)\longrightarrow\overline{\calO}$ is an isomorphism. The bijective  morphism $\tilde{q}$ restricts to $\mathfrak{M}_{\xihat_\calO}^s(\v_\calO,\w)\longrightarrow\calO$.
\label{NC}\end{theorem}

\begin{proof} The first assertion follows from Theorem \ref{CB5}. The second assertion can be proved using the ``First Fundamental Theorem of Invariant Theory'' as in Kraft and Procesi \cite[\S 2]{KP}. The third assertion follows from the second one using Proposition \ref{affinegoodquotient} (this assertion is actually stated in Kraft and Procesi \cite[\S 2]{KP} for nilpotent orbits and  in Crawley-Boevey \cite[Lemma 9.1]{crawleyred} for any orbits). For an arrow of $\overline{\Gamma}_\calO$ with tail $i$ and head $j$, we denote by $B_{i,j}$ the corresponding coordinate of $B$. By Crawley-Boevey \cite[\S 3]{crawley-mat}, we have $f(B,a,b)\in\calO$ if and only if  the $B_{i+1,i}$'s and $a$ are all injective and if the maps $B_{i,i+1}$'s and $b$ are all surjective, i.e., $(B,a,b)$ is a $0$-stable representation. Hence the last assertion.
 \end{proof}

\begin{remark}  If $\calC$ is the $\GL_\v$-orbit of any representation $(B,a,b)\in\mu_{\v_\calO,\w}^{-1}(\xihat_\calO)$ then $a'b'=ab$ for any $(B',a',b')\in\overline{\calC}$.
\label{catorb}\end{remark}

We says that $(n_1,\dots,n_{d-1})\in(\Z_{>0})^{d-1}$ is \emph{decreasing} if $n_1>\cdots>n_{d-1}$.

\begin{remark}Let $\v=(v_1,\dots,v_{d-1})$ be a decreasing sequence with $n>v_1$, and  let $\xihat=(\xi_1,\dots,\xi_{d-1})$. Then there is a total ordering $\leq_t$ on $\calP$ and an adjoint orbit $\calO$ such that $(\xihat,\v)=(\xihat_\calO,\v_\calO)$  if and only if the following condition is satisfied, see Crawley-Boevey \cite[\S 2]{crawleyind}. 

(*) For any $j\in I$ with $\xi_j=0$ we have $v_{j-1}-v_j\geq v_j-v_{j+1}$ with $v_0:=n$. 
\label{monotone}\end{remark}

\subsubsection{Partial resolutions of Zariski closure of adjoint orbits as quiver varieties}\label{parresol}

Let $P$ be a parabolic subgroup of $\GL_n(\K)$ (which for simplicity is assumed to contain the upper triangular matrices), $L$ a Levi subgroup of $P$ and let $\Sigma=\sigma+C$ where $\sigma$ is in the center $z_\mathfrak{l}$ of the Lie algebra $\mathfrak{l}$ of $L$ and where $C$ is a nilpotent orbit of $\mathfrak{l}$. We denote  by $U_P$ the unipotent radical of $P$ and by $\mathfrak{u}_P$ the Lie algebra of $U_P$. The aim of this section is to identify the variety$$\mathbb{X}_{L,P,\Sigma}:=\left\{(X,gP)\in\gl_n\times (\GL_n/P)\,\left|\, g^{-1}Xg\in\overline{\Sigma}+\mathfrak{u}_P\right\}\right.$$with a quiver variety of the form $\mathfrak{M}_{\xihat,\thetahat}(\v,\w)$ when $\K=\C$ (in positive characteristic we have a bijective morphism $\mathfrak{M}_{\xihat,\thetahat}(\v,\w)\rightarrow \mathbb{X}_{L,P,\Sigma}$).

Note that \begin{equation}{\rm dim}\,\mathbb{X}_{L,P,\Sigma}={\rm dim}\, \GL_n-{\rm dim}\, L+{\rm dim}\,\Sigma.\label{dimfor}\end{equation}

Taking a $\GL_n$-conjugate of $L$ if necessary, we may assume that $L=\GL_{s_{p+1}}\times \GL_{s_p}\times\cdots\times\GL_{s_1}$. Since $\sigma$ is in the center of $\mathfrak{l}$, we may write $\sigma$ as the diagonal matrix 

$$\left(\overbrace{\epsilon_{p+1},\dots,\epsilon_{p+1}}^{s_{p+1}},\overbrace{\epsilon_p,\dots,\epsilon_p}^{s_p},\dots,\overbrace{\epsilon_1,\dots,\epsilon_1}^{s_1}\right).$$

The nilpotent orbit $C$ of $\mathfrak{l}$ decomposes as $$C=C_{p+1}\times\cdots\times C_1$$with $C_i$ a nilpotent orbit of $\gl_{s_i}$. For $i=1,\dots,p+1$, let $\mu^i$ be the partition of $s_i$ which gives the size of the blocks of the Jordan form of $C_i$. 

We choose a total ordering $\leq_t$ on $\calP$ such that, re-ordering if necessary, we have $\mu^{p+1}\leq_t \mu^p\leq_t\cdots\leq_t\mu^1$ and the following condition is satisfied 

(**) if $\epsilon_i=\epsilon_j$ then for any $i\leq k\leq j$ we have $\epsilon_k=\epsilon_i$. 

This choice of $\leq_t$ is only for convenience (see above Example \ref{diffvect}).

Let $\alpha_1,\dots,\alpha_k$ be the distinct eigenvalues of $\sigma$ with respective multiplicities $m_1,\dots,m_k$. For each $i=1,\dots,k$, we define a partition $\lambda_i$ of $m_i$ as the sum of the partitions $\mu^r$ where $r$ runs over the set $\{r\,|\, \epsilon_r=\alpha_i\}$. The partitions $\lambda_1,\dots,\lambda_k$ defines a unique nilpotent orbits of the Lie algebra $\mathfrak{m}$ of $M:=C_{\GL_n}(\sigma)$. Let $v$ be an element in this orbit and let $\calO$ be the unique adjoint orbit of $\gl_n$ that contains $\sigma+v$. The following proposition is well-known.

\begin{proposition} The image of the projection $p:\mathbb{X}_{L,P,\Sigma}\rightarrow\gl_n$ is $\overline{\calO}$. Moreover it induces an isomorphism $p^{-1}(\calO)\simeq \calO$. If $M=L$, the map $p$ is an isomorphism $\mathbb{X}_{L,P,\Sigma}\simeq\overline{\calO}$.
 \label{proj}\end{proposition}

We have ${\rm dim}\,\calO={\rm dim}\,\mathbb{X}_{L,P,\Sigma}$ and so

\begin{equation}{\rm dim}\,\calO={\rm dim}\, G-{\rm dim}\, L+{\rm dim}\,\Sigma\label{dimcalO}\end{equation}

We now denote by $\calF$ the variety of partial flags $\{0\}=E^{p+1}\subset E^p\subset\cdots\subset E^1\subset E^0=\K^n$ with ${\rm dim}\, E^{r-1}/E^r=s_r$. For an element $X\in\gl_n$ that leaves stable a partial flag $$\left(\{0\}=E^{p+1}\subset E^p\subset\cdots\subset E^1\subset E^0=\K^n\right)\in\calF$$we denote by $X_r$, $r=1,\dots,p+1$, the induced endomorphism of $E^{r-1}/E^r\simeq\K^{s_r}$.

We denote by $\mathbb{Z}_{L,P,\Sigma}$ (resp. $\mathbb{Z}_{L,P,\Sigma}^o$) the subvariety of $\gl_n\times\calF$ of pairs $(X,f)$ such that $X\cdot f=f$ and such that for all $r=1,\dots,p+1$, we have $X_r\in \epsilon_r{\rm Id}+\overline{C}_r$ (resp. $X_r\in\epsilon_r{\rm Id}+C_r$).

Note that $\calF\simeq \GL_n/P$ and so the two varieties $\mathbb{Z}_{L,P,\Sigma}$ and $\mathbb{X}_{L,P,\Sigma}$ are isomorphic.

There exist a unique positive integer $d$, a decreasing sequence of positive integers $$\v_{L,P,\Sigma}=(v_1,\dots,v_{d-1})\in\left(\Z_{>0}\right)^{d-1},$$and $p$ elements $i_1<\cdots<i_p$ in $\{1,\dots,d-1\}$ such that if we put $i_0:=0$, $v_0:=n$, $i_{p+1}:=d$, and $v_d:=0$, then for each $r=1,\dots,p+1$, we have $v_{i_{r-1}}-v_{i_r}=s_r$, and $\left(v_{i_{r-1}}-v_{i_{r-1}+1},\dots,v_{i_r-1}-v_{i_r}\right)$ is the dual partition of $\mu^r$.

This defines a type $A_{d-1}$ quiver $\Gamma_{L,P,\Sigma}$ as in \S \ref{adjoint}. We keep the same $\w$ as in \S \ref{adjoint} and we define  $\zetahat_{L,P,\Sigma}=(\zeta_1,\dots,\zeta_d)$ by $\zeta_j=\epsilon_{r+1}$ if $i_r<j\leq i_{r+1}$ with $r=0,\dots,p$.
As in \S \ref{adjoint}, this defines a unique parameter $\xihat_{L,P,\Sigma}=(\xi_1,\dots,\xi_{d-1})\in \K^I$ such that  $\xi_i=\zeta_i-\zeta_{i+1}$. We now choose a stability parameter $\thetahat\in\left(\N\right)^I$ with the requirement that $\thetahat_j\neq 0$ exactly when $j\in\{i_1,\dots,i_p\}$.

 The quiver $\Gamma_{L,P,\Sigma}$ defined above is the same as the quiver $\Gamma_\calO$ associated with the adjoint orbit $\calO$ in \S \ref{adjoint}. Denote by $(\v_\calO,\xihat_\calO)$ the datum arising from $\calO$ as in  \S \ref{adjoint} with respect to $\leq_t$. The dimension vector $\v_\calO$ might differ from $\v_{L,P,\Sigma}$ as shown in the example below. However since $\leq_t$ respect the condition (**) on the $\epsilon_i$'s, we always have $\xihat_{L,P,\Sigma}=\xihat_\calO$.

\begin{example} Assume that  $L=\GL_1\times\GL_2\times\GL_2\times\GL_3\times\GL_3$, $C=C_{(1)}\times C_{(1,1)}\times C_{(2)}\times C_{(2,1)}\times C_{(3)}$ where $C_\mu$ denotes the nilpotent orbit corresponding to the partition $\mu$, and that $\sigma$ is the diagonal matrix $$(\underbrace{\alpha,\alpha,\alpha,\alpha,\alpha}_5,\underbrace{\beta,\beta,\beta,\beta,\beta,\beta}_6)$$ with $\alpha\neq\beta$. Clearly $\sigma$ is in the center of $\mathfrak{l}$ and $M=\GL_5\times\GL_6$. The underlying graph of $\Gamma_{L,P,\Sigma}$ is $A_8$ and $\w=(11,0,0,0)$. 

Assume that $\leq_t$ is the lexicographic ordering. The type of $\calO$ is $(5,1)(4,1)\in\tilde{{\bf T}}_{11}^t$. Note that $(1)\leq_t (1,1)\leq_t (2)\leq_t(2,1)\leq_t(3)$. We thus have $\epsilon_1=\epsilon_2=\beta$ and $\epsilon_3=\epsilon_4=\epsilon_5=\alpha$. Hence $\v_{L,P,\Sigma}=(10,9,8,6,5,4,3,1)$, $(i_1,\dots,i_p)=(3,5,7,8)$,  $\thetahat=(0,0,\theta_3,0,\theta_5,0,\theta_7,\theta_8)$ with $\theta_3,\theta_5,\theta_7,\theta_8>0$, $\zetahat_{L,P,\Sigma}=(\beta,\beta,\beta,\beta,\beta,\alpha,\alpha,\alpha,\alpha)$, $\xihat_{L,P,\Sigma}=(0,0,0,0,\beta-\alpha,0,0,0)$. Finally note that $\v_\calO=(9,8,7,6,5,3,2,1)\neq\v_{L,P,\Sigma}$.
\label{diffvect}\end{example}

The aim of the section is to show that there is a bijective morphism $\mathfrak{M}_{\xihat_{L,P,\Sigma},\thetahat}(\v_{L,P,\Sigma},\w)\rightarrow\mathbb{Z}_{L,P,\Sigma}$ which is an isomorphism when $\K=\C$.

Given $(B,a,b)\in\mu_{\v_{L,P,\Sigma},\w}^{-1}(\xihat_{L,P,\Sigma})$ and an arrow of $\overline{\Gamma}_{L,P,\Sigma}$ with tail $i$ and head $j$, we denote by $B_{i,j}$ the corresponding coordinate of $B$.

For a parameter $x\in \K^I$, put $J_x=\{i\in I\,|\, x_i=0\}$ where $I$ denotes the set of vertices of $\Gamma_{L,P,\Sigma}$. We will need the following lemma:

\begin{lemma} Let $(B,a,b)\in \mu_{\v_{L,P,\Sigma},\w}^{-1}(\xihat_{L,P,\Sigma})$. Then  $(B,a,b)$ is $\thetahat$-semistable if and only if for all $i\in I-J_\thetahat$ the map 
$a\circ B_{2,1}\circ\cdots\circ B_{i,i-1}:\K^{v_i}\rightarrow \K^n$ is injective.
\label{injective}\end{lemma}

\begin{proof} Put $V:=\bigoplus_i\K^{v_i}$. We first construct for each $s\in I$ a $B$-invariant graded subspace $L^s=\bigoplus_iL^s_i$ of $V$. Put $L^s_1:={\rm Ker}\,(a)$, for all $i\in\{2,\dots,s\}$ put $L^s_i:={\rm Ker}\,(a\circ B_{2,1}\circ\cdots\circ B_{i,i-1}\big)$, and for $i>s$ put $L^s_i:=B_{i-1,i}\circ B_{i-2,i-1}\circ\cdots\circ B_{s+1,s+2}\circ B_{s,s+1}\,(L^s_s)$. Let us see that $L^s$ is a $B$-invariant subspace of $V$. For $i<s$ we need to see that $B_{i,i+1}(L^s_i)\subset L^s_{i+1}$. We first prove it when $i=1$. We have $ba-B_{2,1}B_{1,2}=\xi_1{\rm Id}$, hence $(a\circ B_{2,1})(B_{1,2}({\rm Ker}\,(a))=a\circ (ba-\xi_1{\rm Id})({\rm Ker}\,(a))=0$ and so $B_{2,1}(L^s_1)\subset L^s_2$.  Assume that this is true for all $j<i$. At the vertex $i$, we have the relation $B_{i-1,i}B_{i,i-1}-B_{i+1,i}B_{i,i+1}=\xi_2{\rm Id}$. For $x\in L^s_i$ we have 

\begin{align*}a\circ B_{2,1}\circ\cdots \circ B_{i,i-1}\circ B_{i+1,i}\,\big(B_{i,i+1}(x)\big)&=a\circ B_{2,1}\circ\cdots\circ B_{i,i-1}\circ (B_{i-1,i}B_{i,i-1}-\xi_2{\rm Id})(x)\\
 &=a\circ B_{2,1}\circ\cdots\circ B_{i,i-1}\circ (B_{i-1,i}B_{i,i-1}(x)).
\end{align*} We need to see that the RHS is $0$. By definition of $L^s$ it is clear that $B_{i,i-1}\,(L^s_i)\subset L^s_{i-1}$ hence $B_{i,i-1}\,(x)\in L^s_{i-1}$. By induction hypothesis we then have $B_{i-1,i}\,\big(B_{i,i-1}\,(x)\big)\subset L^s_i$. By definition of $L^s_i$ we thus have $a\circ B_{2,1}\circ\cdots\circ B_{i,i-1}\big( B_{i-1,i}\circ B_{i,i-1}(x)\big)=0$. To see that $L^s$ is a $B$-invariant subspace of $V$ it remains to see that for all $i\geq s$ we have $B_{i+1,i}\,(L^s_{i+1})\subset L^s_i$ which again can be proved by induction using the relations at the vertices.

Assume that $(B,a,b)$ is $\thetahat$-semistable. Assume that $s\in I-J_\thetahat$. If the map $a_s:=a\circ B_{2,1}\circ\cdots\circ B_{s,s-1}$ is not injective then $L^s$ is a non-trivial $B$-invariant subspace of $V$ such that $\thetahat\centerdot{\rm dim}\,L^s>0$ (as $\theta_s\neq 0$) which contradicts the stability condition (i) of Definition \ref{stab}. Hence the map $a_s$ must be injective for all $s\in I-J_\thetahat$. 

Let us prove the converse. Assume that $V'$ is a $B$-invariant subspace of $V$ such that $V_1'\subset{\rm ker}\,(a)$. Hence for all $i$ and $x\in V_i'$ we have $B_{2,1}\circ\cdots\circ B_{i,i-1}(x)\in{\rm Ker}\,(a)$, i.e., $a\circ B_{2,1}\circ\cdots\circ B_{i,i-1}(x)=0$, and so $V'_i\subset {\rm Ker}\,(a\circ B_{2,1}\circ\cdots\circ B_{i,i-1})$. Hence for $i\in I-J_\thetahat$ we have $V_i'=0$ by assumption. Therefore $\thetahat\centerdot {\rm dim}\,V'=0$ and so the condition (i) of Definition \ref{stab} is satisfied. \end{proof}

For $(B,a,b)\in\mathfrak{M}_{\xihat_{L,P,\Sigma},\thetahat}(\v_{L,P,\Sigma},\w)$, we denote by $f_{(B,a,b)}$ the partial flag $\{0\}=\calE^{p+1}\subset \calE^p\subset\cdots\subset \calE^1\subset \calE^0=\K^n$ with $\calE^r:={\rm Im}\,\left(a\circ B_{2,1}\circ\cdots\circ B_{i_r,i_r-1}\right)$. By Lemma \ref{injective}, we have $f_{(B,a,b)}\in\calF$.

\begin{proposition} The map $\mu_{\v_{L,P,\Sigma},\w}^{-1}(\xihat_{L,P,\Sigma})^{ss}\rightarrow \mathbb{Z}_{L,P,\Sigma}$, $(B,a,b)\mapsto \left(ab+\zeta_1{\rm Id},f_{(B,a,b)}\right)$ is well-defined and induces a canonical bijective morphism $\mathfrak{M}_{\xihat_{L,P,\Sigma},\thetahat}(\v_{L,P,\Sigma},\w)\rightarrow\mathbb{Z}_{L,P,\Sigma}$ which restricts to $\mathfrak{M}_{\xihat_{L,P,\Sigma},\thetahat}^s(\v_{L,P,\Sigma},\w)\rightarrow\mathbb{Z}_{L,P,\Sigma}^o$ and which makes the following diagram commutative

$$\xymatrix{\mathfrak{M}_{\xihat_{L,P,\Sigma},\thetahat}(\v_{L,P,\Sigma},\w)\ar[rr]\ar[d]^{\pi}&&\mathbb{Z}_{L,P,\Sigma}\ar[d]^{pr_1}\\
\mathfrak{M}_{\xihat_{L,P,\Sigma}}(\v_{L,P,\Sigma},\w)\ar[rr]^{\rho}&&\gl_n}$$where $\rho$ maps a semisimple representation $(B,a,b)$ to $ab+\zeta_1{\rm Id}$. If $\K=\C$ this bijective map is  an isomorphism $\mathfrak{M}_{\xihat_{L,P,\Sigma},\thetahat}(\v_{L,P,\Sigma},\w)\overset{\sim}{\rightarrow}\mathbb{Z}_{L,P,\Sigma}$. 
\label{square}\end{proposition}

If $\theta_i>0$ for all $i$ and if $\xihat_{L,P,\Sigma}=0$, then this is a result of Nakajima \cite[Theorem 7.3]{ALE}, see also \cite{shmelkin} for more details.

\begin{proof} The fact that the diagram is commutative  follows from a generalization of Remark \ref{catorb} to any decreasing dimension vector (see Kraft and Procesi \cite[Proposition 3.4]{KP}). To alleviate the notation we omitt $L,P,\Sigma$ from the notation $\xihat_{L,P,\Sigma},\v_{L,P,\Sigma}, \zetahat_{L,P,\Sigma}, \Gamma_{L,P,\Sigma}$. Let us see that the map $$h:\mu_{\v,\w}^{-1}(\xihat)^{ss}\rightarrow \mathbb{Z}_{L,P,\Sigma}, (B,a,b)\mapsto \left(ab+\zeta_1{\rm Id},f_{(B,a,b)}\right)$$is well-defined. Let $(B,a,b)\in \mu_{\v,\w}^{-1}(\xihat)^{ss}$ and put $X:=ab+\zeta_1{\rm Id}$ and $\calE^r:={\rm Im}\,\left(a\circ B_{2,1}\circ\cdots\circ B_{i_r,i_r-1}\right)$. The fact that $X$ leaves stable the partial flag $f_{(B,a,b)}$ is straightforward from the preprojective relations $$B_{i-1,i}B_{i,i-1}-B_{i+1,i}B_{i,i+1}=\xi_i{\rm Id}$$with $B_{0,1}:=b$ and $B_{1,0}:=a$. 

To alleviate the notation, for all $i<j$ we denote by $f_{j,i}$ the map $B_{i+1,i}\circ\cdots\circ B_{j,j-1}:\K^{v_j}\rightarrow\K^{v_i}$. 

Fix $r\in\{1,\dots,p+1\}$ and define $H=\bigoplus_{i\in I\cup\{0\}}H_i$ by $H_i=\K^{v_i}$ if $i\geq i_r$ and by $H_i={\rm Im}\,(f_{i_r,i})$ if not. From the preprojective relations we see that $(B,a,b)$ leaves $H$ stable and so we can consider the restriction $(B_H,a_H,b_H)$ of $(B,a,b)$ to $H$ and the quotient $(\overline{B},\overline{a},\overline{b})$ of $(B,a,b)$ by $(B_H,a_H,b_H)$. Put $U_i:=\K^{v_i}/H_i$. Then $U_i\simeq \K^{v_i-v_{i_r}}$ if $i<i_r$ and $U_i=\{0\}$ otherwise. From the preprojective relations we see that $X_r:\calE^{r-1}/\calE^r\rightarrow \calE^{r-1}/\calE^r$ coincides, with the map $Y_r:U_{i_{r-1}}\rightarrow U_{i_{r-1}}$ induced by $B_{i_{r-1}+1,i_{r-1}}B_{i_{r-1},i_{r-1}+1}+\zeta_{i_{r-1}+1}{\rm Id}$. In other words the diagram

$$\xymatrix{\calE^{r-1}/\calE^r\ar[d]_{X_r}&&U_{i_{r-1}}\ar[d]^{Y_r}\ar[ll]_{f_{i_{r-1},0}/H_{i_{r-1}}}\\
\calE^{r-1}/\calE^r&&U_{i_{r-1}}\ar[ll]_{f_{i_{r-1},0}/H_{i_{r-1}}}}$$is commutative. 

We want to see that the map $Y_r\in{\rm End}(U_{i_{r-1}})\simeq {\rm End}(\K^{s_r})$ leaves in $\zeta_{i_{r-1}+1}{\rm Id}+\overline{C}_r$.

Consider the subquiver $\Gamma'$

$$\xymatrix{\bullet^{i_{r-1}+1}&\ar[l]\cdots&\bullet^{i_r-1}\ar[l]}$$of $\Gamma$. Put $d':=i_r$, $\w':=(v_{i_{r-1}}-v_{i_r},0,\dots,0)$, $\v':=(v_{i_{r-1}+1}-v_{i_r},v_{i_{r-1}+2}-v_{i_r},\dots,v_{i_r-1}-v_{i_r})$, and $\zetahat'=(\zeta_{i_{r-1}+1},\zeta_{i_{r-1}+2},\dots,\zeta_{i_{r-1}+d'})$. We have $\xihat'_i=0$ for all $i=i_{r-1}+1,\dots,i_r-1$, i.e., $\zeta_{i_{r-1}+1}=\zeta_{i_{r-1}+2}=\cdots=\zeta_{i_{r-1}+d'}$. Consider the projection of $(\overline{B},\overline{a},\overline{b})$ on $$\left(\bigoplus_{i\in\{i_{r-1},\dots,i_r-2\}}{\rm Hom}\left(U_i,U_{i+1}\right)\oplus\bigoplus_{i\in\{i_{r-1}+1,\dots,i_r-1\}}{\rm Hom}\left(U_i,U_{i-1}\right)\right)\simeq {\bf M}(\overline{\Gamma}{'},\v',\w')$$and denote by $(B',a',b')$ the corresponding element in ${\bf M}(\overline{\Gamma}{'},\v',\w')$. Note that $a'$ and $b'$ come from $B_{i_{r-1}+1,i_{r-1}}$ and $B_{i_{r-1},i_{r-1}+1}$ respectively. The map $Y_r:U_{i_{r-1}}\rightarrow U_{i_{r-1}}$  is thus $a'b'+\zeta_{i_{r-1}+1}{\rm Id}$. 

The sequence $(w'_1-v'_1,v'_1-v'_2,v'_2-v'_3,\dots, v'_{d'-1})$ is the partition $\mu_r'$. Now  apply Proposition \ref{NC} to $\left(\Gamma',\v',\w',\xihat'\right)$. Then we see that $a'b'$ belongs to the Zariski closure of nilpotent orbit $\overline{C}_r$ proving thus that $Y_r\in \zeta_{i_{r-1}+1}{\rm Id}+\overline{C}_r$.
\end{proof}

By Proposition \ref{square} and Proposition \ref{proj} we have

\begin{corollary} The image of the composition  $\mathfrak{M}_{\xihat_{L,P,\Sigma},\thetahat}(\v_{L_P,\Sigma},\w)\overset{\pi}{\rightarrow}\mathfrak{M}_{\xihat_{L,P,\Sigma}}(\v_{L,P,\Sigma},\w)\overset{\rho}{\rightarrow}\gl_n$ is $\overline{\calO}$. Moreover if $J_\thetahat=J_\xihat$, then $\pi\circ\rho$ is a bijective morphism onto its image (if $\K=\C$, it is an isomorphism).
\label{piorho}\end{corollary}

\begin{remark}Assume that $\K=\C$. The condition in Remark \ref{monotone}
to have $\mathfrak{M}_{\xihat_{L,P,\Sigma}}(\v_{L,P,\Sigma},\w)\simeq\overline{\calO}$ may not be satisfied here. For instance in the example given by Shmelkin \cite[Example 4.3]{shmelkin} we have $\v_{L,P,\Sigma}=(4,1)$, $\w=(5,0)$, $\zetahat_{L,P,\Sigma}=(0,0)$, $\thetahat=(1,1)$, the adjoint orbit $\calO$ is the nilpotent orbit with partition $(3,1,1)$ while $\mathfrak{M}_{\xihat_{L,P,\Sigma}}(\v_{L,P,\Sigma},\w)$ is isomorphic to the Zariski closure of the nilpotent orbit with partition $(3,2)$.\label{shmelkin}\end{remark}

\subsubsection{Geometry of resolutions and parabolic induction}\label{georesol}

We review well-known results on the geometry of resolutions of Zariski  closure of adjoint orbits (Proposition \ref{semi-small1} and Proposition \ref{localres}). In the case where the adjoint orbit is regular nilpotent the results are contained in Borho-Macpherson's paper \cite{BM}. In order to clarify the picture we also find appropriate to review Lusztig's parabolic induction of perverse sheaves \cite{LuCS1}.

Let $L,P,\Sigma,\sigma,C,\calO$ be as in \S \ref{parresol} with $L=\GL_{s_{p+1}}\times\cdots\times \GL_{s_1}\subset \GL_n$. Recall also that $\mu_i$ is a partition of $s_i$ defined by the coordinate of $C$ in $\gl_{s_i}$. For each $i=1,\dots,p+1$, the dual partition $\mu_i'=(\mu'_{i,1},\dots,\mu'_{i,r_i})$ of $\mu_i$ defines a Levi subgroup $\hat{L}_i=\prod_j\GL_{\mu'_{i,j}}\subset\GL_{s_i}$. Let $\hat{P}_i$ be a parabolic subgroup of $\GL_{s_i}$ having $\hat{L}_i$ as a Levi subgroup and containing the upper triangular matrices. Then $\tilde{P}:=\prod_i\hat{P}_i$ is a parabolic subgroup of $L$ having $\hat{L}:=\prod_{i=1}^{p+1}\hat{L}_i$ as  a Levi factor. Put $\hat{P}:=\tilde{P}.U_P$. It is the unique parabolic subgroup of $\GL_n$ having $\hat{L}$ as  a Levi factor and contained in $P$.  

Consider the following maps 

\begin{equation}\xymatrix{\mathbb{X}_{\hat{L},\hat{P},\{\sigma\}}\ar[rr]^{\tilde{\pi}}&&\mathbb{X}_{L,P,\Sigma}\ar[rr]^{p}&&\overline{\calO}}\label{hatresol}\end{equation}where $\tilde{\pi}(X,g\hat{P})=(X,gP)$ and $p(X,gP)=X$.

Note that the variety $\mathbb{X}_{\hat{L},\hat{P},\{\sigma\}}$ is nonsingular and that $\tilde{\pi}$ is surjective. 

The decomposition $\overline{C}=\coprod_\alpha C_\alpha$ as a disjoint union of $L$-orbits provides a stratification $\overline{\Sigma}=\coprod_\alpha\Sigma_\alpha$ with $\Sigma_\alpha=\sigma+C_\alpha$ and therefore a stratification of $\mathbb{X}_{L,P,\Sigma}=\coprod_\alpha\mathbb{X}_{L,P,\Sigma_\alpha}^o$ where $$\mathbb{X}_{L,P,\Sigma_\alpha}^o:=\{(X,gP)\in \mathfrak{g}\times (\GL_n/P)|\, g^{-1}Xg\in \Sigma_\alpha+\mathfrak{u}_P\}$$is the smooth locus of $\mathbb{X}_{L,P,\Sigma_\alpha}$.

The following proposition is a particular case of a result of Lusztig \cite{LuIC} (cf. \cite[proof of Proposition 5.1.19]{letellier} for more details). 

\begin{proposition} For $x\in\overline{\calO}$, put $p^{-1}(x)_\alpha:=p^{-1}(x)\cap\mathbb{X}_{L,P,\Sigma_\alpha}^o$. Then 

$${\rm dim}\,\left\{\left.x\in\overline{\calO}\,\right|\,{\rm dim}\,p^{-1}(x)_\alpha\geq \frac{i}{2}-\frac{1}{2}({\rm dim}\, \Sigma-{\rm dim}\,\Sigma_\alpha)\right\}\leq {\rm dim}\, \overline{\calO}-i$$for all $i\in\N$.
\label{semi-smallgen}\end{proposition}

Hence the map $p$ satifies the condition of Proposition \ref{Lu1} and so $p_*\left(\pIC {\mathbb{X}_{L,P,\Sigma}}\right)$ is a perverse sheaf by Proposition \ref{Lu1}.  If we apply the proposition to $(\hat{L},\hat{P},\{\sigma\})$ instead of $(L,P,\Sigma)$ we find that $p\circ\tilde{\pi}$ is semi-small.

We now recall briefly Lusztig's parabolic induction of perverse sheaves \cite[\S 4]{LuCV}. It will help to clarify the picture and also some references to the literature in \S \ref{actionW}.

Put $V_1:=\{(X,g)\in\gl_n\times \GL_n\,|\, g^{-1}Xg\in\mathfrak{p}\}$ and $V_2:=\{(X,gP)\in\gl_n\times (\GL_n/P)\,|\, g^{-1}Xg\in\mathfrak{p}\}$ and consider the diagram

$$\xymatrix{\mathfrak{l}&V_1\ar[r]^{\rho'}\ar[l]_{\rho}&V_2\ar[r]^{\rho''}&\gl_n}$$where $\rho(X,g)=\pi_P(g^{-1}Xg)$ with $\pi_\mathfrak{p}:\mathfrak{p}=\mathfrak{l}\oplus\mathfrak{u}_P\rightarrow \mathfrak{l}$ the natural projection, $\rho'(X,g)=(X,gP)$, $\rho''(X,gP)=X$. The parabolic induction functor ${\rm Ind}_{\mathfrak{l}\subset\mathfrak{p}}^{\gl_n}$ is a functor from the category $\mathcal{M}_L(\mathfrak{l})$ of $L$-equivariant perverse sheaves on $\mathfrak{l}$ to $\mathcal{D}_c^b(\gl_n)$. Recall that a perverse sheaf $K$ on $\mathfrak{l}$ is said to be $L$-equivariant if  $(pr_2)^*K\simeq m^*K$ where $m:L\times\mathfrak{l}\rightarrow\mathfrak{l}, (l,X)\mapsto lXl^{-1}$ and $pr_2:L\times\mathfrak{l}\rightarrow\mathfrak{l}$ is the projection. The category $\mathcal{M}_L(\mathfrak{l})$ is then a full subcategory of $\mathcal{D}_c^b(\mathfrak{l})$ (see \cite[4.2]{letellier} for a detailed discussion on this). The morphism $\rho$ is $P$-equivariant if we let $P$ acts on $V_1$ as $g\cdot(X,h)=(X,hg^{-1})$ and on $\mathfrak{l}$ as $g\cdot X=\pi_P(g)X\pi_P(g)^{-1}$ where $\pi_P$ is the canonical projection $P=L\ltimes U_P\rightarrow L$. It is also a smooth morphism with connected fibers of dimension $m={\rm dim}\,\GL_n+{\rm dim}\, U_P$. Hence if $K\in\mathcal{M}_L(\mathfrak{l})$ then $\rho^*K[m]$ is a $P$-equivariant perverse sheaf on $V_1$. Since $\rho'$ is a locally trivial (for Zariski topology) principal $P$-bundle, the functor $(\rho')^*[{\rm dim}\, P]$ induces an equivalence of categories from the category of perverse sheaves on $V_2$ to the category of $P$-equivariant perverse sheaves on $V_1$. Hence for any $K\in\mathcal{M}_L(\mathfrak{l})$, there exists a unique (up to isomorphism) perverse sheaf $\tilde{K}$ on $V_2$ such that 

$$\rho^*K[m]\simeq (\rho')^*\tilde{K}[{\rm dim} P].$$We define ${\rm ind}_{\mathfrak{l}\subset\mathfrak{p}}^{\gl_n}(K):=(\rho'')_*\tilde{K}$.

The following result is due to Lusztig \cite[\S 4]{LuCV}.

\begin{proposition}Let $Q=MU_Q$ be another Levi decomposition in $\GL_n$ with corresponding Lie algebra decomposition $\mathfrak{q}=\mathfrak{m}\oplus\mathfrak{u}_Q$. Assume that $L\subset M$ and $P\subset Q$. Let $K\in\mathcal{M}_L(\mathfrak{l})$ and assume that ${\rm Ind}_{\mathfrak{l}\subset \mathfrak{p}\cap\mathfrak{m}}^\mathfrak{m}(K)$ is a perverse sheaf (it is then automatically $M$-equivariant). Then 

$${\rm Ind}_{\mathfrak{l}\subset\mathfrak{p}}^{\gl_n}(K)\simeq {\rm Ind}_{\mathfrak{m}\subset\mathfrak{q}}^{\gl_n}\left({\rm Ind}_{\mathfrak{l}\subset \mathfrak{p}\cap\mathfrak{m}}^\mathfrak{m}(K)\right).$$
\label{transitivity}\end{proposition}

The following result is easy to prove from the following cartesian diagram:

$$\xymatrix{\mathfrak{l}&V_1\ar[l]_\rho\ar[r]^{\rho'}&V_2\ar[r]^{\rho''}&\mathfrak{g}\\
\overline{\Sigma}\ar[u]&\mathbb{Y}_{L,P,\Sigma}\ar[u]\ar[l]_{b_1}\ar[r]^{b_2}&\mathbb{X}_{L,P,\Sigma}\ar[u]\ar[r]^p&\overline{\calO}\ar[u]}$$where $\mathbb{Y}_{L,P,\Sigma}:=\{(X,g)\in\gl_n\times\GL_n\,|\,g^{-1}Xg\in\overline{\Sigma}+\mathfrak{u}_P\}$, and where the vertical arrows are inclusions and $b_1,b_2,p$ are the restrictions of $\rho,\rho',\rho''$.

\begin{lemma} The $\GL_n$-equivariant perverse sheaf $p_*\left(\pIC {\mathbb{X}_{L,P,\Sigma}}\right)$ is isomorphic to ${\rm Ind}_{\mathfrak{l}\subset\mathfrak{p}}^{\gl_n}\left(\pIC {\overline{\Sigma}}\right)$. Similarly the $\GL_n$-equivariant perverse sheaf $(p\tilde{\pi})_*\big(\underline{\kappa}\big)$ is isomorphic to ${\rm Ind}_{\hat{\mathfrak{l}}\subset\hat{\mathfrak{p}}}^{\gl_n}\big(\underline{\kappa}_\sigma\big)$ where $\underline{\kappa}_\sigma$ is the constant sheaf on $\{\sigma\}$ extended by zero on $\hat{\mathfrak{l}}-\{\sigma\}$.
 \label{ind}\end{lemma}

Define $\mathbb{X}_{\hat{L},\tilde{P},\{\sigma\}}:=\{(X,g\tilde{P})\in\mathfrak{l}\times (L/\tilde{P})\,|\,g^{-1}Xg\in\sigma+\mathfrak{u}_{\tilde{P}}\}$ and let $\mathbb{Y}$ be the variety $\{(y,z,g)\in P\times\gl_n\times\GL_n\,|\, g^{-1}zg\in\sigma+\mathfrak{u}_{\hat{P}}\}$ modulo the action of $\hat{P}$ given by $p\cdot (y,z,g):=(yp^{-1},z,gp^{-1})$.

Consider the following Cartesian diagram (see Borho and MacPherson \cite[\S 2.10]{BM} in the case where $\calO$ is regular nilpotent).

\begin{equation}\xymatrix{\mathbb{X}_{\hat{L},\tilde{P},\{\sigma\}}\ar[d]^r&\mathbb{Y}\ar[d]^c\ar[r]^{a_2}\ar[l]_{a_1}&\mathbb{X}_{\hat{L},\hat{P},\sigma}\ar[d]^{\tilde{\pi}}\\
\overline{\Sigma}&\mathbb{Y}_{L,P,\Sigma}\ar[r]^{b_2}\ar[l]_{b_1}&\mathbb{X}_{L,P,\Sigma}\ar[d]^p\\
&&\overline{\calO}}\label{trans}\end{equation}where $a_1(y,z,g)=\left(\pi_\mathfrak{p}(yg^{-1}zgy^{-1}),\pi_P(y)\tilde{P}\right)$, $a_2(y,z,g)=(z,g\hat{P})$, $c(y,z,g)=(z,gy^{-1})$, $r(X,g\tilde{P})=X$ where $\pi_P:L\ltimes U_P\rightarrow L$ is the canonical projection.

We now use this diagram to prove the following proposition.

\begin{proposition} The morphism $\tilde{\pi}$ is  semi-small with respect to $\mathbb{X}_{L,P,\Sigma}=\coprod_\alpha\mathbb{X}_{L,P,\Sigma_\alpha}^o$.
\label{semi-small1}\end{proposition}

\begin{proof} By Proposition \ref{semi-smallgen} applied to $(\hat{L},\tilde{P},\{\sigma\})$ instead of $(L,P,\Sigma)$ we find that  $r:\mathbb{X}_{\hat{L},\tilde{P},\{\sigma\}}\rightarrow\overline{\Sigma}$ is semi-small with respect to the stratification $\overline{\Sigma}=\coprod_\alpha\Sigma_\alpha$. On the other hand we see from the identity (\ref{dimfor}) that \begin{equation}{\rm codim}_{\overline{\Sigma}}(\Sigma_\alpha)={\rm codim}_{\mathbb{Y}_{L,P,\Sigma}}\mathbb{Y}_{L,P,\Sigma_\alpha}={\rm codim}_{\mathbb{X}_{L,P,\Sigma}}\mathbb{X}_{L,P,\Sigma_\alpha}.\label{double=}\end{equation}From the first equality and Lemma \ref{easyfact} we deduce that $c$ is semi-small with respect to $\mathbb{Y}_{L,P,\Sigma}=\coprod_\alpha\mathbb{Y}_{L,P,\Sigma_\alpha}$. Then applying Lemma \ref{easyfact} to the right square of the diagram (\ref{trans}) we deduce the proposition.\end{proof}

\begin{proposition}The restriction of the sheaves $\mathcal{H}^i\left(\tilde{\pi}_*(\kappa)\right)$ to $\mathbb{X}_{L,P,\Sigma_\alpha}^o$ are locally constant for all $i$ and $\alpha$.
\label{localres}\end{proposition}

\begin{proof} From the above diagram (\ref{trans}) we see that \begin{equation}(b_2)^*\left(\tilde{\pi}_*(\underline{\kappa})\right)[{\rm dim}\, P]\simeq (b_1)^*(r_*(\underline{\kappa}))[m].\label{isofor}\end{equation}Since $b_2$ is a locally trivial principal $P$-bundle for the Zariski topology it is enough to prove that the restriction of $\mathcal{H}^i\left(r_*(\kappa)\right)$ to $\Sigma_\alpha$ is locally constant for all $i$ and $\alpha$. The map $r$ is semi-small and $L$-equivariant if we let $L$ acts on $\mathbb{X}_{\hat{L},\tilde{P},\{\sigma\}}$ by $v\cdot (X,m\tilde{P})=(vXv^{-1},vm\tilde{P})$. The complex  $r_*(\underline{\kappa})$ is thus a semisimple $L$-equivariant perverse sheaf. Since $\overline{\Sigma}$ has only a finite number of $L$-orbits, the simple constituents of $r_*(\underline{\kappa})$ are of the form $\pIC {\overline{\Sigma}_\alpha}$.
\end{proof}

\begin{remark} Diagrams similar to (\ref{trans}) are used by Lusztig to prove Proposition \ref{transitivity}. In our situation this works as follows. As in Lemma \ref{ind} we have $r_*(\underline{\kappa})={\rm Ind}_{\hat{\mathfrak{l}}\subset\tilde{\mathfrak{p}}}^\mathfrak{l}(\underline{\kappa}_\sigma)$. Hence it follows from the isomorphism (\ref{isofor}) that $${\rm Ind}_{\mathfrak{l}\subset\mathfrak{p}}^{\gl_n}\left({\rm Ind}_{\hat{\mathfrak{l}}\subset\tilde{\mathfrak{p}}}^\mathfrak{l}\big(\underline{\kappa}_\sigma\big)\right)\simeq{\rm Ind}_{\hat{\mathfrak{l}}\subset\hat{\mathfrak{p}}}^{\gl_n}\big(\underline{\kappa}_\sigma\big)$$which is a particular case of Proposition \ref{transitivity}.
\end{remark}

\section{Comet-shaped quiver varieties}

\subsection{Generic tuples of adjoint orbits}\label{generic}

Let $\calO_1,\dots,\calO_k$ be $k$-orbits of $\gl_n(\K)$ and let
$\tomega_i$ be the type of $\calO_i$, then
$\tomhat:=(\tomega_1,\dots,\tomega_k)$ is called the type of
$(\calO_1,\dots,\calO_k)$. 

\begin{definition} A $k$-tuple $(\calC_1,\dots,\calC_k)$ of semisimple adjoint orbits is said to be \emph{generic} if $\sum_{i=1}^k\Tr\,\calC_i=0$ and the following
holds. If $V\subset \K^n$ is a subspace stable by some
$X_i\in\calC_i$ for each $i$ such that
$$\sum_{i=1}^k\Tr\,(X_i|_V)=0,$$then either $V=0$ or $V=\K^n$.

Let $\calC_i$ be the adjoint orbit of the semisimple part of
an element of $\calO_i$. Then we say that $(\calO_1,\dots,\calO_k)$ is \emph{generic} if the tuple
$(\calC_1,\dots,\calC_k)$ of semisimple orbits is generic. 
\label{genorb}\end{definition}

We have  \cite[Lemma
2.2.2]{hausel-letellier-villegas}:

\begin{lemma} For $i=1,\dots,k$, put
$\tomega_i=\omega_i^1\omega_i^2\cdots\omega_i^{r_i}$ with
$\omega_i^j\in\calP^*$ such that $\sum_j|\omega_i^j|=n$. Put $D={\rm
min}_i{\rm max}_j|\omega_i^j|$ and let $d={\rm
gcd}\,\{|\omega_i^j|\}$. Assume that $${\rm char}(\K)\nmid D!$$If
$d>1$, generic $k$-tuples of adjoint orbits of $\gl_n$ of type
$(\tomega_1,\dots,\tomega_k)$ do not exist. If $d=1$, they
do.\label{existence}\end{lemma} 

\begin{remark} Our definition of generic tuple is equivalent to that given in Kostov \cite[\S 1.2]{Kostov} and in Crawley-Boevey \cite[\S 6]{crawley-mat}. Let us recall that definition as we will need it. To do that, for each $i=1,2,\dots,k$, we let $\alpha_{i,1},\alpha_{i,2},\dots,\alpha_{i,p_i}$ be the distinct eigenvalues of $\calO_i$ with respective multiplicities $m_{i,1},m_{i,2},\dots,m_{i,p_i}$. Then  $(\calO_1,\dots,\calO_k)$ is generic if we have $$\sum_{i=1}^k\sum_{j=1}^{p_i}m_{i,j}\alpha_{i,j}=0$$which corresponds to our condition $\sum_{i=1}^k{\rm Tr}\,(\calO_i)=0$, and if for any integers $0\leq m_{i,j}'\leq m_{i,j}$ such that $\sum_{j=1}^{p_i}m_{i,j}'$ does not depend on $i$ the equality $$\sum_{i=1}^k\sum_{j=1}^{p_i}m_{i,j}'\alpha_{i,j}=0$$holds if and only if $m_{i,j}'=m_{i,j}$ for all $i,j$ or $m_{i,j}'=0$ for all $i,j$.
\label{genericdef2}\end{remark}

\subsection{Affine comet-shaped quiver varieties}\label{star}

Let $(\calO_1,\dots,\calO_k)$ be a  $k$-tuple of
adjoint orbits of $\gl_n(\K)$ and let $g\geq 0$ be an integer. Put

$$\bfO:=(\gl_n)^{2g}\times\overline{\calO}_1\times\cdots\times\overline{\calO}_k,$$

$$\bfO^o:=(\gl_n)^{2g}\times\calO_1\times\cdots\times\calO_k.$$

Consider the affine variety

$$\calV_\bfO:=\left\{(A_1,B_1,\dots,A_g,B_g,X_1,\dots,X_k)\in\bfO\,\left|\, \sum_{i=1}^g[A_i,B_i]+\sum_{i=1}^kX_i=0\right.\right\},$$and let $\calV_\bfO^o$ denote the open subset $\calV_\bfO\cap\bfO^o$  of $\calV_\bfO$. 

We assume that $\sum_{i=1}^k{\rm Tr}\,(\calO_i)=0$ since otherwise $\calV_\bfO$ is clearly empty.

If $(\calO_1',\dots,\calO_k')$ is an other $k$-tuple of adjoint orbits of $\gl_n$, then we write $\bfO'\unlhd\bfO$ if for all $i=1,2,\dots,k$ we have $\calO_i'\subset\overline{\calO}_i$.  Note that is $(\calO_1,\dots,\calO_k)$ is generic and $\bfO'\unlhd\bfO$, then $(\calO_1',\dots,\calO_k')$ is also generic.

Note that we have the finite partition $$\calV_\bfO=\coprod_{\bfO'\unlhd\bfO}\calV_{\bfO'}^o.$$

Let $\PGL_n(\K)$ acts on $\calV_\bfO$ by simultaneoulsy
conjugating the $2g+k$ matrices and define $$\calQ_\bfO:=\calV_\bfO/\!/\PGL_n={\rm Spec}\left(\K[\calV_\bfO]^{\PGL_n}\right).$$We denote by $\calQ_\bfO^o$ the image of $\calV_\bfO^o$ in $\calQ_\bfO$. By Theorem \ref{goodquotient}(3) it an open subset of $\calQ_\bfO$.

\begin{definition} An element $(A_1,B_1,\dots,A_g,B_g,X_1,\dots,X_k)\in\calV_\bfO^o$ is said to be \emph{irreducible} if there is no non-zero proper subspace of $\K^n$ which is preserved by all  matrices $A_1,B_1,\dots,A_g,B_g,X_1,\dots,X_k$.
\end{definition}

When $g=0$, the problem of describing the $k$-tuples $(\calO_1,\dots,\calO_k)$ for which $\calV_\bfO^o$ admits irreducible elements is stated and studied by Kostov (see \cite{Kostov} for a survey) who calls it the (additive)  Deligne-Simpson problem.

In \cite{crawley-mat}, Crawley-Boevey reformulates this problem and Kostov's answer in terms of preprojective algebras and the moment map for representations of quivers. 

Let us now review Crawley-Boevey's work as we will need it later. More precisely we define a quiver $\Gamma_\bfO$ and parameters $\v_\bfO$, $\w$, $\xihat_\bfO$ such that there is a bijective morphism $\mathfrak{M}_{\xihat_\bfO}\left(\v_\bfO,\w\right)\rightarrow\calQ_\bfO$ which is an isomorphism when $\K=\C$.

 Consider the following quiver $\Gamma_\bfO$ \footnote{The picture is from
\cite{daisuke}.}with $g$ loops at the central vertex $0$ and with set of vertices $I=\{0\}\cup\{[i,j]\}_{1\leq i\leq k,1\leq j\leq s_i}$:

\vspace{10pt}

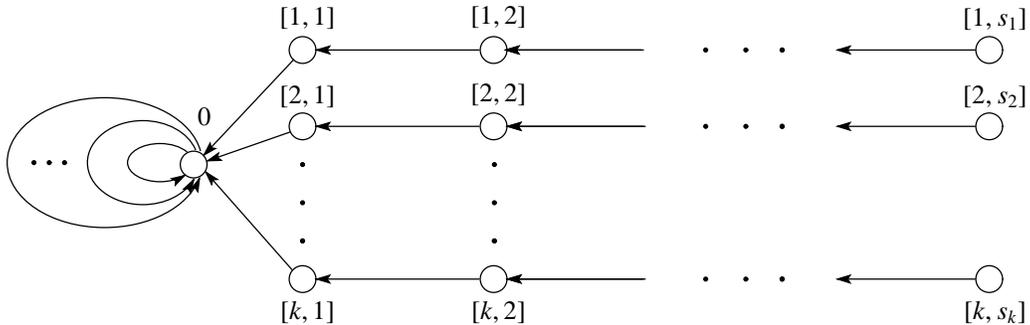
\begin{center}
\unitlength 0.1in
\begin{picture}( 52.1000, 15.4500)(  4.0000,-17.0000)
% CIRCLE 2 0 3 0
% 4 1375 1010 1305 1010 975 1010 975 1010
%
\special{pn 8}%
\special{ar 1376 1010 70 70  0.0000000 6.2831853}%
% CIRCLE 2 0 3 0
% 4 1945 410 1875 410 1545 410 1545 410
%
\special{pn 8}%
\special{ar 1946 410 70 70  0.0000000 6.2831853}%
% CIRCLE 2 0 3 0
% 4 2945 410 2875 410 2545 410 2545 410
%
\special{pn 8}%
\special{ar 2946 410 70 70  0.0000000 6.2831853}%
% CIRCLE 2 0 3 0
% 4 5540 410 5470 410 5140 410 5140 410
%
\special{pn 8}%
\special{ar 5540 410 70 70  0.0000000 6.2831853}%
% CIRCLE 2 0 3 0
% 4 1945 810 1875 810 1545 810 1545 810
%
\special{pn 8}%
\special{ar 1946 810 70 70  0.0000000 6.2831853}%
% CIRCLE 2 0 3 0
% 4 2945 810 2875 810 2545 810 2545 810
%
\special{pn 8}%
\special{ar 2946 810 70 70  0.0000000 6.2831853}%
% CIRCLE 2 0 3 0
% 4 5540 810 5470 810 5140 810 5140 810
%
\special{pn 8}%
\special{ar 5540 810 70 70  0.0000000 6.2831853}%
% CIRCLE 2 0 3 0
% 4 1945 1610 1875 1610 1545 1610 1545 1610
%
\special{pn 8}%
\special{ar 1946 1610 70 70  0.0000000 6.2831853}%
% CIRCLE 2 0 3 0
% 4 2945 1610 2875 1610 2545 1610 2545 1610
%
\special{pn 8}%
\special{ar 2946 1610 70 70  0.0000000 6.2831853}%
% CIRCLE 2 0 3 0
% 4 5540 1610 5470 1610 5140 1610 5140 1610
%
\special{pn 8}%
\special{ar 5540 1610 70 70  0.0000000 6.2831853}%
% VECTOR 2 0 3 0
% 2 1890 1560 1440 1050
%
\special{pn 8}%
\special{pa 1890 1560}%
\special{pa 1440 1050}%
\special{fp}%
\special{sh 1}%
\special{pa 1440 1050}%
\special{pa 1470 1114}%
\special{pa 1476 1090}%
\special{pa 1500 1088}%
\special{pa 1440 1050}%
\special{fp}%
% VECTOR 2 0 3 0
% 2 2870 410 2020 410
%
\special{pn 8}%
\special{pa 2870 410}%
\special{pa 2020 410}%
\special{fp}%
\special{sh 1}%
\special{pa 2020 410}%
\special{pa 2088 430}%
\special{pa 2074 410}%
\special{pa 2088 390}%
\special{pa 2020 410}%
\special{fp}%
% VECTOR 2 0 3 0
% 4 3720 410 3010 410 3730 410 3010 410
%
\special{pn 8}%
\special{pa 3720 410}%
\special{pa 3010 410}%
\special{fp}%
\special{sh 1}%
\special{pa 3010 410}%
\special{pa 3078 430}%
\special{pa 3064 410}%
\special{pa 3078 390}%
\special{pa 3010 410}%
\special{fp}%
\special{pa 3730 410}%
\special{pa 3010 410}%
\special{fp}%
\special{sh 1}%
\special{pa 3010 410}%
\special{pa 3078 430}%
\special{pa 3064 410}%
\special{pa 3078 390}%
\special{pa 3010 410}%
\special{fp}%
% VECTOR 2 0 3 0
% 2 2870 810 2020 810
%
\special{pn 8}%
\special{pa 2870 810}%
\special{pa 2020 810}%
\special{fp}%
\special{sh 1}%
\special{pa 2020 810}%
\special{pa 2088 830}%
\special{pa 2074 810}%
\special{pa 2088 790}%
\special{pa 2020 810}%
\special{fp}%
% VECTOR 2 0 3 0
% 2 2870 1610 2020 1610
%
\special{pn 8}%
\special{pa 2870 1610}%
\special{pa 2020 1610}%
\special{fp}%
\special{sh 1}%
\special{pa 2020 1610}%
\special{pa 2088 1630}%
\special{pa 2074 1610}%
\special{pa 2088 1590}%
\special{pa 2020 1610}%
\special{fp}%
% VECTOR 2 0 3 0
% 4 3730 810 3020 810 3740 810 3020 810
%
\special{pn 8}%
\special{pa 3730 810}%
\special{pa 3020 810}%
\special{fp}%
\special{sh 1}%
\special{pa 3020 810}%
\special{pa 3088 830}%
\special{pa 3074 810}%
\special{pa 3088 790}%
\special{pa 3020 810}%
\special{fp}%
\special{pa 3740 810}%
\special{pa 3020 810}%
\special{fp}%
\special{sh 1}%
\special{pa 3020 810}%
\special{pa 3088 830}%
\special{pa 3074 810}%
\special{pa 3088 790}%
\special{pa 3020 810}%
\special{fp}%
% VECTOR 2 0 3 0
% 4 3730 1610 3020 1610 3740 1610 3020 1610
%
\special{pn 8}%
\special{pa 3730 1610}%
\special{pa 3020 1610}%
\special{fp}%
\special{sh 1}%
\special{pa 3020 1610}%
\special{pa 3088 1630}%
\special{pa 3074 1610}%
\special{pa 3088 1590}%
\special{pa 3020 1610}%
\special{fp}%
\special{pa 3740 1610}%
\special{pa 3020 1610}%
\special{fp}%
\special{sh 1}%
\special{pa 3020 1610}%
\special{pa 3088 1630}%
\special{pa 3074 1610}%
\special{pa 3088 1590}%
\special{pa 3020 1610}%
\special{fp}%
% VECTOR 2 0 3 0
% 2 5465 410 4745 410
%
\special{pn 8}%
\special{pa 5466 410}%
\special{pa 4746 410}%
\special{fp}%
\special{sh 1}%
\special{pa 4746 410}%
\special{pa 4812 430}%
\special{pa 4798 410}%
\special{pa 4812 390}%
\special{pa 4746 410}%
\special{fp}%
% VECTOR 2 0 3 0
% 2 5465 810 4745 810
%
\special{pn 8}%
\special{pa 5466 810}%
\special{pa 4746 810}%
\special{fp}%
\special{sh 1}%
\special{pa 4746 810}%
\special{pa 4812 830}%
\special{pa 4798 810}%
\special{pa 4812 790}%
\special{pa 4746 810}%
\special{fp}%
% VECTOR 2 0 3 0
% 2 5465 1610 4745 1610
%
\special{pn 8}%
\special{pa 5466 1610}%
\special{pa 4746 1610}%
\special{fp}%
\special{sh 1}%
\special{pa 4746 1610}%
\special{pa 4812 1630}%
\special{pa 4798 1610}%
\special{pa 4812 1590}%
\special{pa 4746 1610}%
\special{fp}%
% VECTOR 2 0 3 0
% 2 1880 840 1450 990
%
\special{pn 8}%
\special{pa 1880 840}%
\special{pa 1450 990}%
\special{fp}%
\special{sh 1}%
\special{pa 1450 990}%
\special{pa 1520 988}%
\special{pa 1500 972}%
\special{pa 1506 950}%
\special{pa 1450 990}%
\special{fp}%
% VECTOR 2 0 3 0
% 2 1900 460 1430 960
%
\special{pn 8}%
\special{pa 1900 460}%
\special{pa 1430 960}%
\special{fp}%
\special{sh 1}%
\special{pa 1430 960}%
\special{pa 1490 926}%
\special{pa 1468 922}%
\special{pa 1462 898}%
\special{pa 1430 960}%
\special{fp}%
% DOT 2 0 3 0
% 4 1945 1010 1945 1210 1945 1410 1945 1410
%
\special{pn 8}%
\special{sh 1}%
\special{ar 1946 1010 10 10 0  6.28318530717959E+0000}%
\special{sh 1}%
\special{ar 1946 1210 10 10 0  6.28318530717959E+0000}%
\special{sh 1}%
\special{ar 1946 1410 10 10 0  6.28318530717959E+0000}%
\special{sh 1}%
\special{ar 1946 1410 10 10 0  6.28318530717959E+0000}%
% DOT 2 0 3 0
% 4 4055 410 4265 410 4455 410 4455 410
%
\special{pn 8}%
\special{sh 1}%
\special{ar 4056 410 10 10 0  6.28318530717959E+0000}%
\special{sh 1}%
\special{ar 4266 410 10 10 0  6.28318530717959E+0000}%
\special{sh 1}%
\special{ar 4456 410 10 10 0  6.28318530717959E+0000}%
\special{sh 1}%
\special{ar 4456 410 10 10 0  6.28318530717959E+0000}%
% DOT 2 0 3 0
% 4 4055 810 4265 810 4455 810 4455 810
%
\special{pn 8}%
\special{sh 1}%
\special{ar 4056 810 10 10 0  6.28318530717959E+0000}%
\special{sh 1}%
\special{ar 4266 810 10 10 0  6.28318530717959E+0000}%
\special{sh 1}%
\special{ar 4456 810 10 10 0  6.28318530717959E+0000}%
\special{sh 1}%
\special{ar 4456 810 10 10 0  6.28318530717959E+0000}%
% DOT 2 0 3 0
% 4 4055 1610 4265 1610 4455 1610 4455 1610
%
\special{pn 8}%
\special{sh 1}%
\special{ar 4056 1610 10 10 0  6.28318530717959E+0000}%
\special{sh 1}%
\special{ar 4266 1610 10 10 0  6.28318530717959E+0000}%
\special{sh 1}%
\special{ar 4456 1610 10 10 0  6.28318530717959E+0000}%
\special{sh 1}%
\special{ar 4456 1610 10 10 0  6.28318530717959E+0000}%
\put(19.7000,-2.4500){\makebox(0,0){$[1,1]$}}%
\put(29.7000,-2.4000){\makebox(0,0){$[1,2]$}}%
\put(55.7000,-2.5000){\makebox(0,0){$[1,s_1]$}}%
\put(19.7000,-6.5500){\makebox(0,0){$[2,1]$}}%
\put(29.7000,-6.4500){\makebox(0,0){$[2,2]$}}%
\put(55.7000,-6.5500){\makebox(0,0){$[2,s_2]$}}%
\put(19.7000,-17.8500){\makebox(0,0){$[k,1]$}}%
\put(29.7000,-17.8500){\makebox(0,0){$[k,2]$}}%
\put(55.7000,-17.8500){\makebox(0,0){$[k,s_k]$}}%
\put(14.3000,-7.6000){\makebox(0,0){$0$}}%
\special{pn 8}%
\special{sh 1}%
\special{ar 2950 1010 10 10 0  6.28318530717959E+0000}%
\special{sh 1}%
\special{ar 2950 1210 10 10 0  6.28318530717959E+0000}%
\special{sh 1}%
\special{ar 2950 1410 10 10 0  6.28318530717959E+0000}%
\special{sh 1}%
\special{ar 2950 1410 10 10 0  6.28318530717959E+0000}%
\special{pn 8}%
\special{ar 1110 1000 290 220  0.4187469 5.9693013}%
\special{pn 8}%
\special{pa 1368 1102}%
\special{pa 1376 1090}%
\special{fp}%
\special{sh 1}%
\special{pa 1376 1090}%
\special{pa 1324 1138}%
\special{pa 1348 1136}%
\special{pa 1360 1158}%
\special{pa 1376 1090}%
\special{fp}%
\special{pn 8}%
\special{ar 910 1000 510 340  0.2464396 6.0978374}%
\special{pn 8}%
\special{pa 1400 1096}%
\special{pa 1406 1084}%
\special{fp}%
\special{sh 1}%
\special{pa 1406 1084}%
\special{pa 1362 1138}%
\special{pa 1384 1132}%
\special{pa 1398 1152}%
\special{pa 1406 1084}%
\special{fp}%
\special{pn 8}%
\special{sh 1}%
\special{ar 540 1000 10 10 0  6.28318530717959E+0000}%
\special{sh 1}%
\special{ar 620 1000 10 10 0  6.28318530717959E+0000}%
\special{sh 1}%
\special{ar 700 1000 10 10 0  6.28318530717959E+0000}%
\special{pn 8}%
\special{ar 1200 1000 170 100  0.7298997 5.6860086}%
\special{pn 8}%
\special{pa 1314 1076}%
\special{pa 1328 1068}%
\special{fp}%
\special{sh 1}%
\special{pa 1328 1068}%
\special{pa 1260 1084}%
\special{pa 1282 1094}%
\special{pa 1280 1118}%
\special{pa 1328 1068}%
\special{fp}%
\end{picture}%
\end{center}

\vspace{10pt}

The dimension vector $\mathbf{v}_\bfO$ of $\Gamma_\bfO$ with coordinate $v_i$ at $i\in I$ is defined as follows. We choose $k$ total orderings $\leq_i$ on $\calP$ and for each $i=1,2,\dots,k$, we define the  sequence $v_{[i,1]}> v_{[i,2]}>\cdots> v_{[i,s_i]}$ as the dimension vector $\v_{\calO_i}$ associated with the orbit $\calO_i$ with respect to $\leq_i$ as in \S\ref{type}. Note that the vector $\v_\bfO$ depends only on the type of the adjoint orbits $\calO_1,\dots,\calO_k$.

We also define $\xihat_\bfO\in\K^I$ as follows. For each $i$, let $\zetahat_{\calO_i}=(\zeta_{i,1},\dots,\zeta_{i,s_i+1})$ and $\xihat_{\calO_i}=(\xi_{[i,1]},\dots,\xi_{[i,s_i]})$ be the two sequences defined from  $\calO_i$ as in \S\ref{type}. We also put $\xi_0=-\sum_{i=1}^k\zeta_{i,1}$. This defines an element $\xihat_\bfO=\{\xi_0\}\cup\{\xi_{[i,j]}\}_{i,j}\in\K^I$ such that $\xihat_\bfO\centerdot\v_\bfO=0$. For a representation $\varphi$ of $\overline{\Gamma}_\bfO$, denote by $\varphi_{[i,1]}$ the linear map associated to the arrow whose tail is $[i,1]$, by $\varphi_1,\dots,\varphi_g$ the matrices associated to the loops in $\Omega$ and by $\varphi_1^*,\dots,\varphi_g^*$ the ones associated to the loops in $\overline{\Omega}-\Omega$. We have the following consequence of Proposition \ref{NC} (see Crawley-Boevey \cite{crawley-mat}\cite{crawleyred}).

\begin{proposition}  The map $\mu_{\v_\bfO}^{-1}(\xihat_\bfO)\rightarrow \calV_\bfO$ given by $\varphi\mapsto (A_1,B_1,\dots,A_g,B_g,X_1,\dots,X_k)$, with \begin{equation}A_i=\varphi_i,\, B_i=\varphi_i^*,\,X_i=\varphi_{[i,1]}\varphi_{[i,1]}^*+\zeta_{i,1}{\rm Id},\label{equationXi}\end{equation} is well-defined and maps simple representations onto the subset $\big(\calV_\bfO^o\big)^{\rm irr}$ of irreducible elements. This map induces a bijective morphism $$\mathfrak{M}_{\xihat_\bfO}(\v_\bfO)\longrightarrow\calQ_\bfO$$which maps $\mathfrak{M}_{\xihat_\bfO}^s(\v_\bfO)$ onto  $\big(\calQ_\bfO^o\big)^{\rm irr}$. If $\K=\C$, this bijective map is an isomorphism.\label{crawley}\end{proposition}

The above proposition together with Theorem \ref{non-emptycond} implies a criterion in terms of roots for the non-emptyness of $\big(\calV_\bfO^o\big)^{\rm irr}$ solving thus the additive Deligne-Simpson problem.

From Proposition \ref{crawley} and Theorem \ref{irrquiver} we have the following result:

\begin{corollary} If $\big(\calV_\bfO^o\big)^{\rm irr}\neq\emptyset$ then both $\calV_\bfO$ and $\calQ_\bfO$ are irreducible respectively of dimension ${\rm dim}\,\bfO-n^2+1$ and \begin{equation}d_\bfO=2-{^t}\v_\bfO C_{\Gamma_\bfO}\v_\bfO={\rm dim}\,\bfO-2n^2+2\label{dimeq}\end{equation}where $C_{\Gamma_\bfO}$ is the Cartan matrix of $\Gamma_\bfO$.
\label{lissite} \end{corollary}

We now state a result in the generic case. The proof is omitted as it is an easy generalisation of the  case of semisimple orbits \cite[Proposition 2.2.3]{hausel-letellier-villegas}.

\begin{proposition} Assume that $(\calO_1,\dots,\calO_k)$ is generic. Then  $\big(\calV_\bfO^o\big)^{\rm irr}=\calV_\bfO^o$ and the map $\calV_\bfO\rightarrow\calQ_\bfO$ is a principal $\PGL_n$-bundle for the \'etale topology (and so it is a geometric quotient). In particular the $\PGL_n$-orbits of $\calV_\bfO$ are all closed of same dimension ${\rm dim}\,\PGL_n$. Finally the two varieties $\calV_\bfO^o$ and $\calQ_\bfO^o$ are nonsingular.
\label{affinepGb}\end{proposition}

The following result is a consequence of Proposition \ref{affinepGb} and Corollary \ref{lissite}.

\begin{corollary}Assume that $(\calO_1,\dots,\calO_k)$ is generic. Then the partitions 

\begin{equation}\calV_\bfO=\coprod_{\bfO'\unlhd\bfO}\calV_{\bfO'}^o,\,\,\text{ and }\,\,\calQ_\bfO=\coprod_{\bfO'\unlhd\bfO}\calQ_{\bfO'}^o\label{partition}\end{equation}are stratifications.

\end{corollary}

Crawley-Boevey's criterion for the non-emptyness of $\calV_\bfO$ and $\calV_\bfO^o$ simplifies in the generic case as follows.

\begin{theorem} Assume that $(\calO_1,\dots,\calO_k)$ is a generic tuple. Then the following three assertions are equivalent.

(i) The set $\calV_\bfO$ is not empty.

(ii) The set $\calV_\bfO^o$ is not empty.

(iii) $\v_\bfO\in\Phi^+(\Gamma_\bfO)$. 
\label{nonemptymu}\end{theorem}

Although this theorem is not stated in Crawley-Boevey's papers, the main ingredients for its proof are there. For the convenience of the reader we give the proof in details (repeating if necessary some arguments of Crawley-Boevey).

We start with an intermediate result.

Following Crawley-Boevey's terminology \cite{crawleyind}, we say that a dimension vector $\betahat=\{\beta_i\}_{i\in I}$ of $\Gamma_\bfO$ with $\beta_0=n$ is \emph{strict} if for any $i=1,2,\dots,k$ we have $n\geq \beta_{[i,1]}\geq\cdots\geq\beta_{[i,s_i]}$.

We have the following proposition.

\begin{proposition} Assume that $\calV_\bfO$ is not empty. Then the dimension vector $\v_\bfO$ is a sum $\betahat^1+\betahat^2+\cdots+\betahat^r$ of strict positive roots such that $\xihat_\bfO\centerdot\betahat^i=0$ for all $i=1,2,\dots,r$. If moreover $(\calO_1,\dots,\calO_k)$ is generic, then $r=1$, i.e., $\v_\bfO$ is a positive root. \label{propCB}\end{proposition}

\begin{proof} By Theorem \ref{CB5} and Remark \ref{remCB}, we can choose an element $B\in\mu_{\v_\bfO}^{-1}(\xihat_\bfO)$ whose coordinates $B_h$, where $h$ describes the set of arrows of $\Gamma_\bfO$ which are not loops, are injective. Let $\pi$ be the canonical projection ${\bf M}\left(\overline{\Gamma}_\bfO,\v_\bfO\right)\rightarrow {\bf M}\left(\Gamma_\bfO,\v_\bfO\right)$. Write $\pi(B)$ as a direct sum $I_1\oplus I_2\oplus\cdots\oplus I_r$ of indecomposable representations of $\Gamma_\bfO$ and let $\betahat^m$ be the dimension vector of $I_m$. We have $\v_\bfO=\betahat^1+\cdots+\betahat^r$ and since the maps $B_h$ are injective, the maps $(I_m)_h$ are also injective and so $\betahat^m$ is a strict dimension vector for all $m=1,\dots,r$. It is a well-known theorem of Kac \cite{kac} that the dimension vector of an indecomposable representation is a positive root. Hence   the $\betahat^1,\dots,\betahat^r$ are positive strict roots. It remains to see that $\betahat^m\centerdot \xihat_\bfO=0$ for all $m=1,\dots,r$.  But  $\betahat^m$ is the dimension vector of  a direct summand of a representation of $\Gamma_\bfO$ that lifts to a representation of $\mu_{\v_\bfO}^{-1}(\xihat_\bfO)$, hence by Crawley-Boevey's theorem \cite[Theorem 3.3]{crawley-quiver} we must have $\betahat^m\centerdot\xihat_\bfO=0$.

Assume now that $(\calO_1,\dots,\calO_k)$ is generic. To prove that $r=1$ we repeat Crawley-Boevey's argument in \cite[\S 3]{crawley-mat}. For each $i=1,2,\dots,k$, we let $\alpha_{i,1},\alpha_{i,2},\dots,\alpha_{i,p_i}$ be the distinct eigenvalues of $\calO_i$ with respective multiplicities $m_{i,1},m_{i,2},\dots,m_{i,p_i}$. Let $s\in\{1,\dots,r\}$. For given $1\leq i\leq k, 1\leq f\leq p_i$, define $$m_{i,f}^s=\sum_{\overset{j=1}{\zeta_{i,j}=\alpha_{i,f}}}^{s_i+1}\left(\beta_{[i,j-1]}^s-\beta_{[i,j]}^s\right)$$where for convenience $\beta_{[i,s_i+1]}^s=0$ and $[i,0]$ denotes also the vertex $0$. Since $\betahat^s$ is strict, the integer $m_{i,f}^s$ is positive. Moreover \begin{equation}\sum_{f=1}^{p_i}m_{i,f}^s=\beta_0^s\label{CB6}\end{equation} is independent of $i$. Now

\begin{equation}\sum_{s=1}^rm_{i,f}^s=\sum_{\overset{j=1}{\zeta_{i,j}=\alpha_{i,f}}}^{s_i+1}\left(v_{[i,j-1]}-v_{[i,j]}\right)=m_{i,f}\label{CB7}\end{equation}where $v_{[i,s_i+1]}=0$. Hence $0\leq m_{i,f}^s\leq m_{i,f}$ and \begin{align*}0=\xihat_\bfO\centerdot\betahat^s&=\left(\sum_{i=1}^k\sum_{j=1}^{s_i}\left(\zeta_{i,j}-\zeta_{i,j+1}\right)\beta_{[i,j]}^s\right)-\left(\sum_{i=1}^k\zeta_{i,1}\right)\beta_0^s\\
&=-\sum_{i=1}^k\sum_{j=1}^{s_i+1}\zeta_{i,j}\left(\beta_{[i,j-1]}^s-\beta_{[i,j]}^s\right)\\
&=-\sum_{i=1}^k\sum_{f=1}^{p_i}\alpha_{i,f}m_{i,f}^s\end{align*}which contradicts the genericity condition (see Remark \ref{genericdef2}) unless $m_{i,f}^s=m_{i,f}$ for all $i,f$, or $m_{i,f}^s=0$ for all $i,f$. But since $\betahat^s$ is a strict root we must have $\beta_0^s>0$ and so  by (\ref{CB6}) we can not have $m_{i,f}^s=0$ for all $i,f$. Hence we must have $m_{i,f}^s=m_{i,f}$ for all $i,f$ and so from the identity (\ref{CB7}) we must have $r=1$.\end{proof}

\begin{proof}[Proof of Theorem \ref{nonemptymu}] (ii) implies (i) is trivial and by Proposition \ref{propCB} (i) implies (iii). Hence it remains to see that (iii) implies (ii). But this is exactly what is proved in Crawley-Boevey \cite[\S 6]{crawley-mat}.\end{proof}

For each $i\in I-\{0\}$, we let $s_i:\Z^I\rightarrow\Z^I$ be the reflection defined by $$s_i(x)=x-(x,{\bf e}_i){\bf e}_i,$$where $(\,,\,)$ is the form defined by $({\bf e}_i,{\bf e}_j)=c_{ij}$, cf. \S \ref{uquiver}. 

For $\u\in\Z^I$ and $i=1,2,\dots,k$ denote by $\u^\sharp_i$ the unordered collection of numbers

$$n-u_{[i,1]},\, u_{[i,1]}-u_{[i,2]},\dots,\,u_{[i,s_i-1]}-u_{[i,s_i]},\,u_{[i,s_i]}.$$Since the action of the reflexion $s_{[i,j]}$ has the effect of exchanging the $j$th and $(j+1)$th  terms in this collection we have the following lemma.

\begin{lemma} If $\u,\v\in\Z^I$ satisfies $(\v)_i^\sharp=(\u)_i^\sharp$ for all $i=1,\dots,k$, then there exists an element $w$ in the subgroup of the Weyl group of $\Gamma$ generated by the reflexions $s_{[i,j]}$ such that $\u=w(\v)$.

\label{Ws}\end{lemma}

\begin{proposition} If $g\geq 1$, then $\v_\bfO$ is always an imaginary root.\label{imaginary}\end{proposition}

\begin{proof} Since $\v_\bfO$ is a decreasing dimension vector, for each $i=1,2,\dots,k$,  it defines a unique partition $\mu^i=(\mu^i_1,\dots,\mu^i_{r_i})$ of $n$ whose parts  are of the form  $v_{[i,j]}-v_{[i,j+1]}$, $j=0,\dots,s_i$ (with the convention that $v_{[i,0]}=n$ and $v_{[i,s_i+1]}=0$). Define a dimension vector ${\bf f}$ of $\Gamma_\bfO$ with the requirement that $f_0=n$ and $f_{[i,j]}=n-\sum_{r=1}^j\mu^i_r$. Note that ${\bf f}=\v_\bfO$ if and only if $v_{[i,j]}-v_{[i,j+1]}\geq v_{[i,j+1]}-v_{[i,j+2]}$ for all $i,j$. We have $({\bf e}_0,{\bf f})=(2-2g)n-\sum_{i=1}^kf_{[i,1]}\leq 0$, and $({\bf e}_{[i,j]},{\bf f})=\mu^i_{j+1}-\mu^i_j\leq 0$. Hence ${\bf f}$ is in the fundamental set of imaginary roots by definition (see Kac \cite[Chapter 1]{kac}). By Lemma \ref{Ws}, the vector ${\bf f}$ can be obtained from $\v_\bfO$ by an element in the Weyl group of $\Gamma_\bfO$, we conclude that $\v_\bfO$ is always an imaginary root of $\Gamma_\bfO$.
\end{proof}

Theorem \ref{nonemptymu} and Proposition \ref{imaginary} have the following consequence.

\begin{corollary} If $(\calO_1,\dots,\calO_k)$ is generic and $g\geq 1$, then $\calV_\bfO^o$ is not empty.
\end{corollary}

The following proposition is due to Crawley-Boevey \cite{crawley-mat}.

\begin{proposition}If $(\calO_1,\dots,\calO_k)$ is generic and $g=0$, then $\v_\bfO$ is a real root if and only if $\calV_\bfO^o$ consists of a single $\PGL_n$-orbit (in which case $\calV_\bfO^o=\calV_\bfO$).\label{real}\end{proposition}

\begin{example} Here we assume that $g=0$, $k=3=n$. Let  $\calO$ be the regular nilpotent orbits of $\gl_3$ and let $\mathcal{S}$ be the regular semisimple adjoint orbit with eigenvalues $1,2,-3$. The tuple  $(\calO_1,\calO_2,\calO_3)=(\calO,\calO,\mathcal{S})$ is then generic, the underlying graph of the associated quiver $\Gamma_\bfO$ is $\tilde{E}_6$ and $\v_\bfO$ is the indivisible positive imaginary root. Hence $\calV_\bfO$ is not empty by Theorem \ref{nonemptymu}. Moreover we can use again Theorem \ref{nonemptymu} to verify that the only non-empty strata of $\calV_\bfO$ are $\calV_{\bfO}^o$ and  the two strata $\calV_{\bfO_1}^o$ and $\calV_{\bfO_2}^o$ corresponding respectively to $(\calO,\calC,\mathcal{S})$ and $(\calC,\calO,\mathcal{S})$ where $\calC$ is the nilpotent subregular adjoint orbit.  Note that $\v_{\bfO_i}$, $i=1,2$, is the real root $\alpha_1+\alpha_2+2\alpha_3+3\alpha_4+2 \alpha_5+\alpha_6$ of $E_6$ (in the notation of \cite[PLANCHE V]{bourbaki}) and so $\calV_{\bfO_i}$ is a single $\PGL_n$-orbit by Proposition \ref{real}.\end{example}

\begin{remark} If $\calV_{\bfO'}$ is not empty then for any $\bfO$ such that $\bfO'\unlhd\bfO$ the variety $\calV_\bfO$ will be also not empty. We may use this together with the equivalence between the two assertions (i) and (iii) of Theorem \ref{nonemptymu} to construct new roots of quivers from known ones. 
\end{remark}

\subsection{General comet-shaped quiver varieties}\label{intervar}

Let $(\calO_1,\dots,\calO_k)$ be a tuple of adjoint orbits of $\gl_n(\K)$, and for each $i=1,\dots,k$, let $(L_i,P_i,\sigma_i,C_i)$ be as in \S \ref{parresol} such that the image of the first projection $p_i:\mathbb{X}_{L_i,P_i,\Sigma_i}\rightarrow\gl_n$ is $\overline{\calO}_i$ where $\Sigma_i=\sigma_i+C_i$. As in the introduction we put ${\bf P}=P_1\times\cdots\times P_k$, ${\bf L}=L_1\times\cdots\times L_k$ and ${\bf \Sigma}=\Sigma_1\times\cdots\times\Sigma_k$, ${\bf C}:=C_1\times\cdots\times C_k$. Put $\mathbb{O}_{\bf L,P,\Sigma}=(\gl_n)^{2g}\times\mathbb{X}_{\bf L,P,\Sigma}$, $\mathbb{O}_{\bf L,P,\Sigma}^o=(\gl_n)^{2g}\times\mathbb{X}_{\bf L,P,\Sigma}^o$ and

$$\mathbb{V}_{\bf L,P,\Sigma}:=\left\{\left(\left.A_1,B_1,\dots,A_g,B_g,(X_1,\dots,X_k,g_1P_1,\dots,g_kP_k)\right)\in\mathbb{O}_{\bf L,P,\Sigma}\,\right|\, \sum_j[A_j,B_j]+\sum_iX_i=0\right\}.$$

Let ${\rm p}=(id)^{2g}\times p_1\times\cdots\times p_k:\mathbb{O}_{\bf L,P,\Sigma}\rightarrow \bfO$ and let $\rho:\mathbb{V}_{\bf L,P,\Sigma}\rightarrow\calV_\bfO$ be its restriction. The map $\rho$ is clearly projective. Let $\GL_n$ act on $\mathbb{V}_{\bf L,P,\Sigma}$ diagonally by conjugation on the first $2g+k$ coordinates and by left multiplication on the last $k$ coordinates. These actions of $\GL_n$ on $\mathbb{V}_{\bf L,P,\Sigma}$ and $\calV_\bfO$ induces actions of $\PGL_n$ for which the morphism $\rho$ is $\PGL_n$-equivariant.

\begin{proposition} Assume that the tuple $(\calO_1,\dots,\calO_k)$ is generic. Then the geometric quotient $\mathbb{V}_{\bf L,P,\Sigma}\rightarrow\mathbb{Q}_{\bf L,P,\Sigma}$ exists and is a principal $\PGL_n$-bundle. Moreover the diagram 

$$\xymatrix{\mathbb{V}_{\bf L,P,\Sigma}\ar[rr]^\rho\ar[d]&&\calV_\bfO\ar[d]\\\mathbb{Q}_{\bf L,P,\Sigma}\ar[rr]^{\rho/_{\PGL_n}}&&\calQ_\bfO}$$is Cartesian. If $\K=\overline{\F}_q$ and if our data $({\bf L,P,\Sigma})$ is defined over $\F_q$ then the above diagram is also defined over $\F_q$.
\label{quotientgen}\end{proposition}

\begin{proof} Since the tuple $(\calO_1,\dots,\calO_k)$ is generic, the quotient $\calV_\bfO\rightarrow\calQ_\bfO$ is a principal $\PGL_n$-bundle in the \'etale topology (see Proposition \ref{affinepGb}) and so the result follows from Corollary \ref{goodquotient1}.\end{proof}

In general (i.e. when the tuple $(\calO_1,\dots,\calO_k)$ is not necessarily generic) we can always define the GIT quotient $$\mathbb{V}_{\bf L,P,\Sigma}/\!/_\Psi\GL_n$$with respect to some $\GL_n$-linearization $\Psi$ of some ample line bundle $M$ on $\mathbb{V}_{\bf L,P,\Sigma}$. Indeed $\mathbb{V}_{\bf L,P,\Sigma}$ is projective over $\calV_\bfO$ and so such a pair $(M,\Psi)$ always exists (see above Corollary \ref{goodquotient1}).

Assuming that $(\calO_1,\dots,\calO_k)$ is generic, we show in this section that  the quotient $\mathbb{Q}_{\bf L,P,\Sigma}/\!/_\Psi\GL_n$ can be identified (at least when $\K=\C$) with a quiver variety $\mathfrak{M}_{\xihat,\thetahat}(\v)$ for appropriate choices of $\xihat,\thetahat,\v$.

For each $i=1,\dots,k$, we can define a type $A$ quiver $\Gamma_{L_i,P_i,\Sigma_i}$ together with parameters $\xihat_{L_i,P_i,\Sigma_i},\thetahat_i,\v_{L_i,P_i,\Sigma_i}$ as in \S \ref{parresol} such that  there is a canonical bijective morphism $\mathbb{X}_{L_i,P_i,\Sigma_i}\rightarrow \mathfrak{M}_{\xihat_{L_i,P_i,\Sigma_i},\thetahat_i}(\v_{L_i,P_i,\Sigma_i},\w)$ which is an isomorphism when $\K=\C$.

We now define a comet shaped quiver $\Gamma_{\bf L,P,\Sigma}$ as in \S \ref{star} such that each leg with vertices $[1,1],\dots,[1,s_i]$ is exactly the quiver $\Gamma_{L_i,P_i,\Sigma_i}$. I.e., if we delete the central vertex $\{0\}$ from $\Gamma_{\bf L,P,\Sigma}$, we recover the $k$ type $A$ quivers $\Gamma_{L_1,P_1,\Sigma_1},\dots,\Gamma_{L_k,P_k,\Sigma_k}$. We denote by $I$ the set of vertices of $\Gamma_{\bf L,P,\Sigma}$, and we define a dimension vector $\v_{\bf L,P,\Sigma}=\{v_i\}_{i\in I}$ by putting $v_0:=n$ and, for each $i=1,\dots,k$,  $(v_{[i,1]},\dots,v_{[i,s_i]}):=\v_{L_i,P_i,\Sigma_i}$. Multiplying the vectors $\thetahat_i$ by a strictly positive integer if necessary, there is   $\thetahat\in \Z^I$ such that its projection on $\Gamma_{L_i,P_i,\Sigma_i}$ is $\thetahat_i$ for each $i$ and such that $\thetahat\centerdot\v_{\bf L,P,\Sigma}=0$. There is a unique $\xihat_{\bf L,P,\Sigma}\in\K^I$ whose projection on $\Gamma_{L_i,P_i,\Sigma_i}$ is $\xihat_{L_i,P_i,\Sigma_i}$ for all $i$ and $\xihat_{\bf L,P,\Sigma}\centerdot\v_{\bf L,P,\Sigma}=0$. Note that $\theta_0$ must be negative.

The quiver $\Gamma_{\bf L,P,\Sigma}$ and the parameter $\xihat_{\bf L,P,\Sigma}$ are the same as $\Gamma_\bfO$ and $\xihat_\bfO$ obtained from $(\calO_1,\dots,\calO_k)$, see above Example \ref{diffvect}. However in general the dimension vector $\v_{\bf L,P,\Sigma}$ differs from $\v_\bfO$ as shown in Example \ref{diffvect}. 

To alleviate the notation we will use $\Gamma$, $\xihat$, $\v$ instead of $\Gamma_{\bf L,P,\Sigma}$, $\xihat_{\bf L,P,\Sigma}$ and $\v_{\bf L,P,\Sigma}$.

Let $\Gamma^{\dagger}$ be the quiver obtained from $\Gamma$ by deleting the central vertex (i.e., it is the union of  the quivers $\Gamma_{L_1,P_1,\Sigma_1},\dots,\Gamma_{L_k,P_k,\Sigma_k}$). We denote by $I^{\dagger}=\{[i,j]\}_{i,j}$ the set of vertices of $\Gamma^{\dagger}$. For a parameter $x\in\K^I$, we denote by $x^{\dagger}$ its projection on $\K^{I^{\dagger}}$.

We put $${\bf Z}\left(\overline{\Gamma}{^\dagger},\mathbf{v}^\dagger,\mathbf{w}\right):=(\gl_n)^{2g}\times\mathbf{M}\left(\overline{\Gamma}{^{\dagger}},\mathbf{v}^\dagger,\mathbf{w}\right).$$

We let $\GL_{\v^\dagger}$ acts on ${\bf Z}\left(\overline{\Gamma}{^{\dagger}},\mathbf{v}^\dagger,\mathbf{w}\right)$  by the trivial action on $(\gl_n)^{2g}$ and by the usual action on the second coordinate.

We identify in the obvious way $\mathbf{M}\left(\overline{\Gamma},\v\right)$ with $\mathbf{Z}\left(\overline{\Gamma}{^\dagger},\v^\dagger,\w\right)$ and we regard $\mu_\v^{-1}(\xihat)$ as a $\GL_{\v^\dagger}$-stable closed subvariety of $(\gl_n)^{2g}\times\mu_{\v^\dagger,\w}^{-1}(\xihat^{\dagger})$. To avoid any confusion, for a closed $\GL_\v$-stable subset $X$  of $\mathbf{M}\left(\overline{\Gamma},\v\right)=\mathbf{Z}\left(\overline{\Gamma}{^\dagger},\v^\dagger,\w\right)$ we denote by $X^{ss}(\Phi)$ the set of $\thetahat$-semistable points of $X$ and by $X^{ss}(\Phi^\dagger)$ the set of $\thetahat^\dagger$-semistable points. Clearly $X^{ss}(\Phi)\subset X^{ss}(\Phi^\dagger)$.

%Put ${\bf Z}_{\thetahat^{\dagger}}^{ss}\left(\overline{\Gamma}{^{\dagger}},\mathbf{v}^\dagger,\mathbf{w}\right):=(\gl_n)^{2g}\times\mathbf{M}_{\thetahat^\dagger}^{ss}\left(\overline{\Gamma}{^\dagger},\mathbf{v}^\dagger,\mathbf{w}\right)$

Define 

$$\mathfrak{Z}_{\xihat^\dagger,\thetahat^\dagger}\left(\v^\dagger,\w\right):=\left((\gl_n)^{2g}\times\mu_{\v^\dagger,\w}^{-1}(\xihat^\dagger)\right)\left/\left/_{\thetahat^\dagger}\GL_{\v^\dagger}\right.\right.\simeq (\gl_n)^{2g}\times\mathfrak{M}_{\xihat^\dagger,\thetahat^\dagger}(\v^\dagger,\w).$$There is a canonical bijective map $f_1:\mathfrak{Z}_{\xihat^\dagger,\thetahat^\dagger}(\v^\dagger,\w)\longrightarrow\mathbb{O}_{\bf L,P,\Sigma}$ (which is an isomorphism when $\K=\C$). Let $q:\left((\gl_n)^{2g}\times\mu_{\v^\dagger,\w}^{-1}(\xihat^\dagger)\right)^{ss}(\Phi^\dagger)\rightarrow \mathfrak{Z}_{\xihat^\dagger,\thetahat^\dagger}\left(\v^\dagger,\w\right)$ denote the quotient map. By Proposition \ref{affinegoodquotient} the map $f_1$ restricts to a bijective morphism $q\left(\mu_\v^{-1}(\xihat)^{ss}(\Phi^\dagger)\right)\rightarrow \mathbb{V}_{\bf L,P,\Sigma}$ and there is a canonical bijective map  $\mu_\v^{-1}(\xihat)/\!/_{\thetahat^\dagger}\GL_{\v^\dagger}\rightarrow q\left(\mu_\v^{-1}(\xihat)^{ss}(\Phi^\dagger)\right)$. Composing the two bijective morphisms we end up with a bijective morphism $f_2:\mu_\v^{-1}(\xihat)/\!/_{\thetahat^\dagger}\GL_{\v^\dagger}\rightarrow\mathbb{V}_{\bf L,P,\Sigma}$ which is an isomorphism when $\K=\C$. 

\begin{proposition} Assume that tuple $(\calO_1,\dots,\calO_k)$ is generic, then an element of $\mu_\v^{-1}(\xihat)$ is $\thetahat$-semistable (resp. $\thetahat$-stable) if and only if it is $\thetahat^{\dagger}$-semistable (resp. $\thetahat^{\dagger}$-stable).
\label{thetastability}\end{proposition}

\begin{proof} Assume that $\varphi\in \mu_\v^{-1}(\xihat)$ is $\thetahat^{\dagger}$-semistable. Let $\psi$ be a subrepresentation of $\varphi$. It is an element in $\mu_{\v'}^{-1}(\xihat)$ for some $\v'\leq\v$. We need to verify that $\thetahat\centerdot\v'\leq 0$. If $v'_0=v_0$, then we must have $\thetahat\centerdot\v'\leq\thetahat\centerdot\v= 0$ since $\thetahat^{\dagger}\in\N^{I^{\dagger}}$. If $v'_0=0$, then the subspaces $V_{[i,1]}'$ are contained in ${\rm Ker}\,(a_i)$ for all $i=1,\dots,k$ and so  $\thetahat\centerdot\v'=\thetahat^{\dagger}\centerdot(\v')^{\dagger}\leq 0$ since $\varphi$ is $\thetahat^{\dagger}$-semistable. Let $(A_1^\varphi,B_1^\varphi,\dots,A_g^\varphi,B_g^\varphi,X_1^\varphi,\dots,X_k^\varphi)$ be given by Formula (\ref{equationXi}).  Since $\psi$ is a subrepresentation of $\varphi$, the subspace $V'_0\subset V_0=\K^{v_0}$ is preserved by all matrices $A_1^\varphi,B_1^\varphi,\dots,A_g^\varphi,B_g^\varphi,X_1^\varphi,\dots,X_k^\varphi$. Recall also that any tuple $(\calO'_1,\dots,\calO'_k)\unlhd(\calO_1,\dots,\calO_k)$ is  generic. Hence by Proposition \ref{affinepGb}, the tuple $(A_1^\varphi,B_1^\varphi,\dots,A_g^\varphi,B_g^\varphi,X_1^\varphi,\dots,X_k^\varphi)$, which belongs to some $(\calO'_1,\dots,\calO'_k)\unlhd(\calO_1,\dots,\calO_k)$, is irreducible. Hence $v'_0=0$ or $v'_0=n$.\end{proof}

\begin{proposition}  Assume that  $(\calO_1,\dots,\calO_k)$ is generic. Then the morphism $f_2$ induces a bijective morphism  $\mathfrak{M}_{\xihat,\thetahat}(\v)\rightarrow\mathbb{Q}_{\bf L,P,\Sigma}$ (which is an isomorphism when $\K=\C$).
\end{proposition}

\begin{proof}The proposition follows from  Proposition \ref{thetastability} and Proposition \ref{quotientransi} applied to $X=\mu_\v^{-1}(\xihat)$, $G''=\GL_\v=\GL_n\times\GL_{\v^\dagger}$.
\end{proof}

\begin{remark} If $(\calO_1,\dots,\calO_k)$ is not generic, a-priori we only have a bijective morphism $\mathfrak{M}_{\xihat,\thetahat}(\v)$ onto an open subset  of a quotient $\mathbb{V}_{\bf L,P,\Sigma}/\!/_\Psi\GL_n$.\end{remark}

{\bf We now assume until the end of this section that the tuple $(\calO_1,\dots,\calO_k)$ is generic}. 

Thanks to Proposition \ref{thetastability} we can now omitt $\Phi$ and $\Phi^\dagger$ from the notation $\mu_\v^{-1}(\xihat)^{ss}(\Phi)$ or $\mu_\v^{-1}(\xihat)^{ss}(\Phi^\dagger)$ and write simply $\mu_\v^{-1}(\xihat)^{ss}$.

\begin{remark} Assume that the $\thetahat_i$'s, $i=1,\dots,k$, have striclty positive coordinates. Then $\mu_\v^{-1}(\xihat)^{ss}=\mu_\v^{-1}(\xihat)^s$. This identity also happens when $\thetahat$ is generic. We want to notice that in this situation we can actually choose our $\thetahat_i$'s (taking larger values of the coordinates if necessary) such that $\thetahat$ is generic. Indeed the set $\mu_\v^{-1}(\xihat)^{ss}$ depends only on the position of the non-zero coordinates of the $\thetahat_i$'s and not on their values (cf. Remark \ref{zeros} (ii)).
\label{assgeneric}\end{remark}

Put $$\mathfrak{N}_{\xihat,\thetahat^{\dagger}}(\v^{\dagger},\w):=\mu_\v^{-1}(\xihat)/\!/_{\thetahat^{\dagger}}\GL_{\v^{\dagger}}.$$

We summarize what we said with the following commutative diagram

\begin{equation}\xymatrix{\mathfrak{Z}_{\xihat^{\dagger},\thetahat^{\dagger}}(\v^{\dagger},\w)\ar[rr]^{f_1}&&\mathbb{O}_{\bf L,P,\Sigma}\ar[rr]^{\rm p}&&\bfO\\
\mathfrak{N}_{\xihat,\thetahat^{\dagger}}(\v^{\dagger},\w)\ar[rr]^{f_2}\ar[u]\ar[d]&&\mathbb{V}_{\bf L,P,\Sigma}\ar[rr]^{\rho}\ar[d]\ar[u]&&\calV_\bfO\ar[u]\ar[d]\\
\mathfrak{M}_{\xihat,\thetahat}(\v)\ar[rr]^{f_3}&&\mathbb{Q}_{\bf L,P,\Sigma}\ar[rr]^{\rho/_{\PGL_n}}&&\calQ_\bfO}.\label{mainpicture}\end{equation}where $\mathbb{Q}_{\bf L,P,\Sigma}$ is defined as in Proposition \ref{quotientgen} and where $f_3$ is the factorization morphism (as $q\circ\pi_2\circ f_2$ is constant on $\GL_n$-orbits).  The top vertical arrows are the canonical inclusions and the bottom vertical arrows are the canonical quotient maps. 

\begin{remark}When $\K=\C$, the maps $f_1,f_2,f_3$ are isomorphisms and the diagram is Cartesian.\label{cartesianpic}\end{remark}

Recall that $\Sigma_i=\sigma_i+C_i$. Put ${\bf C}=C_1\times\cdots\times C_k$. Then the decomposition of $\overline{\bf C}=\coprod_\alpha {\bf C}_\alpha$ as a union of ${\bf L}$-orbits provides a stratification $\overline{\bf \Sigma}=\coprod_\alpha {\bf \Sigma}_{\alpha}$. We thus a have a decomposition \begin{equation}\mathbb{V}_{\bf L,P,\Sigma}=\coprod_\alpha \mathbb{V}_{\bf L,P,\Sigma_\alpha}^o\label{partition1}\end{equation}where $\mathbb{V}_{\bf L,P,\Sigma_\alpha}^o:=\mathbb{V}_{\bf L,P,\Sigma_\alpha}\cap \mathbb{O}_{\bf L,P,\Sigma_\alpha}^o$. By Proposition \ref{square}, the subset $\mathbb{V}_{\bf L,P,\Sigma}^o\subset \mathbb{V}_{\bf L,P,\Sigma}$ corresponds to the stable points, i.e., $\mathbb{V}_{\bf L,P,\Sigma}^o\simeq \mathfrak{N}_{\xihat,\thetahat^{\dagger}}^s(\v^{\dagger},\w)=\mu_\v^{-1}(\xihat)^s/\GL_{\v^{\dagger}}$. The image of $\mathbb{V}_{\bf L,P,\Sigma_\alpha}$ by the projective morphism $$\rho:\mathbb{V}_{\bf L,P,\Sigma}\longrightarrow\calV_\bfO$$ is of the form $\calV_{\bfO_\alpha}$ for some $\bfO_\alpha\unlhd\bfO$.

\begin{theorem} The variety $\mathbb{V}_{\bf L,P,\Sigma_\alpha}$ is not empty if and only if $\v_{\bf L,P,\Sigma_\alpha}$ is a root of $\Gamma_{\bf L,P,\Sigma_\alpha}$. In this case the piece $\mathbb{V}_{\bf L,P,\Sigma_\alpha}^o$ is also not empty and is an irreducible nonsingular dense open subset of $\mathbb{V}_{\bf L,P,\Sigma_\alpha}$ of dimension

$$(2g+k-1)n^2+1-{\rm dim}\,{\bf L}+{\rm dim}\,{\bf \Sigma}_\alpha.$$In particular the partition (\ref{partition1}) is a stratification. \label{strat}\end{theorem}

Since $\mathbb{V}_{\bf L,P,\Sigma}\rightarrow\mathbb{Q}_{\bf L,P,\Sigma}$ is a principal $\PGL_n$-bundle we have the following result.

\begin{corollary} The stratum $\mathbb{Q}_{\bf L,P,\Sigma_\alpha}^o$ is irreducible and the decomposition $$\mathbb{Q}_{\bf L,P,\Sigma}=\coprod_\alpha \mathbb{Q}_{\bf L,P,\Sigma_\alpha}^o$$is a stratification.
 \label{strat2}\end{corollary}

Recall that $\v_\bfO$ be the dimension vector of $\Gamma$ obtained from the tuple $(\calO_1,\dots,\calO_k)$ as in \S \ref{star}. Let $W(\Gamma^\dagger)$ denote the Weyl group of $\Gamma^\dagger$.

\begin{lemma} The two vectors $\v$ and $\v_\bfO$ are in the same $W(\Gamma^\dagger)$-orbit.
\label{interlemma} \end{lemma}

\begin{proof} It follows from Lemma \ref{Ws} as for each $i=1,\dots,k$, we have $(\v_\bfO)^\sharp_i=\v^\sharp_i$.

\end{proof}

\begin{proof}[Proof of Theorem \ref{strat}] We prove it for ${\bf \Sigma}={\bf \Sigma}_\alpha$ as the proof will be the same for any ${\bf \Sigma}_\alpha$. Note that $\mathbb{V}_{\bf L,P,\Sigma}$ is not empty if and only if $\calV_\bfO$ is not empty. Hence the first assertion follows from Lemma \ref{interlemma} and Proposition \ref{nonemptymu}.

Assume that  $\mathbb{V}_{\bf L,P,\Sigma}$ is not empty. Then $\calV_\bfO$ is not empty and so by Proposition \ref{nonemptymu} the set $\calV_\bfO^o$ is also not empty. Since the inverse image of $\calV_\bfO^o$ by the map $\rho:\mathbb{V}_{\bf L,P,\Sigma}\rightarrow\calV_\bfO$ is contained in  $\mathbb{V}_{\bf L,P,\Sigma}^o$, the open subset  $\mathbb{V}_{\bf L,P,\Sigma}^o$  of $\mathbb{V}_{\bf L,P,\Sigma}$ is not empty.

Consider $\mathbb{Y}_{L,P,\Sigma}^o:=\{(X,g)\in \gl_n\times\GL_n\,|\, g^{-1}Xg\in\Sigma+\mathfrak{u}_P\}$. Then the canonical map $\mathbb{Y}_{L,P,\Sigma}^o\rightarrow \mathbb{X}_{L,P,\Sigma}^o$, $(X,g)\mapsto (X,gP)$ is a locally trivial principal $P$-bundle (for the Zariski topology). Note that $\mathbb{Y}_{L,P,\Sigma}^o\simeq G\times (\Sigma+\mathfrak{u}_P)$. Now consider the set $
\mathbb{L}_{\bf L,P,\Sigma}^o$ of $(2g+k)$-tuples $\left(A_1,B_1,\dots,A_g,B_g,(g_1,y_1),\dots,(g_k,y_k)\right)$ in 
$(\gl_n)^{2g}\times\left(G\times(\Sigma_1+\mathfrak{u}_P)\right)\times\cdots\times\left(G\times(\Sigma_k+\mathfrak{u}_P)\right)$ such that $$\sum_j[A_j,B_j]+\sum_ig_iy_ig_i^{-1}=0.$$The natural map $\mathbb{L}_{\bf L,P,\Sigma}^o\rightarrow\mathbb{V}^o_{\bf L,P,\Sigma}$ is then a locally trivial principal ${\bf P}$-bundle. Hence we are reduced to prove that $\mathbb{L}_{\bf L,P,\Sigma}^o$ is nonsingular. A sufficient condition for a point ${\bf x}\in \mathbb{L}_{\bf L,P,\Sigma}^o$ to be nonsingular is that the differential $d_{\bf x}\mu$ of the map 

$$\mu:(\gl_n)^{2g}\times\left(G\times(\Sigma_1+\mathfrak{u}_P)\right)\times\cdots\times\left(G\times(\Sigma_k+\mathfrak{u}_P)\right)\longrightarrow \mathfrak{sl}_n$$given by $\left(A_1,B_1,\dots,A_g,B_g,(g_1,\sigma_1),\dots,(g_k,\sigma_k)\right)\mapsto \sum_j[A_j,B_j]+\sum_ig_iy_ig_i^{-1}$ is surjective.

Let $y_i$ be the coordinate of ${\bf x}$ in $\Sigma_i+\mathfrak{u}_P$. Consider the restriction $\lambda$ of $\mu$  to the closed subset $(\gl_n)^{2g}\times\left(G\times\{y_1\}\right)\times\cdots\times\left(G\times\{y_k\}\right)$. It is enough to prove that the differential $d_{\bf x}\lambda$ is surjective. But this what we prove to see that the variety $\calV_{\bf S}^o$ is nonsingular (${\bf S}$ being $(\gl_n)^{2g}\times \overline{S}_1\times\cdots\times \overline{S}_k$ where $S_i\subset\overline{\calO}_i$ is the adjoint orbit of $y_k$), see Theorem \ref{affinepGb} and references therein. The variety $\mathbb{V}^o_{\bf L,P,\Sigma}$ is thus nonsingular and its irreducible components are all of same dimension. To compute the dimension of $\mathbb{V}^o_{\bf L,P,\Sigma}$ we may use what we just said or use the fact that there is a bijective morphism $\mathfrak{N}_{\xihat,\thetahat^{\dagger}}^s(\v^{\dagger},\w)\rightarrow\mathbb{V}^o_{\bf L,P,\Sigma}$ and then use Theorem \ref{irrquiver}.

Let us see now that $\mathbb{V}_{\bf L,P,\Sigma}$ is irreducible. Let ${\bf \hat{L},\hat{P}}$ be defined as in \S \ref{georesol} and put $\sigma:=(\sigma_1,\dots,\sigma_k)$. The canonical map $\mathbb{V}_{\bf \hat{L},\hat{P},\{\bf \sigma\}}\rightarrow \mathbb{V}_{\bf L,P,\Sigma}$ defined by $(X,g{\bf \hat{P}})\mapsto (X,g{\bf P})$  being surjective it is enough to show that $\mathbb{V}_{\bf \hat{L},\hat{P},\{\bf \sigma\}}$ is irreducible. We are thus reduced to prove the irreducibility of $\mathbb{V}_{\bf L,P,\Sigma}$ when $\Sigma$ is reduced to a point $\{\bf \sigma\}$ which we now assume. Hence $\mathbb{V}_{\bf L,P,\{\sigma\}}=\mathbb{V}_{\bf L,P,\{\sigma\}}^o$ and the parameter $\thetahat$ satisfies $\theta_i>0$ for all $i\in I^{\dagger}$. By Remark \ref{assgeneric}, we may assume that $\thetahat$ is generic with respect to $\v$. We now need to prove the irreducibility of $\mathfrak{N}_{\xihat,\thetahat^{\dagger}}(\v,\w)$. Since $\mathfrak{N}_{\xihat,\thetahat^{\dagger}}(\v,\w)\rightarrow \mathfrak{M}_{\xihat,\thetahat}(\v)$ is a principal $\PGL_n$-bundle, we are reduced to prove that $\mathfrak{M}_{\xihat,\thetahat}(\v)$ is irreducible.

Assume first that $\K=\C$. Then by Theorem \ref{HLRpure} we have  $H_c^i\left(\mathfrak{M}_{\xihat,\thetahat}(\v),\C\right)\simeq H_c^i\left(\mathfrak{M}_{\thetahat,\thetahat}(\v),\C\right)$. Recall that the dimension of $H_c^{2e}(X,\C)$ where $e$ is the dimension of $X$ equals the number of irreducible components of $X$ of dimension $e$. The varieties $\mathfrak{M}_{\xihat,\thetahat}(\v)$ and $\mathfrak{M}_{\thetahat,\thetahat}(\v)$ are both of pure dimension by Theorem \ref{irrquiver}. Hence  we are reduced to see that $\mathfrak{M}_{\thetahat,\thetahat}(\v)$ is irreducible. The representations in $\mu_\v^{-1}(\thetahat)$ are all simple because $\thetahat$ is generic, hence $\mathfrak{M}_{\thetahat}(\v)$ is irreducible and nonsingular (see Theorem \ref{irrquiver}). The canonical map $\mathfrak{M}_{\thetahat,\thetahat}(\v)\rightarrow \mathfrak{M}_{\thetahat}(\v)$ being a resolution of singularities is thus an isomorphism and so $\mathfrak{M}_{\thetahat,\thetahat}(\v)$ is irreducible.

Assume that $\K=\overline{\F}_q$. By Theorem \ref{Nak} there exists $r_0$ such that for all $r\geq r_0$ we have $\sharp\{\mathfrak{M}_{\xihat,\thetahat}(\v)(\F_{q^r})\}=\sharp\{\mathfrak{M}_{\thetahat,\thetahat}(\v)(\F_{q^r})\}$. As the canonical map $\mathfrak{M}_{\thetahat,\thetahat}(\v)\rightarrow \mathfrak{M}_{\thetahat}(\v)$ is an isomorphism we actually have \begin{equation}\sharp\{\mathfrak{M}_{\xihat,\thetahat}(\v)(\F_{q^r})\}=\sharp\{\mathfrak{M}_\thetahat(\v)(\F_{q^r})\}.\label{eqcount}\end{equation}Note that the dimension of the  compactly supported $\ell$-adic cohomology group $H_c^{2e}(X,\kappa)$ with $\ell$ invertible in $\K$ and $e={\rm dim}\, X$ also equals the number $m$ of irreducible components of $X$ of dimension $e$. Moreover if $X$ is defined over $\F_q$, then the Frobenius $F^*$ acts on $H_c^{2e}(X,\kappa)$ as multiplication by $q^e$. Therefore, the coefficient of $q^e$ in $\sharp\{X(\F_q)\}$ equals $m$. From the identity (\ref{eqcount}) we deduce that $\mathfrak{M}_{\xihat,\thetahat}(\v)$ is irreducible if and only if $\mathfrak{M}_\thetahat(\v)$ is irreducible. But as above the variety $\mathfrak{M}_\thetahat(\v)$ is irreducible as $\thetahat$ is generic.\end{proof}

\subsection{A restriction property}\label{restrictionprop}

We keep the notation  of \S \ref{intervar} and we assume that  $(\calO_1,\dots,\calO_k)$ is generic and that   $\calV_\bfO$ is not empty. Note that $\mathbb{V}_{\bf L,P,\Sigma}^o$ is then also not empty by Theorem \ref{strat}.

The aim of this section is to prove the following theorem.

\begin{theorem}Let $i$ be the natural inclusion $\mathbb{V}_{\bf L,P,\Sigma}\hookrightarrow \mathbb{O}_{\bf L,P,\Sigma}$. Then $$i^*\left(\IC {\mathbb{O}_{\bf L,P,\Sigma}}\right)=\IC {\mathbb{V}_{\bf L,P,\Sigma}}.$$
 
\label{restriction}\end{theorem}

By \S \ref{intervar}, we have a stratification 

$$\mathbb{O}_{\bf L,P,\Sigma}=\coprod_\alpha\mathbb{O}_{\bf L,P,\Sigma_\alpha}^o$$with $\mathbb{O}_{\bf L,P,\Sigma_\alpha}^o:=(\gl_n)^{2g}\times\mathbb{X}_{\bf L,P,\Sigma_\alpha}^o$. It satisfies the conditions (i) of Proposition \ref{proprest}.

We consider the semi-small resolution $\tilde{\pi}:\mathbb{O}_{\bf \hat{L},\hat{P},\{\sigma\}}\rightarrow \mathbb{O}_{\bf L,P,\Sigma}$ considered in \S \ref{georesol} and its restriction $\tilde{\rho}:\mathbb{V}_{\bf \hat{L},\hat{P},\{\sigma\}}\rightarrow \mathbb{V}_{\bf L,P,\Sigma}$.

\begin{proposition} The morphism $\tilde{\rho}$ is a semi-small resolution. Moreover the diagram 
$$\xymatrix{\mathbb{O}_{\bf \hat{L},\hat{P},\{\sigma\}}\ar[rr]^-{\tilde{\pi}}&&\mathbb{O}_{\bf L,P,\Sigma}\\\mathbb{V}_{\bf \hat{L},\hat{P},\{\sigma\}}\ar[u]\ar^{\tilde{\rho}}[rr]&&\ar[u]\mathbb{V}_{\bf L,P,\Sigma}}$$is Cartesian (the vertical arrows being the canonical inclusions) and the restriction of the sheaf $\mathcal{H}^i(\tilde{\pi}_*(\kappa))$ to each piece $\mathbb{O}_{\bf L,P,\Sigma_\alpha}^o$ is a locally constant sheaf.
\label{cartesian} \end{proposition}

\begin{proof} The diagram is Cartesian by definition of the varieties $\mathbb{V}_{\bf L,P,\Sigma}$. The variety $\mathbb{V}_{\bf \hat{L},\hat{P},\{\sigma\}}$ is also nonsingular by Theorem \ref{strat}. Hence $\tilde{\rho}$ is a resolution of singularities.

By Proposition \ref{semi-small1} the map $\tilde{\pi}$ is semi-small with respect to  $\mathbb{O}_{\bf L,P,\Sigma}=\coprod_\alpha\mathbb{O}_{\bf L,P,\Sigma_\alpha}^o$. By Theorem \ref{strat} we see that the codimension of $\mathbb{V}_{\bf L,P,\Sigma_\alpha}^o$ in $\mathbb{V}_{\bf L,P,\Sigma}$ equals the codimension of $\mathbb{O}_{\bf L,P,\Sigma_\alpha}^o$ in $\mathbb{O}_{\bf L,P,\Sigma}$, hence $\tilde{\rho}$ is also semi-small. The last assertion of the proposition follows from Proposition \ref{localres}.
\end{proof}

Theorem \ref{restriction} is now a consequence of Proposition \ref{cartesian} and Proposition \ref{proprest}. 

We have the following particular case of Theorem \ref{restriction}.

\begin{proposition} Let $i$ denotes the inclusion $\calV_\bfO\hookrightarrow\bfO$. Then $i^*\left(\IC {\bfO}\right)\simeq\IC {\calV_\bfO}$.
\label{restrictionVw}\end{proposition}

\section{Characters and Fourier transforms}

Here $\K$ is an algebraic closure of a finite field $\F_q$. In this section we put $G:=\GL_n(\K)$ and $\mathfrak{g}:=\gl_n(\K)$. We denote by $F$ the standard Frobenius endomorphism $\mathfrak{g}\rightarrow\mathfrak{g}$ that maps a matrix $(a_{ij})_{i,j}$ to $(a_{ij}^q)_{i,j}$ so that $G^F=\GL_n(\F_q)$ and $\mathfrak{g}^F=\gl_n(\F_q)$. 

\subsection{Preliminaries on finite groups}\label{prefini}

Let $\kappa$ be an algbraically closed field of characteristic $0$. Let $z\mapsto \overline{z}$ be an involution of $\kappa$ that maps roots of unity to their inverses. For a finite set $E$, we define $\langle\,,\,\rangle_E$ on the space of all functions $E\rightarrow\kappa$ by $$\langle f,g\rangle_E=\frac{1}{|E|}\sum_{x\in E}f(x)\overline{g(x)}.$$

Now let $H$ be a subgroup of a finite group $K$ and let  $\tilde{H}$ be a subgroup of $N_K(H)$ containing $H$. Let $\rho^1:\tilde{H}\rightarrow \GL(V^1)$ and $\rho^2:\tilde{H}\rightarrow \GL(V^2)$ be two representations of $\tilde{H}$ in the finite dimensional $\kappa$-vector spaces $V^1, V^2$. We denote by $\chi^1$ and $\chi^2$ their associated characters. The group $\tilde{H}$ acts on the space ${\rm Hom}\,(V^1,V^2)$ as follows. For $f\in {\rm Hom}\,(V^1,V^2)$, we define $r\cdot f:V^1\rightarrow V^2$ by $(r\cdot f)(v)=r\cdot f(r^{-1}\cdot v)$. Moreover we see that the subspace ${\rm Hom}_H(V^1,V^2)$ of fixed points of ${\rm Hom}\,(V^1,V^2)$ by $H$  is clearly $\tilde{H}$-stable (it is therefore an $\kappa[\tilde{H}/H]$-module).

 For any $r\in \tilde{H}$, we have 

\begin{equation}{\rm Tr}\,\left(r\,|\, {\rm Hom}\,(V^1,V^2)\right)=\chi^1(r)\,\chi^2(r^{-1}).\label{hom}\end{equation}

For $s\in \tilde{H} $, we denote by $\chi^i_s$ the restriction of $\chi^i$ to the coset $Hs:=\{hs\,|\, h\in H\}$.

\begin{proposition} Let $s\in \tilde{H}$. We have $${\rm Tr}\,\left(s\,\left|\,{\rm Hom}_H(V^1,V ^2)\right.\right)=\left\langle \chi^1_s,\chi^2_s\right\rangle_{Hs}.$$\label{tranches}\end{proposition}
 
\begin{proof} Put $E:={\rm Hom}\,(V^1,V^2)$ and $E_H:={\rm Hom}_H(V^1,V^2)$ and denote $p:E\rightarrow E_H$ the map $p(x)=\frac{1}{|H|}\sum_{h\in H}h\cdot x$. Then $E':={\rm Ker}\,p$ is an $\tilde{H}$-stable subspace of $E$ and $E=E_H\oplus E'$. Since $$\left.\left(\frac{1}{|H|}\sum_{h\in H}hs\right)\right|_{E'}=0$$we deduce that $${\rm Tr}\,(s\,|\,E_H)=\frac{1}{|H|}\sum_{h\in H}{\rm Tr}\,(hs\,|\, E).$$By Formula (\ref{hom}), the right hand side of this equation is  $\left\langle \chi^1_s,\chi^2_s\right\rangle_{Hs}$.
\end{proof}

We now let $\varphi$ and $\psi$ be the characters of $\tilde{H}$ and $K$ associated respectively to representations $\tilde{H}\rightarrow \GL(V)$ and $K\rightarrow\GL(W)$. The group $\tilde{H}$ acts on the $K$-module ${\rm Ind}_H^K(V):=\kappa[K]\otimes_{\kappa[H]}V$ by $t\cdot(x\otimes v)=xt^{-1}\otimes t\cdot v$. Its restriction to $H$ being trivial, it factorizes through an action of $\tilde{H}/H$ on ${\rm Ind}_H^K(V)$. Under the natural isomorphism (Frobenius reciprocity) 

\begin{equation}{\rm Hom}_H(V,W)\simeq{\rm Hom}_K\left({\rm Ind}_H^K(V),W\right)\label{Frobeniusrecip}\end{equation}the action of $\tilde{H}/H$ on ${\rm Hom}_H(V,W)$ described earlier corresponds to the action of $\tilde{H}/H$ on the $\kappa$-vector space ${\rm Hom}_K\left({\rm Ind}_H^K(V),W\right)$ given by $(t\cdot f)(x\otimes v)=f(t^{-1}\cdot (x\otimes v))$. For a subset $E$ of $K$ and a function $f:E\rightarrow \kappa$, we define ${\rm Ind}_E^K(f):K\rightarrow \kappa$ by 

$${\rm Ind}_E^K(f)(k)=\frac{1}{|E|}\sum_{\{g\in K\,|\,g^{-1}kg\in E\}}f(g^{-1}kg).$$

Then we have the following generalization of Frobenius reciprocity for functions:

\begin{lemma} Let $h:K\rightarrow\kappa$ be a function. Then $$\left\langle {\rm Ind}_E^K(f),h\right\rangle_K=\left\langle f,{\rm Res}^K_E(h)\right\rangle_E.$$
\end{lemma}

\begin{proof} It follows from a straightforward calculation.\end{proof}

By Proposition \ref{tranches}, (\ref{Frobeniusrecip}) and the above lemma, we have the following proposition:

\begin{proposition} Let $v\in \tilde{H}/H$ and let $\dot{v}\in \tilde{H}$ be a representative of $v$. Then 

$${\rm Tr}\,\left(v\,\left|\,{\rm Hom}_K\left({\rm Ind}_H^K(V),W\right.\right)\right)=\left\langle {\rm Ind}_{H\dot{v}}^K(\varphi_v),\psi\right\rangle_K$$where $\varphi_v$ denotes the restriction of $\varphi$ to $H\dot{v}$.\label{indrel}
\end{proposition}

\subsection{Littlewood-Richardson coefficients}\label{LR}

For a positive integer $m$, we denote by $S_m$ the symmetric group in $m$ letters. 

\begin{notation} For a subgroup $H$ of a group $K$, we denote by $W_K(H)$ the quotient $N_K(H)/H$.
 \label{notation}\end{notation}

Fix a sequence $\tau_o=(a_1,m_1)(a_2,m_2)\cdots(a_s,m_s)$ with $a_i,m_i\in\Z_{>0}$ such that $\sum_ia_im_i=n$ and $m_i\neq m_j$ if $i\neq j$. Put $$S:=(S_{m_1})^{a_1}\times\cdots\times (S_{m_s})^{a_s}\subset S_n$$where $(S_m)^d$ stands for $S_m\times\cdots\times S_m$ ($d$ times). Then we may write $N_{S_n}(S)$ as the semidirect product $S\rtimes\left(\prod_{i=1}^sS_{a_i}\right)$ where each $S_{a_i}$ acts on $(S_{m_i})^{a_i}$ by permutation of the coordinates. 

Hence\begin{equation}W_{S_n}(S)\simeq \prod_{i=1}^sS_{a_i}.\label{iso1}\end{equation}The group $N_{S_n}(S)$ acts on the category of $\kappa[S]$-modules in the natural way, i.e., if $\rho:S\rightarrow \GL(V)$ and $n\in N_{S_n}(S)$, then $n^*(\rho)$ is the representation $\rho\circ n^{-1}:S\rightarrow\GL(V)$.

For a representation $\rho:S\rightarrow \GL(V)$, we denote by $W_{S_n}(S,\rho)$ the quotient $N_{S_n}(S,\rho)/S$ where $$N_{S_n}(S,\rho)=\{n\in N_{S_n}(S)\,|\, n^*(\rho)\simeq\rho\}.$$

Let $\rho:S\rightarrow\GL(V)$ be an irreducible representation. Then for each $i=1,\dots,s$, there exists a partition $(d_{i,1},\dots,d_{i,r_i})$ of $a_i$ and non-isomorphic irreducible $\kappa[S_{m_i}]$-modules  $V_{i,1},\dots, V_{i,r_i}$ such that $$V=\bigotimes_{i,j}T^{d_{i,j}}\left(V_{i,j}\right)$$where for a $\kappa$-vector space $E$, we put $T^dE:=E\otimes\cdots\otimes E$ with $E$ repeated $d$ times.

Then the isomorphism (\ref{iso1}) restricts to an isomorphism 

$$W_{S_n}(S,\rho)\simeq \prod_{i,j}S_{d_{i,j}}.$$

For each $(i,j)$, the group $(S_{m_i})^{d_{i,j}}\rtimes S_{d_{i,j}}$ acts on $T^{d_{i,j}}(V_{i,j})=V_{i,j}\otimes\cdots \otimes V_{i,j}$ as $$(w,s)\cdot(x_1\otimes\cdots\otimes x_{d_{i,j}})=(w_1\cdot x_{s^{-1}(1)}\otimes\cdots \otimes w_{d_{i,j}}\cdot x_{s^{-1}(d_{i,j})}).$$This defines an action of $N_{S_n}(S,\rho)\simeq \prod_{i,j}\left((S_{m_i})^{d_{i,j}}\rtimes S_{d_{i,j}}\right)\simeq S\rtimes\left(\prod_{i,j}S_{d_{i,j}}\right)$ on $V$. We denote by $\tilde{\chi}:N_{S_n}(S,\rho)\rightarrow\kappa$ the corresponding character, and for $v\in \prod_{i,j}S_{d_{i,j}}$ we denote by $\tilde{\chi}_v$ its restriction to the coset $Sv$. 

By Proposition \ref{indrel} we have:

\begin{proposition} For any $\kappa[S_n]$-module $W$ with character $\psi$ and any $v\in W_{S_n}(S,\rho)$ we have 

$${\rm Tr}\,\left(v\,\left|\,{\rm Hom}_{S_n}\left({\rm Ind}_S^{S_n}(V),W\right.\right)\right)=\left\langle {\rm Ind}_{S v}^{S_n}(\tilde{\chi}_v),\psi\right\rangle_{S_n}.$$
\label{indrel2}\end{proposition}

\begin{lemma} Let $\chi_{i,j}$ be the character associated with the $\kappa[S_{m_i}]$-modules $V_{i,j}$. Assume that $v$ acts on each $(S_{m_i})^{d_{i,j}}$ by circular permutation of the coordinates, namely $v\cdot (g_1,\dots,g_{d_{i,j}})=(g_2,g_3,\dots,g_{d_{i,j}},g_1)$. Let $w_{i,j}=(w_{i,j,1},w_{i,j,2},\dots,w_{i,j,d_{i,j}})\in (S_{m_i})^{d_{i,j}}$ and let $w\in S=\prod_{i,j}(S_{m_i})^{d_{i,j}}$ be the element with coordinates $w_{i,j}$. We have$$\tilde{\chi}(w,v)=\prod_{i,j}\chi_{i,j}(w_{i,j,1}w_{i,j,2}\cdots w_{i,j,d_{i,j}}).$$
\end{lemma}

We now show that  this trace  is also a Littlewood-Richardson coefficient (or more precisely a twisted version of it). We will use this result later on.

Let $\x=\{x_1,x_2,\dots\}$ be an infinite set of variables and let $\Lambda(\x)$ be the corresponding ring of symmetric functions. For a partition $\lambda$, let $s_\lambda(\x)$ be the associated Schur symmetric function. 
Let $\unlhd$ denote the dominance ordering on the set of partitions $\calP$. For a type $\omega=(d_1,\omega^1)\cdots(d_r,\omega^r)\in{\bf T}_n$, define $\omega^+$ as the partition $\sum_{i=1}^rd_i\cdot\omega^i$. 

For a type $\omega=(d_1,\omega^1)\cdots(d_r,\omega^r)\in{\bf T}_n$, we define $\{c_\omega^\mu\}_{\mu\in\calP_n}$ by 

$$s_\omega(\x):=s_{\omega^1}(\x^{d_1})s_{\omega^2}(\x^{d_2})\cdots s_{\omega^r}(\x^{d_r})=\sum_{\mu\unlhd \omega^+} c_\omega^\mu s_\mu(\x)$$where $\x^d:=\{x_1^d,x_2^d,\dots\}$. We call the coefficients $c_\omega^\mu$ the \emph{twisted Littlewood-Richardson coefficients}. If $d_1=d_2=\cdots=d_r=1$, these are the usual Littlewood-Richardson coefficients.

For $\lambda=(1^{m_1},2^{m_2},\dots)\in\calP$, put $$z_\lambda:=\prod_{i\geq 1}i^{\,m_i}.m_i!.$$It is also the cardinality of the centralizer in $S_{|\lambda|}$ of an element of type $\lambda$ (i.e. whose decomposition as a product of disjoint cycles is given by $\lambda$). We denote by $\chi^\lambda$ the irreducible character of $S_{|\lambda|}$ associated to $\lambda$ as in Macdonald \cite[I, \S 7]{macdonald} and by $\chi^\lambda_\mu$ its value at an element of type $\mu$. 

\begin{proposition} We have 

$$c_\omega^\mu=\sum_\rho\chi^\mu_\rho\sum_{\alpha}\left(\prod_{i=1}^rz_{\alpha_i}^{-1}\chi^{\omega^i}_{\alpha^i}\right)$$where the second sum runs over the $\alpha=(\alpha^1,\dots,\alpha^r)\in \calP_{|\omega^1|}\times\cdots\times\calP_{|\omega^r|}$ such that $\cup_i d_i\cdot\alpha^i=\rho$.
\label{twistedLR}\end{proposition}

\begin{proof} We have $s_\lambda(\x^d)=\sum_\rho z_\rho^{-1}\chi^\lambda_\rho p_\rho(\x^d)$ where $p_\rho$ is the power symmetric function (see \cite{macdonald}). On the other hand, $p_{\rho^1}(\x^{d_1})\cdots p_{\rho^r}(\x^{d_r})=p_{\cup_id_i\cdot\rho^i}(\x)$. Hence 

$$s_\omega(\x)=\sum_\rho\left(\sum_\alpha\prod_iz_{\alpha^i}^{-1}\chi^{\omega^i}_{\alpha^i}\right)p_\rho(\x)$$where the second sum runs over the $\alpha=(\alpha^1,\dots,\alpha^r)\in \calP_{|\omega^1|}\times\cdots\times\calP_{|\omega^r|}$ such that $\cup_i d_i\cdot\alpha^i=\rho$. We now decompose $p_\rho$ in the basis $\{s_\lambda\}_\lambda$ and we get the result.
\end{proof}

For $\lambda\in\calP$, we denote by $V_\lambda$ an irreducible $\kappa[S_{|\lambda|}]$-module with corresponding character $\chi^\lambda$.

\begin{proposition}Put $V_\omega:=\bigotimes_{i=1}^rT^{d_i}V_{\omega^i}$ and  $S:=\prod_i (S_{|\omega^i|})^{d_i}$ and let $\rho$ be the representation $S\rightarrow \GL(V_\omega)$. Let $v\in W_{S_n}(S,\rho)$ be the element which acts on each $(S_{|\omega^i|})^{d_i}$ by circular permutation of the coordinates. For any $\mu\in\calP_n$ we have
 
$${\rm Tr}\,\left(v\,\left|\,{\rm Hom}_{S_n}\left({\rm Ind}_S^{S_n}(V_\omega),V_\mu\right.\right)\right)=c_\omega^\mu.$$

\label{twistedLRTR}\end{proposition}

\begin{proof} This is a consequence of Proposition \ref{indrel2} and Proposition \ref{twistedLR}.
 
\end{proof}

\subsection{Rational Levi subgroups and Weyl groups}\label{ratLevi}

By a \emph{Levi subgroup} of $G$, we shall mean a Levi subgroup of a parabolic subgroup of $G$, i.e., a subgroup of $G$ which is $\GL_n$-conjugate to some  subgroup of the form $\prod_{i=1}^r\GL_{n_i}$ with $\sum_i n_i=n$. A maximal torus of $G$ is a Levi subgroup which is isomorphic to $(\K^{\times})^n$. Let $L$ be an $F$-stable Levi subgroup of $G$. An $F$-stable  subtorus of $S$ of $L$ of rank $r$ is said to be \emph{split} if there is an isomorphism $S\simeq (\K^{\times})^r$ which is defined over $\F_q$, i.e., $S^F\simeq (\F_q^{\times})^r$. The $\F_q$-rank of $L$ is defined as the maximal value of the ranks of the split subtori of $L$. Since the maximal torus of diagonal matrices is split, any $F$-stable Levi subgroup that contains diagonal matrices is of $\F_q$-rank $n$. 

If $T$ is an $F$-stable maximal torus of $L$ of same $\F_q$-rank as $L$, in which case we say that $T$ is an $L$-split maximal torus of $L$. In this case we denote by $W_L$, instead of $W_L(T)$ (see Notation \ref{notation}), the Weyl group of $L$ with respect to $T$.

If $f$ is a group automorphism of $K$, we say that two elements $k$ and $h$ of $K$ are $f$-\emph{conjugate} if there exists $g\in K$ such that $k=ghf(g)^{-1}$.

The identification of the symmetric group $S_n$ with the monomial matrices in $\GL_n$ with entries in $\{0,1\}$ gives an isomorphism $S_n\simeq W_G$. Fix a sequence of integers ${\bf m}=(m_1,\dots,m_r)$ such that $\sum_im_i=n$ and consider the Levi subgroup $L_o=\GL_{\bf m}:=\prod_{i=1}^r\GL_{m_i}$. Then $W_{L_o}=S_{\bf m}:=\prod_{i=1}^rS_{m_i}$. The $G^F$-conjugacy classes of the $F$-stable Levi subgroups of $G$ that are $G$-conjugate to $L_o$ are parametrized by the conjugacy classes of $W_G(L_o)=W_{S_n}(S_{\bf m})$ \cite[Proposition 4.3]{DM}. For $v\in N_{S_n}(S_{\bf m})$, we denote by $L_v$ a representative of the $G^F$-conjugacy class (of  $F$-stable Levi subgroups) which corresponds to the conjugacy class of $v$ in $W_{S_n}(S_{\bf m})$. Then $(L_v,F)\simeq (L_o,vF)$, i.e., the action of the Frobenius $F$ on $L_v$ corresponds to the action of $vF$ on $L_o$ given by $vF(g):=vF(g)v^{-1}$ for any $g\in L_o$. Since $F$ acts trivially on $W_G\simeq S_n$, we have $(W_{L_v},F)\simeq (S_{\bf m},v)$. By \S\ref{LR}, there exists a decomposition 

$$S_{\bf m}=(S_{n_1})^{d_1}\times\cdots\times(S_{n_r})^{d_r}$$for some sequence $(d_1,n_1)(d_2,n_2)\cdots(d_r,n_r)$ and a specific choice of an element $\sigma$ in the coset $vS_{\bf m}$ which acts on each component $(S_{n_i})^{d_i}$ by circular permutation of the coordinates. Taking the $G^F$-conjugate $L_\sigma$ of $L_v$ if necessary we may assume that $v=\sigma$. We also have

$$L_o=\prod_{i=1}^r\left(\GL_{n_i}\right)^{d_i},\text{ and }\hspace{.2cm}(L_v)^F\simeq(L_o)^{vF}\simeq \prod_{i=1}^r\GL_{n_i}(\F_{q^{d_i}}).$$

Now let $L$ be any $F$-stable Levi subgroup of $G$. Consider the semi-direct product $W_L\rtimes\langle F\rangle$ where $\langle F\rangle$ is the cyclic group generated by the Frobenius automorphism on $W_L$. If $\psi$ is a character of $W_L\rtimes\langle F\rangle$, then for all $a\in W_L$, we have $\psi(F(a))=\psi(a)$ since $(F(a),1)\in W_L\rtimes\langle F\rangle$ is the conguate of $(a,1)$ by $(1,F)$. Hence the restriction of $\psi$ to $W_L$ is an $F$-stable character of $W_L$.
Conversely, given an $F$-stable character $\chi$ of $W_L$, we now define an extension $\tilde{\chi}$ of $\chi$ to $W_L\rtimes\langle F\rangle$ as follows. We have $L=L_v$ for some $\bf m$ and $v\in N_{S_n}(S_{\bf m})$ by the above discussion so that we may identify $W_L\rtimes\langle F\rangle$ with $S_{\bf m}\rtimes\langle v\rangle$. For an $v$-stable character $\chi$ of $S_{\bf m}$ we define the extension $\tilde{\chi}$ of $S_{\bf m}\rtimes\langle v\rangle$ as in \S\ref{LR}.

The $L^F$-conjugacy classes of the $F$-stable maximal tori of $L$ are parametrized by the $F$-conjugacy classes of $W_L$ \cite[Proposition 4.3]{DM}. If $w\in W_L$, we denote by $T_w$ an $F$-stable maximal torus of $L$ which is in the $L^F$-conjugacy class associated to the $F$-conjugacy class of $w$. We put $\mathfrak{t}_w:={\rm Lie}\,(T_w)$.

\subsection{Springer correspondence for relative Weyl groups}\label{actionW}

Let $P$ be a parabolic subgroup of $G$ and $L$ a Levi factor of $P$. Let $\mathfrak{l}$ be the Lie algebra of 
$L$ and let $z_{\mathfrak{l}}$ denotes its center. Recall that the classical Springer correspondence gives a bijection $$\mathfrak{C}=\mathfrak{C}_L:{\rm Irr}\,W_L\rightarrow\{\text{nilpotent orbits of }\mathfrak{l}\}$$which maps the trivial character to the regular nilpotent orbit. Moreover if $L$ is $F$-stable then $\mathfrak{C}$ restricts to a bijection between the $F$-stable irreducible characters of $W_L$ and the $F$-stable nilpotent orbits of $\mathfrak{l}$.
Recall that if $L=G$ and $\lambda\in\calP_n$, then the size of the Jordan blocks of the nilpotent orbit $\mathfrak{C}(\chi^{\lambda})$ are given by the partition $\lambda$.

Let $\epsilon\in{\rm Irr}\, W_L$ be the sign character. For $\chi\in{\rm Irr}\,W_L$ put $\chi':=\chi\otimes\epsilon$. Then let $\mathfrak{C}_\epsilon:{\rm Irr}\,W_L\overset{\sim}{\rightarrow}\{\text{nilpotent orbits of }\mathfrak{l}\}$ be the map which sends $\chi$ to $\mathfrak{C}(\chi')$. The bijection $\mathfrak{C}_\epsilon$ was actually the first correspondence to be discovered \cite{springer}.

Let $C$ be a nilpotent orbit of $\mathfrak{l}$ and put $\Sigma=\sigma+C$ with $\sigma\in z_{\mathfrak{l}}$. Consider the relative Weyl group $$W_G(L,\Sigma):=\{n\in N_G(L)\,|\, n\Sigma n^{-1}=\Sigma\}/L.$$Recall that $\Sigma$ is of the form $\sigma+C$ with $C$ a nilpotent orbit of $\mathfrak{l}$ and $\sigma\in z_\mathfrak{l}$. Put $M:=C_G(\sigma)$, then $W_G(L,\Sigma)=W_M(L,C)$. Let $\calO$ be the orbit of $\gl_n$ whose Zariski closure is the image of the projection $p:\mathbb{X}_{L,P,\Sigma}\rightarrow\g$ on the first coordinate.

Let $\g_\sigma$ be the set of elements $x\in\g$ whose semisimple part is $G$-conjugate to $\sigma$. Note that the image of $p$ is contained in $\g_\sigma$. The set $\g_\sigma$ has a finite number of $G$-orbits which are indexed by the irreducible characters of $W_M$ by $\mathfrak{C}_M$. If $\chi$ is an irreducible character of $W_M$ we denote by $\calO_\chi$ the corresponding adjoint orbit in $\g_\sigma$. For $\chi\in{\rm Irr}\, W_M$, put $$A_\chi={\rm Hom}_{W_M}\left({\rm Ind}_{W_L}^{W_M}(V_C),V_\chi\right)$$where $V_C$ is the irreducible $W_L$-module corresponding to the nilpotent orbit $C$ under $\mathfrak{C}$. 

We have the following result due to Springer in the case where $\calO$ is nilpotent regular (see Borho and MacPherson \cite[3.1]{BM} for the regular nilpotent case and Lusztig \cite[2.5]{LuCV} for the general case).

\begin{proposition}We have $${\rm Ind}_{\mathfrak{l}\subset\mathfrak{p}}^{\g}\left(\pIC{\overline{\Sigma}}\right)=p_*\left(\pIC {\mathbb{X}_{L,P,\Sigma}}\right)= \bigoplus_{\chi\in{\rm Irr}\,W_M} A_\chi\otimes\pIC {\overline{\calO}_\chi}$$and $A_\chi=0$ if $\calO_\chi$ is not included in $\overline{\calO}$. The multiplicity $A_{\chi_o}$  corresponding to $\calO=\calO_{\chi_o}$ is the trivial character of $W_M(L,C)$.\label{Springer}\end{proposition}

 If $\calO$ is regular nilpotent, $L=T$ and if $\Sigma=\{0\}$, then this is the classical Springer correspondence.

The group $W_M(L,C)$ is naturally isomorphic to $W_{W_M}(W_L,\rho)$. As shown in \S \ref{LR}, the action of $W_L$ on $V_C$ can be extended to an action of $N_{W_M}(W_L,\rho)$ on $V_C$. By \S \ref{prefini} it gives a structure of $W_M(L,C)$-module on each $A_\chi$ and so by Proposition \ref{Springer} we have an action of $W_M(L,C)$ on $p_*\left(\pIC {\mathbb{X}_{L,P,\Sigma}}\right)$.

\begin{remark} It is also possible to define an action of $W_M(L,C)$ on $p_*\left(\pIC {\mathbb{X}_{L,P,\Sigma}}\right)$ using the approach in Bohro and MacPherson \cite{BM} by considering partial simultaneous resolutions. \end{remark}

To alleviate the notation put $K:=p_*\left(\pIC {\mathbb{X}_{L,P,\Sigma}}\right)$ and $K_\chi:=A_\chi\otimes\pIC {\overline{\calO}_\chi}$. Assume now that $(M,Q,L,P,\Sigma)$ is $F$-stable and let $F:\mathbb{X}_{L,P,\Sigma}\rightarrow \mathbb{X}_{L,P,\Sigma}$ be the Frobenius given by $F(X,gP)=(F(x),F(g)P)$. Then the morphism $f$ commutes with the Frobenius endomorphisms. Let $\varphi:F^*(\kappa)\simeq\kappa$ be the isomorphism (in the category of sheaves on $\mathbb{X}_{L,P,\Sigma}^o$) which induces the identity on stalks at $\mathbb{F}_q$-points. It induces a canonical isomorphism $F^*\left(\pIC {\mathbb{X}_{L,P,\Sigma}}\right)\simeq \pIC {\mathbb{X}_{L,P,\Sigma}}$ which in turns induces a canonical isomorphism $\tilde{\varphi}:F^*(K)\simeq K$. Note that the orbits $\calO_\chi$ are $F$-stable and $F$ acts trivially on $W_M$. Hence $F^*(K_\chi)\simeq K_\chi$ and so $\tilde{\varphi}$ induces an isomorphism $\tilde{\varphi}_\chi:F^*(K_\chi)\simeq K_\chi$ for each $\chi$. Now we define an isomorphism $\phi_\chi:F^*(\pIC {\overline{\calO}_\chi})\simeq \pIC {\overline{\calO}_\chi}$ with the requirement that its tensor product with the identity on $A_\chi$ gives $\tilde{\varphi}_\chi$.

We then have $${\bf X}_{\pIC {\overline{\calO}_\chi},\phi_\chi}=q^{\frac{1}{2}\left({\rm dim}\,\calO-{\rm dim}\,\calO_\chi\right)}{\bf X}_{\pIC {\overline{\calO}_\chi}}.$$

Since the $A_\chi$ are $W_M(L,C)$-modules, each $v\in W_M(L,C)$ induces an isomorphism $K_\chi\simeq K_\chi$ and so an isomorphism $\theta_v:K\simeq K$ such that $${\bf X}_{K,\theta_v\circ\tilde{\varphi}}=\sum_\chi{\rm Tr}\,\left(v\,|\, A_\chi\right)\,q^{{\frac{1}{2}({\rm dim}\, \calO-{\rm dim}\,\calO_\chi)}}\,{\bf X}_{\pIC {\overline{\calO}_\chi}}.$$

\subsection{Deligne-Lusztig induction and Fourier transforms}

Here we recall the definition of Deligne-Lusztig induction both in the group setting (which is now standard \cite{DLu}) and in the Lie algebra setting \cite{letellier1}. We then recall the commutation formula between Fourier transforms  and Deligne-Lusztig induction (in the Lie algebra case) which is the main result of \cite{letellier}. This commutation formula is an essential ingredient in the proof of the main theorem of the paper. Although the theory is available for any connected reductive algebraic groups we keep our assumption $G=\GL_n(\overline{\F}_q)$.

For any subset $Y$ of $X$, we denote by $1_Y$ the function $X\rightarrow\kappa$ that takes the value $1$ on $Y$ and the value $0$ elsewhere.

\subsubsection{Generalized induction}\label{generalized} Let $H$ and $K$ be two finite groups and let $M$ be a finite dimensional $\K$-vector space. We say that $M$ is an \emph{$H$-module-$K$} if it is a left $\kappa[H]$-module and a right $\kappa[K]$-module such that $(a\cdot x)\cdot b=a\cdot(x\cdot b)$ for any $a \in\kappa[H]$, $b\in\kappa[K]$ and $x\in M$. Then $M$ defines a functor from the category of finite dimensional left $\kappa[K]$-modules to the category of finite dimensional left $\kappa[H]$-modules by $V\mapsto M\otimes_{\kappa[K]}V$. This functor induces an obvious  $\kappa$-linear map $R_K^H:\calC(K)\rightarrow\calC(H)$ on $\kappa$-vector spaces of class functions.

The approach of generalized induction with bi-modules is due to Brou\'e. We have the following formula \cite[4.5]{DM}.

\begin{proposition}\label{trace} Let $f\in\calC(K)$ and $g\in H$, then 

$$R_K^H(f)(g)=|K|^{-1}\sum_{k\in K}{\rm Trace}\,\left(\left.\left(g,k^{-1}\right)\,\right|\, M\right) f(k).$$
\end{proposition}

\subsubsection{The group setting: Deligne-Lusztig induction}\label{DLgroup}

Let $L$ be an $F$-stable Levi subgroup of a parabolic subgroup $P$ of $G$ and let $V$ be the unipotent radical of $P$. Consider the Lang map $\mathcal{L}_G:G\rightarrow G$, $x\mapsto x^{-1}F(x)$. In \cite{Lufinite}, Lusztig considers the variety $\mathcal{L}_G^{-1}(V)$ which is endowed with an action of $G^F$ by left multiplication and with an action of $L^F$ by right  multiplication. These actions commutes and so make $H_c^i\left(\mathcal{L}_G^{-1}(V),\kappa\right)$ into a $G^F$-module-$L^F$. Consider the virtual $G^F$-module-$L^F$ $$H_c^*\left(\mathcal{L}_G^{-1}(V)\right)=\sum_i(-1)^iH_c^i\left(\mathcal{L}_G^{-1}(V),\kappa\right).$$The $\kappa$-linear map $R_L^G:\calC(L^F)\rightarrow\calC(G^F)$ associated with this virtual bi-module is called  \emph{Deligne-Lusztig induction}.

Let us put $$S_L^G(g,h):={\rm Trace}\,\left(\left(g,h^{-1}\right)\,\left|\,H_c^*\left(\mathcal{L}_G^{-1}(V)\right)\right. \right).$$

By Proposition \ref{trace} we have for any $f\in\calC(L^F)$

\begin{equation}R_L^G(f)(g)=|L^F|^{-1}\sum_{h\in L^F}S_L^G(g,h) f(h).\label{trace1}\end{equation}If $M$ is an $F$-stable Levi subgroup of $G$ containing $L$, we define  $R_L^M$ exaclty as above replacing the letter $G$ by the letter $M$.

Let $L_{\rm uni}$ be the subvariety of unipotent element of $L$. We now list some properties of this induction which are standard.

\begin{proposition} \label{basic}(i) $R_L^G$ does not depend on the choice of the parabolic subgroup $P$ having $L$ as  a Levi subgroup.

\noindent (ii) If $L\subset M$ is an inclusion of Levi subgroups, then $R_M^G\circ R_L^M=R_L^G$.

\noindent (iii) ${\rm Res}^G_{G_{\rm uni}}\circ R_L^G=R_L^G\circ{\rm Res}^L_{L_{\rm uni}}$ where ${\rm Res}^L_{L_{\rm uni}}:\calC(L^F)\rightarrow\calC (L^F)$ maps a function $f$ to the unipotently supported function that takes the same values as $f$ on $L_{\rm uni}^F$.
\end{proposition}

For $w\in W_L$ we put $$Q_{T_w}^L:=R_{T_w}^L(1_1)$$ where $1_1$ denotes the function with value $1$ at $1$ and with value $0$ elsewhere. We call the function $Q_{T_w}^L$ the \emph{Green functions} of $L^F$. They are defined by Deligne and Lusztig in \cite{DLu}.

When $L=G$ in which case $W_L\simeq S_n$, the Green functions are related to the well-known Green polynomials as follows. The decomposition of $w$ as product of disjoint cycles gives a partition say $\lambda$. Then the value of $Q_{T_w}^G$ at the unipotent conjugacy class associated with the partition $\mu$ is the Green polynomial $Q^\mu_\lambda$ in the notation of \cite[III, 7]{macdonald}.

Because of Proposition \ref{basic} (iii), we may also write the function $Q_{T_w}^L$ as ${\rm Res}^L_{L_{\rm uni}}\circ R_{T_w}^L(1_{T_w})$. 

We have the following important formula \cite{DLu} due to Deligne and Lusztig.

\begin{theorem} Let $f\in\calC(T_w^F)$ and let $l\in L^F$. Then

\begin{equation}R_{T_w}^L(f)(l)=|C_L(l_s)^F|^{-1}\sum_{\{h\in L^F\,|\, l_s\in hT_wh^{-1}\}}Q_{hT_wh^{-1}}^{C_L(l_s)}(l_u)f(h^{-1}l_sh).\label{DLu1}\end{equation}where $l=l_sl_u$ is the Jordan decomposition of $l$ with $l_s$ the semisimple part and $l_u$ the unipotent part.

\label{DLu}\end{theorem}

\subsubsection{The Lie algebra setting: Fourier transforms}

Fourier transforms of functions on reductive Lie algebras over finite fields were first investigated by Springer in the study of the geometry of nilpotent orbits \cite{springer}. Interesting applications in the representation theory of connected reductive groups over finite fields were then found by many authors including Kawanaka (e.g. \cite{Kawanaka}), Lusztig (e.g. \cite{Lufour}), Lehrer (e.g. \cite{Lehrer}), Waldspurger \cite{Wald} and the author himself (e.g. \cite{letellier}). 

Let us recall the definition and basic properties of Fourier transforms. The most important property of Fourier transforms will be stated in the next section \S \ref{DLla}.

We fix once for all a non-trivial additive character
$\Psi:\mathbb{F}_q\rightarrow\kappa^{\times}$
and we denote by $\mu:\mathfrak{g}\times\mathfrak{g}\rightarrow
\K$ the trace map $(a,b)\mapsto \text{Trace}
(ab)$. It is a non-degenerate $G$-invariant symmetric bilinear form
defined over $\mathbb{F}_q$. Let $\text{Fun}(\mathfrak{g}^F)$ be the $\kappa$-vector space of all functions $\mathfrak{g}^F\rightarrow \kappa$. We define the Fourier transform
$\mathcal{F}^{\mathfrak{g}}:
\text{Fun}(\mathfrak{g}^F)\rightarrow\text{Fun}(\mathfrak{g}^F)$ with
respect to $(\Psi,\mu)$ by
$$
\mathcal{F}^{\mathfrak{g}}(f)(x)=
\sum_{y\in\mathfrak{g}^F}\Psi\big(\mu(x,y)\big)\,f(y).
$$
A detailed review on properties of Fourier transforms can be found in \cite{Lehrer}. Here we just recall what we will need.

Define the convolution product $*$ on ${\rm Fun}\,(\mathfrak{g}^F)$ as $$(f*g)(x)=\sum_{y\in\mathfrak{g}^F}f(y)g(x-y)$$for all $x\in\mathfrak{g}^F$. Then for all $f,g\in {\rm Fun}\,(\mathfrak{g}^F)$, we have \cite[Proposition 3.2.1]{hausel-letellier-villegas}$$\calF^\mathfrak{g}(f*g)=\calF^\mathfrak{g}(f)\cdot\calF^\mathfrak{g}(g).$$For any $f\in{\rm Fun}\,(\mathfrak{g}^F)$ it is straightforward to check that \begin{equation}|\mathfrak{g}^F|\cdot f(0)=\sum_{x\in\mathfrak{g}^F}\calF^\mathfrak{g}(f)(x).\label{intfor}\end{equation}

\subsubsection{The Lie algebra setting: Deligne-Lusztig induction}\label{DLla}

We now review Deligne-Lusztig induction in the Lie algebra setting. Details and proofs can be found in \cite{letellier1} \cite{letellier}.

Consider $L,P,V$ as in \S \ref{DLgroup} and let $\mathfrak{l},\mathfrak{p},\mathfrak{n}$ be their respective Lie algebras.
We denote by $\calC(\mathfrak{g}^F)$ the $\kappa$-vector space of functions $\mathfrak{g}^F\rightarrow\kappa$ which are constant on adjoint orbits.

It is not clear wether there is a Lie algebra analogue of the variety $\mathcal{L}_G^{-1}(V)$. The naive guess $\mathcal{L}_\mathfrak{g}^{-1}(\mathfrak{n})$ with $\mathcal{L}_\mathfrak{g}:\mathfrak{g}\rightarrow\mathfrak{g}, x\mapsto F(x)-x$ does not give anything interesting.

However we have the following formula \cite[Lemma 12.3]{DM} obtained independently by Digne-Michel and Lusztig.

$$S_L^G(g,l)=|L^F|^{-1}\sum_{\{h\in G^F\,|\,hl_sh^{-1}=l_s\}} |C_L(l_s)^F| |C_G(l_s)^F|^{-1}S_{C_L(l_s)}^{C_G(l_s)}\left(h^{-1}g_uh,l_u\right).$$This formula reduces the computation of $S_L^G(g,l)$ to its computation at unipotent elements.

We define our $S_\mathfrak{l}^\mathfrak{g}(x,y)$ using the Lie algebra analogue of this formula as follows. Let $\mathfrak{g}_{\rm nil}$ be the variety of nilpotent elements of $\mathfrak{g}$ and let $\omega:\mathfrak{g}_{\rm nil}\rightarrow G_{\rm uni}$ be the isomorphism given by $x \mapsto x+1$. For $(x,y)\in\mathfrak{g}^F\times\mathfrak{l}^F$, we put

$$S_\mathfrak{l}^\mathfrak{g}(x,y):=|L^F|^{-1}\sum_{\{h\in G^F\,|\,hy_sh^{-1}=x_s\}} |C_L(y_s)^F| |C_G(y_s)^F|^{-1}S_{C_L(y_s)}^{C_G(y_s)}\left(h^{-1}\omega(x_n)h,\omega(y_n)\right)$$
where $x=x_s+x_n$ is the Jordan decomposition of $x$ with $x_s$ the semisimple part and $x_n$ the nilpotent part.

We define our Lie algebra version of Deligne-Lusztig induction  $R_\mathfrak{l}^\mathfrak{g}:\calC(\mathfrak{l}^F)\rightarrow\calC(\mathfrak{g}^F)$ as $$R_\mathfrak{l}^\mathfrak{g}(f)(x)=|L^F|^{-1}\sum_{y\in\mathfrak{l}^F}S_\mathfrak{l}^\mathfrak{g}(x,y)f(y).$$

This definition of $R_\mathfrak{l}^\mathfrak{g}$ works also if we replace the isomorphism $\omega$ by any $G$-equivariant isomorphism $\mathfrak{g}_{\rm nil}\simeq G_{\rm uni}$ defined over $\F_q$ (e.g. the exponential map when the characteristic is large enough). We actually prove in \cite[Remark 5.5.17]{letellier} that the definition of $R_\mathfrak{l}^\mathfrak{g}$ does not depend on the choice of such an isomorphism. 

It is also easy to prove that our induction $R_\mathfrak{l}^\mathfrak{g}$ satisfies the analogous properties in Proposition \ref{basic}, see \cite{letellier1} for details.

The Lie algebra analogue of Theorem \ref{DLu} is by definition of $R_{\mathfrak{t}_w}^\mathfrak{l}$: If $f\in\calC(\mathfrak{t}_w^F)$ and  $x\in \mathfrak{t}^F$, then 

\begin{equation}R_{\mathfrak{t}_w}^\mathfrak{l}(f)(x)=|C_L(x_s)^F|^{-1}\sum_{\{h\in L^F\,|\, x_s\in h\mathfrak{t}_wh^{-1}\}}Q_{hT_wh^{-1}}^{C_L(x_s)}\left(\omega(x_n)\right)f(h^{-1}x_s h).\label{charfordef1}\end{equation}

We will also use the following properties \cite[Proposition 3.2.24, Proposition 7.1.8]{letellier}.

\begin{proposition} Let $\calC$ be an $F$-stable  nilpotent orbit of $\mathfrak{l}$ and let $\sigma\in (z_\mathfrak{l})^F$ be such that $C_G(\sigma)=L$. Denote by $\calO^L$ the adjoint orbit $\sigma+\calC$ of $\mathfrak{l}$ and by $\calO$ the unique orbit of $\g$ which contains $\calO^L$. Then we have:

\noindent (i) $R_\mathfrak{l}^\g\left(1_{\calO^L}\right)=1_\calO$,

\noindent (ii) $R_\mathfrak{l}^\g\left({\bf X}_{\IC {\overline{\calO}^L}}\right)={\bf X}_{\IC {\overline{\calO}}}$.
\label{letellierprop}\end{proposition}

Our definition of $R_\mathfrak{l}^\mathfrak{g}$ is not natural and is thus a little bit frustrating especially for other reductive groups where we do not always have an isomorphism between the nilpotent elements and the unipotent ones in small characteristics. However the following theorem \cite[Corollary 6.2.17]{letellier} shows that our definition of $R_\mathfrak{l}^\mathfrak{g}$ behaves well with Fourier transforms (which are not well-defined in the group setting).

\begin{theorem}Put $\epsilon_L=(-1)^{\F_q-{\rm rank}(L)}$. We have $$\calF^\mathfrak{g}\circ R_\mathfrak{l}^\mathfrak{g}=\epsilon_G\epsilon_L q^{{\rm dim} V}R_\mathfrak{l}^\mathfrak{g}\circ\calF^\mathfrak{l}.$$
 \label{theolet}

\end{theorem}

This formula suggests that a more conceptual definition of $R_\mathfrak{l}^\mathfrak{g}$ should exist. In \cite{letellier3} we investigate this problem in greater details and bring a partial answer in terms of the geometry of the semi-direct product $G\ltimes\mathfrak{g}$.

It is proved by Lehrer \cite{Lehrer} that Fourier transforms commute with Harish-Chandra induction. Moreover when the parabolic $P$ is $F$-stable the induction $R_\mathfrak{l}^\mathfrak{g}$ coincides with Harish-Chandra induction (see \cite{letellier1}). Hence Lehrer's result is a particular case of the theorem. 

We also mention that when $\sigma\in\mathfrak{t}_w^F$ is regular (i.e. $C_G(\sigma)=T_w$) then it follows from Kazhdan and Springer's results \cite{Kazhdan}\cite{springer} that $\calF^\mathfrak{g}\circ R_{\mathfrak{t}_w}^\mathfrak{g}(1_\sigma)=\epsilon_G\epsilon_{T_w} q^{{\rm dim} U}R_{\mathfrak{t}_w}^\mathfrak{g}\circ\calF^{\mathfrak{t}_w}(1_\sigma)$ where $U$ is the unipotent radical of a Borel subgroup of $G$.

\subsection{Characters of finite general linear groups}\label{irrG}

The character table of $G^F$ was first computed by Green \cite{green}. In \cite{LSr}, Lusztig and Srinivasan describe it in terms of Deligne-Lusztig theory \cite{LSr}. This is done as follows.

Let $L$ be an $F$-stable Levi subgroup of $G$ and let $\varphi$ be an $F$-stable irreducible character of $W_L$. The function $\mathcal{X}_\varphi^L:L^F\rightarrow \kappa$ defined by \begin{equation}\mathcal{X}_{\varphi}^L=|W_L|^{-1}\sum_{w\in
W_L}\tilde{\varphi}\,(wF)R_{T_w}^L(1_{T_w})\label{unipotentchar}\end{equation}is an irreducible
character of $L^F$ (here $\tilde{\varphi}$ is the extension of $\varphi$ defined in \S \ref{ratLevi}). The characters $\mathcal{X}_{\varphi}^L$ are
called the \emph{unipotent characters} of $L^F$. \vspace{.2cm}

For $g\in G^F$ and $\theta\in \text{Irr}\,(L^F)$, let
${^g}\theta\in\text{Irr}\,(gL^Fg^{-1})$ be defined by
${^g}\theta(glg^{-1})=\theta(l)$. We say that a linear character
$\theta:L^F\rightarrow \kappa^{\times}$ is
\emph{regular} if it satisfies the conditions (a) and (b) of \cite[3.1]{LSr}. We denote by $\text{Irr}_{\rm reg}(L^F)$ the set
of regular linear characters of $L^F$. For
$\theta^L\in\text{Irr}_{\rm reg}(L^F)$, the virtual character
\begin{equation}\mathcal{X}:=\epsilon_G\epsilon_LR_L^G\big(\theta^L\cdot\mathcal{X}_{\varphi}^L\big)=
\epsilon_G\epsilon_L|W_L|^{-1}\sum_{w\in
W_L}\tilde{\varphi}\,(wF)R_{T_w}^G(\theta^{T_w})\label{charform1}\end{equation}
where $\theta^{T_w}:=\theta^L|_{T_w}$, is an irreducible true
character of $G^F$ and any irreducible character of $G^F$ is
obtained in this way \cite{LSr}. An irreducible character of $G^F$
is thus completely determined by the $G^F$-conjugacy class of a
datum $(L,\theta^L,\varphi)$ with $L$ an $F$-stable Levi subgroup of
$G$, $\theta^L\in\text{Irr}_{\rm reg}(L^F)$ and $\varphi\in
\text{Irr}\,\big(W_L\big)^F$. Characters associated to triples of the form $(L,\theta^L,1)$ are called \emph{semisimple}.

The characters $\epsilon_G\epsilon_{T_w}R_{T_w}^G(\theta)$ are called \emph{Deligne-Lusztig characters}.

\subsection{Fourier transforms of orbital simple perverse sheaves}

We have the Deligne-Fourier transform $\mathcal{F}^{\mathfrak{g}}:\mathcal{D}_c^b(\mathfrak{g})\rightarrow \mathcal{D}_c^b(\mathfrak{g})$ which is defined as follows. 

 We denote by $\mathbb{A}^1$ the affine line over $\K$. Let $h : \mathbb{A}^1\rightarrow \mathbb{A}^1$ be the Artin-Shreier covering defined by $h(t) = t^q-t$. Then, since $h$ is a Galois covering of
$\mathbb{A}^1$ with Galois group $\F_q$ , the sheaf $h_*(\kappa)$ is a local system on $\mathbb{A}^1$ on which
$\F_q$ acts. We denote by $\mathcal{L}_\Psi$ the subsheaf of $h_*(\kappa)$ on which $\F_q$ acts as $\Psi^{-1}$.
                                                ∼
There exists an isomorphism $\varphi_\Psi:F^*(\mathcal{L}_\Psi)\rightarrow\mathcal{L}_\Psi$ such that for any integer $i\geq 1$, we have $\mathbf{X}_{\mathcal{L}_\Psi,\varphi^{(i)}_\Psi}= \Psi\circ{\rm Trace}_{\F_{q^i}/\F_q}:\F_{q^i}\rightarrow \kappa$ (see Katz \cite[3.5.4]{Kat}).
Then for a complex $K\in \mathcal{D}_c^b(\mathfrak{g})$ we define $$\mathcal{F}^{\mathfrak{g}}(K):=(p_1)_!\big((p_2)^*(K)\otimes\mu^*(\mathcal{L}_\Psi)\big)[{\rm dim}\,\mathfrak{g}]$$where $p_1,p_2:\mathfrak{g}\times\mathfrak{g}\rightarrow\mathfrak{g}$ are the two projections. If $\varphi: F^*(K)\rightarrow K$ is an isomorphism, then it induces a natural ismorphism $\mathcal{F}(\varphi):F^*(\mathcal{F}^{\mathfrak{g}}(K))\rightarrow \mathcal{F}^{\mathfrak{g}}(K)$. Moreover,$$\mathbf{X}_{\mathcal{F}^{\mathfrak{g}}(K),\mathcal{F}(\varphi)}=(-1)^{{\rm dim}\,\mathfrak{g}}\mathcal{F}^{\mathfrak{g}}\big(\mathbf{X}_{K,\varphi}\big).$$

We will need to compute the characteristic functions of the perverse sheaves  $\mathcal{F}^{\mathfrak{g}}\big(\pIC {\overline{\calO}}\big)$, where $\calO$ an $F$-stable adjoint orbit of $\mathfrak{g}$. It is known by results of Lusztig that these perverse sheaves are closely related to the character sheaves on $G$ \cite{Lufour} and that the characteristic functions of character sheaves on $G$ give the irreducible characters of $G^F$ \cite{Greenpoly}. We thus expect to have a tight connection between the characteristic functions of the sheaves $\mathcal{F}^{\mathfrak{g}}\big(\pIC {\overline{\calO}}\big)$ on $\g$ and the irreducible characters of $G^F$.

 More precisely, let $x\in\calO^F$ and put $L=C_G(x_s)$. Let $\varphi$ be the $F$-stable irreducible character of $W_L$ that corresponds to the nilpotent orbit $\calO_{x_n}^L$ of $\mathfrak{l}={\rm Lie}\,(L)$ via the Springer correspondence $\mathfrak{C}_\epsilon$.

\begin{theorem} We have \begin{equation}\mathcal{F}^{\mathfrak{g}}\big(\mathbf{X}_{\IC {\overline{\calO}}}\big)=\epsilon_G\epsilon_Lq^{\frac{1}{2}{\rm dim}\,\calO}|W_L|^{-1}\sum_{w\in
W_L}\tilde{\varphi}(wF)\mathcal{R}_{\mathfrak{t}_w}^{\mathfrak{g}}(\eta^{\mathfrak{t}_w})\label{charforla}\end{equation}where $\eta^{\mathfrak{t}_w}:\mathfrak{t}_w^F\rightarrow \kappa$ is the character $z\mapsto \Psi\big(\mu(x_s,z)\big)$. \label{charfor}
 \end{theorem}

\begin{remark} Note that Formula (\ref{charform1}) is similar to Formula (\ref{charforla}). It shows that $\mathcal{F}^{\mathfrak{g}}\big(\mathbf{X}_{\IC {\overline{\calO}}}\big)$ arises from the $G^F$-conjugacy class of a triple $(\mathfrak{l},\eta^{\mathfrak{l}},\varphi)$ with $\eta^{\mathfrak{l}}:\mathfrak{l}^F\rightarrow\kappa^{\times}, z\mapsto \Psi\big(\mu(x_s,z)\big)$ exactly as in the group setting.
 \label{rem5}\end{remark}

\begin{proof}[Proof of Theorem \ref{charfor}] Let $\calO^L$ be the $L$-orbit of $x$ in $\mathfrak{l}:={\rm Lie}\,(L)$. Then $\overline{\calO}{^L}$ decomposes as $x_s+\overline{\calO}{^L_{x_n}}$ where $\calO_{x_n}^L$ denotes the $L$-orbit of $x_n$ in $\mathfrak{l}$. Then \begin{equation*}\mathbf{X}_{\IC {\overline{\calO}{^L}}}=1_{x_s}*\mathbf{X}_{\IC {\overline{\calO}{^L_{x_n}}}}.\end{equation*}By Proposition \ref{letellierprop}, we have \begin{equation}
\mathbf{X}_{\IC {\overline{\calO}}}=\mathcal{R}_\mathfrak{l}^{\mathfrak{g}}\big(\mathbf{X}_{\IC {\overline{\calO}{^L}}}\big)\label{induc}\end{equation}
Hence from the commutation formula in Theorem \ref{theolet} we have \begin{align*}\mathcal{F}^{\mathfrak{g}}\big(\mathbf{X}_{\IC {\overline{\calO}}}\big)&=\epsilon_G\epsilon_Lq^{\frac{1}{2}({\rm dim}\,G-{\rm dim}\,L)}\,\mathcal{R}_\mathfrak{l}^{\mathfrak{g}}\circ\mathcal{F}^{\mathfrak{l}}\big(\mathbf{X}_{\IC {\overline{\calO}{^L}}}\big)\\ 
&=\epsilon_G\epsilon_Lq^{\frac{1}{2}({\rm dim}\,G-{\rm dim}\,L)}\,\mathcal{R}_\mathfrak{l}^{\mathfrak{g}}\left(\mathcal{F}^{\mathfrak{l}}\big(1_{x_s}\big)\cdot\mathcal{F}^{\mathfrak{l}}\big(\mathbf{X}_{\IC {\overline{\calO}{^L_{x_n}}}}\big)\right)
\end{align*}
We also have: \begin{equation}\mathbf{X}_{\IC {\overline{\calO}{^L_{x_n}}}}=q^{-\delta}|W_L|^{-1}\sum_{w\in
W_L}\tilde{\varphi}'(wF)\mathcal{R}_{\mathfrak{t}_w}^{\mathfrak{l}}(1_0)\label{ic}\end{equation}where $\delta=\frac{1}{2}\big({\rm dim}\,C_L(x_n)-{\rm dim}\,T\big)$. 

Indeed, by Formula (\ref{charfordef1}) the function $\mathcal{R}_{\mathfrak{t}_w}^{\mathfrak{l}}(1_0)$ corresponds to the Green function $Q_{T_w}^L$ via the isomorphism $\omega:\mathfrak{l}_{\rm nil}\simeq L_{\rm uni}$. Moreover if we put $C^L=\omega\big(\calO_{x_n}^L\big)$, then by Lusztig \cite{Greenpoly}, we have ${\rm Res}^L_{L_{\rm uni}}\left(\mathcal{X}_{\varphi'}^L\right)=q^{\delta}\,\mathbf{X}_{\IC {\overline{C}^L}}$ where $\mathcal{X}_{\varphi'}^L$ is the unipotent character of $L^F$ associated to $\varphi'$. Hence Formula (\ref{ic}) is obtained from Formula (\ref{unipotentchar}) via the isomorphism $\omega$.

We now deduce from Formula (\ref{ic}) and Theorem \ref{theolet} that $$\mathcal{F}^{\mathfrak{l}}\big(\mathbf{X}_{\IC {\overline{\calO}{^L_{x_n}}}}\big)=q^{-\delta}|W_L|^{-1}\sum_{w\in
W_L}\tilde{\varphi}'(wF)\epsilon_L\epsilon_{T_w}q^{\frac{1}{2}({\rm dim}\,L-{\rm dim}\,T_w)}\mathcal{R}_{\mathfrak{t}_w}^{\mathfrak{l}}(1_{\mathfrak{t}_w})$$Since $x_s$ is central in $\mathfrak{l}$, we deduce that $$\mathcal{F}^{\mathfrak{l}}\big(1_{x_s}\big)\cdot\mathcal{F}^{\mathfrak{l}}\big(\mathbf{X}_{\IC {\overline{\calO}{^L_{x_n}}}}\big)=q^{-\delta}|W_L|^{-1}\sum_{w\in
W_L}\tilde{\varphi}'(wF)\epsilon_L\epsilon_{T_w}q^{\frac{1}{2}({\rm dim}\,L-{\rm dim}\,T_w)}\mathcal{R}_{\mathfrak{t}_w}^{\mathfrak{l}}(\theta_{x_s}^w).$$From the transitivity property of Deligne-Lusztig induction  and the fact that $C_G(x)=C_L(x_n)$ we deduce that:
$$\mathcal{F}^{\mathfrak{g}}\big(\mathbf{X}_{\IC {\overline{\calO}}}\big)=\epsilon_G\epsilon_Lq^{\frac{1}{2}{\rm dim}\,\calO}|W_L|^{-1}\sum_{w\in
W_L}\tilde{\varphi}'(wF)\epsilon_L\epsilon_{T_w}\mathcal{R}_{\mathfrak{t}_w}^{\mathfrak{g}}(\theta_{x_s}^w).$$The map $W_L\rightarrow \{1,-1\}$, $w\mapsto \epsilon_L\epsilon_{T_w}$ is the sign character $\epsilon$ of $W_L$.
 \end{proof}

\begin{lemma} The functions $\mathcal{F}^{\mathfrak{g}}\big(\mathbf{X}_{\IC {\overline{\calO}}}\big)$ are $G^F$-invariant (i.e. constant on adjoint orbits) characters of the finite abelian group $(\mathfrak{g}^F,+)$.
 \end{lemma}

\begin{proof} The functions $\mathcal{F}^{\mathfrak{g}}\big(\mathbf{X}_{\IC {\overline{\calO}}}\big)$ are clearly $G^F$-invariant. The function $\mathcal{F}^{\mathfrak{g}}(1_\calO)$ is a sum of linear characters of $\mathfrak{g}^F$ and therefore is character of $\mathfrak{g}^F$. We thus need to see that if we write $\mathbf{X}_{\IC {\overline{\calO}}}=\sum_C n_C 1_C$ as a sum over the adjoint orbits of $\mathfrak{g}^F$, then $n_C\in\N$. 
 Let us use the notation introduced in the proof of Theorem \ref{charfor}. Write $$\mathbf{X}_{\IC{\overline{\calO}{^L}}}=1_{x_s}*\mathbf{X}_{\IC {\overline{\calO}{^L_{x_n}}}}=1_{x_s}*\left(\sum_{C'}n_{C'}1_{C'}\right)=\sum_{C'}n_{C'}1_{x_s+C'}$$where the sum runs over the nilpotent $L^F$-orbits of $\mathfrak{l}^F$ (note that $x_s+C'$ is an $L^F$-orbit of $\mathfrak{l}^F$ since $x_s$ is central). By Proposition \ref{letellierprop}(i), for a nilpotent adjoint orbit of $\mathfrak{l}^F$, the function $\mathcal{R}_\mathfrak{l}^{\mathfrak{g}}(1_{x_s+C'})$ is the characteristic function of the $G^F$-orbit of an element in $x_s+C'$. By Formula (\ref{induc}) we are reduced to see that $n_{C'}\in\N$. We have $L^F\simeq \prod_i\GL_{n_i}(\F_{q^{d_i}})$ for some $n_i,d_i\in\N$, and so $\mathbf{X}_{\IC {\overline{\calO}{^L_{x_n}}}}$ is a product of functions of the form $\mathbf{X}_{\IC {\overline{\calO}_i}}$ on $\gl_{n_i}(\F_{q^{d_i}})$ where $\calO_i$ is a nilpotent orbit of $\gl_{n_i}(\overline{\F}_q)$. By Lusztig \cite{Greenpoly}, the values of the functions $\mathbf{X}_{\IC {\overline{\calO}_i}}$ are non-negative integers.
\end{proof}

\subsection{Generic characters and generic orbits}\label{gen}

Let $(L,\theta^L,\varphi)$ be a triple as in \S\ref{irrG} with $L$ an $F$-stable Levi subgroup, $\theta^L\in{\rm Irr}_{\rm reg}(L^F)$ and $\varphi\in{\rm Irr}\,(W_L)^F$ and let $\mathcal{X}$ be the associated irreducible character of $G^F$. Then we say that the $G^F$-conjugacy class of the pair $(L,\varphi)$ is the \emph{type} of $\mathcal{X}$. Similarly we define the \emph{type} of an adjoint orbit $\calO^F$ of $\mathfrak{g}^F$ as follows. Let $x\in\calO^F$ and let $M=C_G(x_s)$ and let $C^M$ be the $M$-orbit of $x_n\in\mathfrak{m}$. Then the $G^F$-conjugacy class of the pair $(M,C^M)$ is called the type of $\calO^F$.

 From the pair $(L,\varphi)$ we define $\omega=(d_1,\omega^1)(d_2,\omega^2)\cdots(d_r,\omega^r)\in\mathbf{T}_n$ as follows. There exist positive integers $d_i,n_i$ such that $L\simeq \prod_{i=1}^r\GL_{n_i}(\overline{\F}_q)^{d_i}$ and $L^F\simeq \prod_{i=1}^r\GL_{n_i}(\F_{q^{d_i}})$. The $F$-stable irreducible characters of $W_L$ correspond then to ${\rm Irr}\,(S_{n_1})\times\cdots\times {\rm Irr}\,(S_{n_r})$ and the later set is in bijection with $\calP_{n_1}\times\cdots\times\calP_{n_r}$ via Springer correspondence $\mathfrak{C}_\epsilon$ that sends the trivial character of $S_m$ to the partition $(1^m)$. If $q> n$, the set of types of irreducible characters of $G^F$ is thus parametrized by $\mathbf{T}_n$. Under this parameterisation, semisimple irreducible characters correspond to types of the form  $(d_1,(1^{n_1}))\cdots(d_r,(1^{n_r}))$ and unipotent characters to types of the form $(1,\lambda)$.

 From the pair $(M,C^M)$ we define $\tau=(d_1,\tau^1)(d_2,\tau^2)\cdots(d_r,\tau^r)\in\mathbf{T}_n$ as follows. There exist positive integers $d_i,n_i$ such that $M\simeq \prod_{i=1}^r\GL_{n_i}(\overline{\F}_q)^{d_i}$ and $M^F\simeq \prod_{i=1}^r\GL_{n_i}(\F_{q^{d_i}})$. The Jordan form of $C^M$ defines partitions $\tau^1,\dots,\tau^r$ of $n_1,\dots,n_r$ respectively. If $q\geq n$, the set of types of adjoint orbits of $\mathfrak{g}^F$ is thus parametrized by $\mathbf{T}_n$.

\begin{remark} Note that if $\calO^F$ is an orbit of $\mathfrak{g}^F$ of type $\omega=(d_1,\omega^1)\cdots(d_r,\omega^r)$, then in the sense of \S\ref{type} the $G$-orbit $\calO$ is of type  $$\tilde{\omega}:=\underbrace{\omega^1\cdots\omega^1}_{d_1}\underbrace{\omega^2\cdots\omega^2}_{d_2}\cdots\underbrace{\omega^r\cdots\omega^r}_{d_r}.$$In particular, the two notions coincide  if the eigenvalues of $\calO$ are in $\F_q$. 
\label{remgen}\end{remark}

\begin{definition} Let $\calO_1^F,\dots,\calO_k^F$ be $k$ adjoint orbits of $\mathfrak{g}^F$. We say that the tuple $(\calO_1^F,\dots,\calO_k^F)$ is \emph{generic} if $(\calO_1,\dots,\calO_k)$ is generic in the sense of Definition \ref{genorb}.
 \end{definition}
\vspace{.2cm}

Assume that $L$ is an $F$-stable Levi subgroup of $G$. We say that a linear additive character of $z_{\mathfrak{l}}^F$ is \emph{generic} if its restriction to $z_{\mathfrak{g}}^F$ is trivial and its restriction to $z_{\mathfrak{m}}^F$ is non-trivial for any proper $F$-stable Levi subgroup $M$ of $G$ which contains $L$.

Put $$(z_\mathfrak{l})_{\rm reg}:=\{x\in z_\mathfrak{l}\,|\, C_G(x)=L\}.$$

\noindent Let $\{(d_i,n_i)\}_{i=1,\dots,r}$ be pairs of positive integers such that $L\simeq \prod_{i=1}^r\left(\GL_{n_i}(\overline{\F}_q)\right)^{d_i}$ and $L^F\simeq \prod_{i=1}^r\GL_{n_i}(\F_{q^{d_i}})$.

Define $$K_L^o=\begin{cases}(-1)^{r-1}d^{r-1}\mu(d)(r-1)!&\,\,\text{  if }d_i=d\,\,\text{ for all }i.\\0&\,\, \text{ otherwise. }\end{cases}$$where $\mu$ is the ordinary M\"obius function. 

The proof of the following proposition is completely similar  to that of Proposition 4.2.1 in \cite{hausel-letellier-villegas}.

\begin{proposition} Let $\Gamma$ be a generic character of $z_\mathfrak{l}^F$. Then 

$$\sum_{z\in (z_\mathfrak{l})^F_{\rm reg}}\Gamma(z)=qK_L^o.$$

\label{gen1}\end{proposition}

For a group $H$, we denote by $Z_H$ its center.

\begin{lemma} Let $(\calO_1^F,\dots,\calO_k^F)$ be a generic tuple of adjoint orbits of $\mathfrak{g}^F$. Let $(L_i,\eta_i,\varphi_i)$ be a datum defining the character $\mathcal{F}^{\mathfrak{g}}\big(\mathbf{X}_{\IC {\overline{\calO}_i}}\big)$, see Remark \ref{rem5}. Then $\prod_{i=1}^k\big({^{g_i}}\eta_i\big)|_{z_{\mathfrak{m}}}$
is a generic character of $z_{\mathfrak{m}}^F$ for any $F$-stable Levi subgroup
$M$ of $G$ which satisfies the following condition: For all
$i\in\{1,\dots,k\}$, there exists $g_i\in G^F$ such that $Z_M$ is contained in $g_iL_ig_i^{-1}$.
\label{lemgen}\end{lemma}

\begin{proof} We may write $\eta_i=\mathcal{F}^{\mathfrak{l}_i}(1_{\sigma_i})$ where $\sigma_i\in z_{\mathfrak{l}_i}$ is the semisimple  part of an element of $\calO_i^F$. Note that $g_i\sigma_ig_i^{-1}$ is in the center of $g_i\mathfrak{l}_ig_i^{-1}$ and so it commutes with the elements of $z_\mathfrak{m}\subset g_i\mathfrak{l}_ig_i^{-1}$, i.e., $g_i\sigma_ig_i^{-1}\in C_\mathfrak{g}(z_\mathfrak{m})=\mathfrak{m}$.  Let $z\in z_{\mathfrak{m}}^F$. Then $$\prod_{i=1}^k\big({^{g_i}}\eta_i\big)(z)=\prod_{i=1}^k\mathcal{F}^{\mathfrak{l}_i}(1_{\sigma_i})(g_i^{-1}zg_i)=\prod_{i=1}^k\Psi\big(\mu(\sigma_i,g_i^{-1}zg_i)\big)=\prod_{i=1}^k\Psi\big(\mu(g_i\sigma_ig_i^{-1},z)\big)=\Psi\left(\mu\big(\sum_ig_i\sigma_ig_i^{-1},z\big)\right).$$If $z=\lambda.{\rm Id}\in z_{\mathfrak{g}}$, then $\mu\big(\sum_ig_i\sigma_ig_i^{-1},z\big)=\lambda\,{\rm Tr}\, \big(\sum_ig_i\sigma_ig_i^{-1}\big)=0$ by the first  genericity condition (see Definition \ref{genorb}). Let $L$ be an $F$-stable Levi subgroup such that $M\subsetneq L\subsetneq G$, i.e., such that $z_\mathfrak{g}\subsetneq z_\mathfrak{l}\subsetneq z_\mathfrak{m}$ and assume that $\prod_{i=1}^k\big({^{g_i}}\eta_i\big)|_{z_{\mathfrak{l}}}$ is trivial. There is a decomposition $\K^n=V_1\oplus V_2\oplus\cdots \oplus V_r$ such that $\mathfrak{l}\simeq \bigoplus_i \gl(V_i)$. Then any element $z\in z_\mathfrak{l}$ is a of the form $(\lambda_1.{\rm Id},\dots,\lambda_r.{\rm Id})$ for some $\lambda_1,\dots,\lambda_r\in\K$. Since $g_i\sigma_ig_i^{-1}\in \mathfrak{m}\subset\mathfrak{l}$ for all $i$, we may write $\sum_ig_i\sigma_ig_i^{-1}=(x_1,\dots,x_r)\in \gl(V_1)\oplus\cdots \oplus \gl(V_r)$. Since $\prod_{i=1}^k\big({^{g_i}}\eta_i\big)|_{z_{\mathfrak{l}}}$ is trivial we have $\sum_{i=1}^r\lambda_i{\rm Tr}\,(x_i)=0$ for all $\lambda_1,\dots,\lambda_r\in\K$. Hence ${\rm Tr}\,(x_i)=0$ for all $i=1,\dots,r$. This contradicts the second genericity assumption.\end{proof}

A linear character of $Z_L^F$ is said to be \emph{generic} if its restriction to $Z_G^F$ is trivial and its restriction to $Z_M^F$ is non-trivial for any $F$-stable proper Levi subgroup $M$ of $G$ such that $L\subset M$. 

Put $$(Z_L)_{\rm reg}:=\{x\in Z_L\,|\, C_G(x)=L\}.$$

We have the following proposition \cite[Proposition 4.2.1]{hausel-letellier-villegas}.

\begin{proposition}Let $\Gamma$ be a generic character of $Z_L^F$. Then
 
$$\sum_{z\in (Z_L)_{\rm reg}^F}\Gamma(z)=(q-1)K_L^o.$$
\label{gen2}\end{proposition}

\begin{definition} Let $\mathcal{X}_1,\dots,\mathcal{X}_k$ be
$k$-irreducible characters of $G^F$. For each $i$, let
$(L_i,\theta_i,\varphi_i)$ be a datum defining $\mathcal{X}_i$. We
say that the tuple $(\mathcal{X}_1,\dots,\mathcal{X}_k)$ is
\emph{generic} if $\prod_{i=1}^k\big({^{g_i}}\theta_i\big)|_{Z_M}$
is a generic character of $Z_M^F$ for any $F$-stable Levi subgroup
$M$ of $G$ which satisfies the following condition: For all
$i\in\{1,\dots,k\}$, there exists $g_i\in G^F$ such that $Z_M\subset
g_iL_ig_i^{-1}$. \label{genericcondi}\end{definition}

\begin{example} Let $\mu^1,\dots,\mu^k$ be $k$ partitions of $n$ and denote by $R_{\mu^1},\dots, R_{\mu^k}$ the corresponding unipotent characters of $G^F$ (see beginning of this section). Consider $k$ linear characters $\alpha_1,\dots,\alpha_k$ of $\F_q^\times$. For each $i$, put $\calX_i:=(\alpha_i\circ{\rm det})\cdot R_{\mu^i}$. Then $\calX_i$ is an irreducible character of $G^F$ of same type as $R_{\mu^i}$. Then according to Definition \ref{genericcondi}, the tuple $(\calX_1,\dots,\calX_k)$ is generic if and only if the size of the subgroup of ${\rm Irr}\,\F_q^\times$ generated by $\alpha_1\cdots\alpha_k$ equals $n$.

\end{example}

Given $\omhat=(\omega_1,\dots,\omega_k)\in\big(\mathbf{T}_n\big)^k$, and assuming that ${\rm char}\,(\F_q)$ does not divide the gcd of $\{|\omega_i^j|\}_{i,j}$ and that $q$ is large enough, we can always find a generic tuple $(\mathcal{X}_1,\dots,\mathcal{X}_k)$ of irreducible characters of $G^F$ of type $\omhat$. The proof of this is similar to the proof of the existence of generic tuples of conjugacy classes of $\GL_n$ of a given type, see \cite{hausel-letellier-villegas}. 

\begin{definition} We say that an adjoint orbit of $\mathfrak{g}^F$ (or an irreducible character of $G^F$) is \emph{split} if the degrees of its type are all equal to $1$.
 
\end{definition}

\subsection{Multiplicities in tensor products}

Let $(\mathcal{X}_1,\dots,\mathcal{X}_k)$ be a generic tuple of irreducible characters of $G^F$. Assume that there exists a generic tuple $(\calO_1^F,\dots,\calO_k^F)$ of adjoint orbits of $\mathfrak{g}^F$ of same type as $(\calX_1,\dots,\calX_k)$. We put $d_\bfO=(2g-2)n^2+2+\sum_i{\rm dim}\,\calO_i$ as in Corollary \ref{lissite}.

Let $\Theta:\g^F\rightarrow\kappa$ be given by $x\mapsto q^{gn^2+g\,{\rm dim}\, C_G(x)}$, and let $\Lambda: G^F\rightarrow\kappa$ be given by $x\mapsto q^{g\,{\rm dim}\, C_G(x)}$. If $g=1$, note that $\Lambda$ is the character of the representation of $G^F$ in the group algebra $\kappa[\g^F]$ where $G^F$ acts on $\g^F$ by conjugation. 

\begin{theorem} We have $$\left\langle \Lambda\otimes\mathcal{X}_1\otimes\cdots\otimes\mathcal{X}_k,1\right\rangle_{G^F}=\frac{q^{-d_\bfO/2}(q-1)}{|G^F|}\left\langle\Theta\otimes\mathcal{F}^{\mathfrak{g}}\big(\mathbf{X}_{\IC {\overline{\calO}_1}}\big)\otimes\cdots \otimes\mathcal{F}^{\mathfrak{g}}\big(\mathbf{X}_{\IC {\overline{\calO}_k}}\big),1\right\rangle_{\mathfrak{g}^F}.$$
\label{multicomp1}\end{theorem}

\begin{proof} For each $i=1,\dots,k$, let $(L_i,\theta_i,\varphi_i)$ be a datum defining $\mathcal{X}_i$. Then 

\begin{align*}|G^F|\left\langle \Lambda\otimes\mathcal{X}_1\otimes\cdots\otimes\mathcal{X}_k,1\right\rangle_{G^F}&=\sum_{x\in G^F}q^{g\,{\rm dim}\,C_G(x)}\prod_{i=1}^k\left(\epsilon_G\epsilon_{L_i}|W_{L_i}|^{-1}\sum_{w\in W_{L_i}}\tilde{\varphi}_i(wF)R_{T_w}^G(\theta_i)(x)\right)\\&=\prod_{i=1}^k\left(\epsilon_G\epsilon_{L_i}|W_{L_i}|^{-1}\right)\sum_{x\in G^F}q^{g\,{\rm dim}\,C_G(x)}\sum_{(w_1,\dots,w_k)\in W_{L_1}\times\cdots\times W_{L_k}}\prod_{i=1}^k\tilde{\varphi}_i(w_iF)R_{T_{w_i}}^G(\theta_i)(x)\\
&=\sum_{(w_1,\dots,w_k)\in W_{L_1}\times\cdots\times W_{L_k}}\left(\prod_{i=1}^k\epsilon_G\epsilon_{L_i}|W_{L_i}|^{-1}\tilde{\varphi}_i(w_iF)\right)\sum_{x\in G^F}q^{g\,{\rm dim}\,C_G(x)}\prod_{i=1}^kR_{T_{w_i}}^G(\theta_i)(x).\end{align*}

The type of $\calO_i$ is the $G^F$-conjugacy class of $(L_i,\calO_i^{L_i})$ where $\calO_i^{L_i}$ is an $F$-stable nilpotent orbit of $\mathfrak{l}_i$ that corresponds to $\varphi_i$ via Springer's correspondence. 

For $i=1,\dots,k$, let $(L_i,\eta_i,\varphi_i)$ be a datum defining $\calF^\mathfrak{g}({\bf X}_{\IC {\overline{\calO}_i}})$ as explained in Remark \ref{rem5}. Using Theorem \ref{charfor} we may proceed as above to get
\vspace{.2cm}

$\left\langle\Theta\otimes\mathcal{F}^{\mathfrak{g}}\big(\mathbf{X}_{\IC {\overline{\calO}_1}}\big)\otimes\cdots \otimes\mathcal{F}^{\mathfrak{g}}\big(\mathbf{X}_{\IC {\overline{\calO}_k}}\big),1\right\rangle_{\mathfrak{g}^F}$ \begin{align*}&=|\mathfrak{g}^F|^{-1}
\sum_{(w_1,\dots,w_k)\in W_{L_1}\times\cdots\times W_{L_k}}\left(\prod_{i=1}^k\epsilon_G\epsilon_{L_i}q^{\frac{1}{2}{\rm dim}\,\calO_i}|W_{L_i}|^{-1}\tilde{\varphi}_i(w_iF)\right)\sum_{x\in \mathfrak{g}^F}q^{gn^2+g\,{\rm dim}\,C_G(x)}\prod_{i=1}^kR_{\mathfrak{t}_{w_i}}^{\mathfrak{g}}(\eta_i)(x)\\&=q^{gn^2-n^2+\frac{1}{2}\sum_i{\rm dim}\,\calO_i}
\sum_{(w_1,\dots,w_k)\in W_{L_1}\times\cdots\times W_{L_k}}\left(\prod_{i=1}^k\epsilon_G\epsilon_{L_i}|W_{L_i}|^{-1}\tilde{\varphi}_i(w_iF)\right)\sum_{x\in \mathfrak{g}^F}q^{g\,{\rm dim}\,C_G(x)}\prod_{i=1}^kR_{\mathfrak{t}_{w_i}}^{\mathfrak{g}}(\eta_i)(x).\end{align*}Since $d_\bfO/2=gn^2-n^2+1+\frac{1}{2}\sum_i{\rm dim}\,\calO_i$, we need to see that:$$(q-1)\sum_{x\in \mathfrak{g}^F}q^{g\,{\rm dim}\,C_G(x)}\prod_{i=1}^kR_{\mathfrak{t}_{w_i}}^{\mathfrak{g}}(\eta_i)(x)=q\sum_{x\in G^F}q^{g\,{\rm dim}\,C_G(x)}\prod_{i=1}^kR_{T_{w_i}}^G(\theta_i)(x).$$Since the functions $R_{T_{w_i}}^G(\theta_i)$ and $R_{\mathfrak{t}_{w_i}}^{\mathfrak{g}}(\eta_i)$ are constant respectively on conjugacy classes and adjoint orbits, we need to verify that for a given type $\omega\in\mathbf{T}_n$:\begin{equation}(q-1)\sum_{x\sim\omega}\prod_{i=1}^kR_{\mathfrak{t}_{w_i}}^{\mathfrak{g}}(\eta_i)(x)=q\sum_{x\sim\omega}\prod_{i=1}^kR_{T_{w_i}}^G(\theta_i)(x).\label{qq}\end{equation}where $x\sim\omega$ means that the $G$-conjugacy class of $x$ is of type $\omega$. Let $(M,C)$ with $M$ an $F$-stable Levi subgroup and $C$ an $F$-stable nilpotent orbit of $\mathfrak{m}$ such that the $G^F$-conjugacy class of $(M,C)$ corresponds to $\omega$ as in \S \ref{gen}. Recall that $x\in\mathfrak{g}^F$ is of type $(M,C)$ if there exists $y$ in the $G^F$-orbit of $x$ such that $M=C_G(y_s)$ and $y_n\in C^F$. Similarly, an element $x\in G^F$ is of type $(M,C)$ if there exists $y$ in the $G^F$-orbit of $x$ such that $M=C_G(y_s)$ and $y_u-1\in C^F$.

\noindent Then the proof of Formula (\ref{qq}) reduces to the proof of the following identity: $$(q-1)\sum_{z\in(z_\mathfrak{m})_{\rm reg}^F}\prod_{i=1}^kR_{\mathfrak{t}_{w_i}}^{\mathfrak{g}}(\eta_i)(z+v)=q\sum_{z\in(Z_M)_{\rm reg}^F}\prod_{i=1}^kR_{T_{w_i}}^G(\theta_i)(zu)$$where $v$ is a fixed element in $C^F$ and $u=v+1$. By formulas (\ref{DLu1}) and (\ref{charfordef1}) we have

$R_{\mathfrak{t}_{w_i}}^{\mathfrak{g}}(\eta_i)(z+v)=|M^F|^{-1}\sum_{\{h\in G^F\,|\,z\in h\mathfrak{t}_{w_i}h^{-1}\}}Q_{hT_{w_i}h^{-1}}^M(u)\,\eta_i(h^{-1}zh)$, 

$R_{T_{w_i}}^G(\theta_i)(zu)=|M^F|^{-1}\sum_{\{h\in G^F\,|\,z\in hT_{w_i}h^{-1}\}}Q_{hT_{w_i}h^{-1}}^M(u)\,\theta_i(h^{-1}zh)$,

Since $C_G(z)=M$, we have $\{h\in G^F\,|\,z\in h\mathfrak{t}_{w_i}h^{-1}\}=\{h\in G^F\,|\,hT_{w_i}h^{-1}\subset M\}$. We thus have:

\begin{equation*}\sum_{z\in(z_\mathfrak{m})_{\rm reg}^F}\prod_{i=1}^kR_{\mathfrak{t}_{w_i}}^{\mathfrak{g}}(\eta_i)(z+v)=\sum_{h_1,\dots,h_k}\left(\prod_{i=1}^k|M^F|^{-1}Q_{h_iT_{w_i}h_i^{-1}}^M(u)\right)\sum_{z\in (z_\mathfrak{m})_{\rm reg}^F}\prod_{i=1}^k\eta_i(h_i^{-1}zh_i)\label{sec6eq1}\end{equation*}where the first sum runs over the set $\prod_{i=1}^k\{h\in G^F\,|\,hT_{w_i}h^{-1}\subset M\}$. Similarly we have 

\begin{equation*}\sum_{z\in(Z_M)_{\rm reg}^F}\prod_{i=1}^kR_{T_{w_i}}^G(\theta_i)(zu)=\sum_{h_1,\dots,h_k}\left(\prod_{i=1}^k|M^F|^{-1}Q_{h_iT_{w_i}h_i^{-1}}^M(u)\right)\sum_{z\in (Z_M)_{\rm reg}^F}\prod_{i=1}^k\theta_i(h_i^{-1}zh_i).\label{sec6eq2}\end{equation*}The inclusion $h_iT_{w_i}h_i^{-1}\subset M$ implies that $Z_M\subset h_iT_{w_i}h_i^{-1}\subset h_iL_ih_i^{-1}$. By Lemma \ref{lemgen}, the character $\big(\prod_{i=1}^k\!{^{h_i}\eta_i}\big)|_{z_\mathfrak{m}}$ is a generic character of $z_\mathfrak{m}$ and so by Proposition \ref{gen1} we have $$\sum_{z\in (z_\mathfrak{m})_{\rm reg}^F}\prod_{i=1}^k\eta_i(h_i^{-1}zh_i)=qK_M^o.$$Similarly, by Proposition \ref{gen2} we have $$\sum_{z\in (Z_M)_{\rm reg}^F}\prod_{i=1}^k\theta_i(h_i^{-1}zh_i)=(q-1)K_M^o.$$
\end{proof}

When the tuples $(\calX_1,\dots,\calX_k)$ and $(\calO_1^F,\dots,\calO_k^F)$ are not generic we do not have such a nice relation between mulitplicities. For instance let us choose $(\calX_1,\dots,\calX_k)$ and $(\calO_1^F,\dots,\calO_k^F)$ to be respectively unipotent and nilpotent of same type. With the notation in the proof of the theorem we have $L_i=G$ for all $i$ and the linear characters $\eta_i$ and $\theta_i$ are the trivial characters. Then 

$$\sum_{z\in (z_\mathfrak{m})_{\rm reg}^F}\prod_{i=1}^k\eta_i(h_i^{-1}zh_i)=\left|\left.(z_\mathfrak{m})^F_{\rm reg}\right.\right|,$$

$$\sum_{z\in (Z_M)_{\rm reg}^F}\prod_{i=1}^k\theta_i(h_i^{-1}zh_i)=\left|\left.(Z_M)^F_{\rm reg}\right.\right|.$$Hence, unlike the generic case, the relation between these two terms invloves the rational function $\frac{\left|\left.(z_\mathfrak{m})^F_{\rm reg}\right.\right|}{\left|\left.(Z_M)^F_{\rm reg}\right.\right|}$ which depends on $M$.  The independence of $M$ is crucial as we obtain the multiplicities by summing over $M$.

\subsection{Multiplicities and symmetric functions}\label{MSF}

\subsubsection{Definitions}

Consider $k$ separate sets $\x_1,\x_2,\dots,\x_k$ of infinitely many variables and denote by $\Lambda_k:=\Q(q)\otimes_\Z\Lambda(\x_1)\otimes_\Z\cdots\otimes_\Z\Lambda(\x_k)$ the ring of functions separately symmetric in each set $\x_1,\x_2,\dots,\x_k$ with coefficients in $\Q(q)$ where $q$ is an indeterminate. On $\Lambda(\x_i)$ consider the Hall pairing $\langle\,,\,\rangle_i$ that makes the set $\{m_\lambda(\x_i)\}_{\lambda\in\calP}$ of monomial symmetric functions and the set $\{h_\lambda(\x_i)\}_{\lambda\in\calP}$ of complete symmetric functions dual bases. On $\Lambda_k$, put $\langle\,,\,\rangle=\prod_i\langle\,,\,\rangle_i$. 

%Let $\x$ be another set of infinitely many variables  disjoint from the $\x_i$'s and denote by $\Lambda$ the ring $\Lambda(\x)\otimes_\Z\Q$ of symmetric functions in the variables $\x$. Recall that the set of power sums $\{p_n\}_{n\geq 1}$ generate the $\Q$-algebra $\Lambda$. For a positive integer $d$, we denote by $\x^d$ the set of variables $\{x_1^d,x_2^d,\dots\}$. Define the operation $\circ:\Lambda\times\Lambda_k[[T]]\rightarrow \Lambda_k[[T]]$ by the following two properties:(i) $p_n\circ f(\x_1,\cdots,\x_k,q,T)=f(\x_1^n,\cdots,\x_k^n,q^n,T^n)$ for all $n\geq 1$,(ii) $(\cdot,f):\Lambda_k[[T]]\rightarrow\Lambda_k[[T]]$ is a ring homomorphism for all $f\in\Lambda_k[[T]]$.

\noindent Consider $$\psi_n:\Lambda_k[[T]]\rightarrow\Lambda_k[[T]],\, f(\x_1,\dots,\x_k;q,T)\mapsto f(\x_1^n,\dots,\x_k^n;q^n,T^n)$$where we denote by $\x^d$ the set of variables $\{x_1^d,x_2^d,\dots\}$. The $\psi_n$ are called the \emph{Adams operations}. 

%They satisfy $\psi_1={\rm Id}$, and $\psi_n(a+b)=\psi_n(a)+\psi_n(b)$ for all $a,b\in\Lambda_k[[T]]$. 

%The complete graded ring $\Lambda_k[[T]]$ endowed with $\circ$ is thus a $\lambda$-ring. 

Define $\Psi:T\Lambda_k[[T]]\rightarrow T\Lambda_k[[T]]$ by $$\Psi(f)=\sum_{n\geq 1}\frac{\psi_n(f)}{n}.$$Its inverse is given by $$\Psi^{-1}(f)=\sum_{n\geq 1}\mu(n)\frac{\psi_n(f)}{n}$$where $\mu$ is the ordinary M\"obius function. 

Following Getzler \cite{getzler} we define $\Log:1+T\Lambda_k[[T]]\rightarrow T\Lambda_k[[T]]$ and its inverse $\Exp:T\Lambda_k[[T]]\rightarrow 1+\Lambda_k[[T]]$ as 

$$\Log(f)=\Psi^{-1}\left(\log(f)\right)$$and $$\Exp(f)=\exp\left(\Psi(f)\right).$$

\subsubsection{Cauchy function}\label{cauchy}

For an infinite set of variable $\x$, the transformed Hall-Littlewood symmetric function $\tilde{H}_\lambda(\x,q)\in\Lambda(\x)\otimes_\Z\Q(q)$ is defined as 
$$\tilde{H}_\lambda(\x,q):=\sum_\lambda\tilde{K}_{\nu\lambda}(q)s_\nu(\x)$$where  $\tilde{K}_{\nu\lambda}(q)=q^{n(\lambda)}K_{\nu\lambda}(q^{-1})$ is the transformed Kostka polynomial \cite[III (7.11)]{macdonald}.

For a partition $\lambda$, put 
$$\mathcal{H}_\lambda(q):=\frac{q^{g\langle\lambda,\lambda\rangle}}{a_\lambda(q)}$$where $a_\lambda(q)$ denotes the cardinality of the centralizer of a unipotent element of $\GL_n(\F_q)$ with Jordan form of type $\lambda$ \cite[IV, (2.7)]{macdonald}. Define the $k$-points Cauchy function$$\Omega(q):=\sum_{\lambda\in\calP}\left(\prod_{i=1}^k\tilde{H}_{\lambda}(\x_i,q)\right)\mathcal{H}_\lambda(q)T^{|\lambda|}.$$It leaves in $1+T\Lambda_k[[T]]$. These functions were considered  by Garsia and Haiman \cite{garsia-haiman}. 

Given a family of symmetric functions $u_\lambda(\x,q)$ indexed by partitions, we extend its definition to a type $\omega=(d_1,\omega^1)\cdots(d_r,\omega^r)\in{\bf T}_n$ by $u_\omega(\x,q):=\prod_{i=1}^ru_{\omega^i}(\x^{d_i},q^{d_i})$. 

For a multi-type $\omhat=(\omega_1,\dots,\omega_k)\in\big(\mathbf{T}_n\big)^k$, put $u_{\omhat}:=u_{\omhat_1}(\x_1,q)\cdots u_{\omhat_k}(\x_k,q)\in\Lambda_k$. 

Recall that $\lambda'$ denotes the dual partition of $\lambda$. For a type $\omega=(d_1,\lambda_1)\cdots(d_r,\lambda_r)$, we denote by $\omega'$ the type $(d_1,\lambda_1')\cdots(d_r,\lambda_r')$. 

Let $\omhat=(\omega_1,\dots,\omega_k)\in\big(\mathbf{T}_n\big)^k$ with $\omega_i=(d_1^i,\omega_i^1)\cdots(d_{r_i}^i,\omega_i^{r_i})$ and define
 \begin{equation}\H_\omhat(q):=(-1)^{r(\omhat)}(q-1)\left\langle s_{\omhat'},\Log\,\big(\Omega(q)\big)\right\rangle\label{Hom}\end{equation}where $r(\omhat):=kn+\sum_{i,j}|\omega_i^j|$ and where  $\left\langle s_{\omhat'},\Log\,\big(\Omega(q)\big)\right\rangle$ is the Hall pairing of $s_{\omhat'}$ with the coefficient of $\Log\,\big(\Omega(q)\big)$ in $T^n$.

Note that if the degrees $d_i^j$ are all equal to $1$, then $r(\omhat)=2kn$.

We rewrite Formula (\ref{Hom}) in some special cases:

\subsubsection{The split semisimple case}\label{Omegass}

We say that $\omega\in{\bf T}_n$ is a \emph{semisimple type} if it is the type of a semisimple adjoint orbit of $\g_n^F$ (or equivalently the type of a semisimple character of $G^F$). It is then of the form $(d_1,(1^{n_1}))\cdots(d_r,(1^{n_r}))$. If moreover $\omega$ is split, i.e., $d_i=1$ for all $i$, then $\lambda=(n_1,\dots,n_r)$ is a partition of $n$ and any partition of $n$ is obtained in this way from a unique split semisimple type of ${\bf T}_n$. Note that for a split semisimple type $\omega$ with corresponding partition $\lambda$, we have $s_{\omega'}(\x)=h_\lambda(\x)$.

For a multipartition $\lambdahat=(\lambda_1,\dots,\lambda_k)\in(\calP_n)^k$ with corresponding split semisimple multitype $\omhat\in({\bf T}_n)^k$ we put $\H_\lambdahat^{ss}(q):=\H_\omhat(q)$. Then Formula (\ref{Hom}) reads

$$\H_\lambdahat^{ss}(q)=(q-1)\left\langle h_\lambdahat,\Log(\Omega(q))\right\rangle.$$

Since $\{h_\lambdahat\}$ and $\{m_\lambdahat\}$ are dual bases with respect to the Hall pairing, we may recover $\Omega(q)$ from $\H_\lambdahat^{ss}(q)$ by the formula 

\begin{equation}\Omega(q)=\Exp\left(\sum_{n\geq 1}\sum_{\lambdahat\in(\calP_n)^k}\frac{\H_\lambdahat^{ss}(q)}{q-1}m_\lambdahat T^n\right)\label{semisimple}.\end{equation}

\subsubsection{The nilpotent case}

We say that a type $\omega\in{\bf T}_n$ is nilpotent if it is the type of a nilpotent adjoint orbit of $\g^F$ (or the type of a unipotent character of $G^F$) in which case it is of the form $\omega=(1,\lambda)$ for some partition $\lambda$ of $n$, and $s_\omega(\x)=s_\lambda(\x)$.

For a multipartition $\lambdahat=(\lambda_1,\dots,\lambda_k)\in(\calP_n)^k$, we put $\H_\lambdahat^n(q):=\H_\omhat(q)$, where $\omhat=\left((1,\lambda_1),\dots,(1,\lambda_k)\right)$.

Since the base $\{s_\lambda\}_{\lambda\in\calP}$ is auto-dual, we recover $\Omega(q)$ from the $\H_\lambdahat^n(q)$ by the formula 

\begin{equation}\Omega(q)=\Exp\left(\sum_{n\geq 1}\sum_{\lambdahat\in(\calP_n)^k}\frac{\H_{\lambdahat'}^n(q)}{q-1}s_\lambdahat T^n\right).\label{nilpotent}\end{equation}

\subsubsection{The regular semisimple case}

We say that a type $\omhat\in{\bf T}_n$ is semisimple regular if it is the type of a semisimple regular adjoint orbit of $\G^F$ (or the type of an irreducible Deligne-Lusztig character, see \S\ref{irrG}). Then it is of the form $\omhat=(d_1,1)\cdots(d_r,1)$ and so $\lambda=(d_1,\dots,d_r)$ is a partition of $n$. In this case, the fonction  $s_\omhat(\x)$ is the power symmetric function $p_\lambda(\x)$. 

For a multipartition $\lambdahat$ with corresponding regular semisimple multitype $\omhat$, we use the notation $\H_\lambdahat^{rss}(q)$ and $r(\lambdahat)$ instead of $\H_\omhat(q)$ and $r(\omhat)$.

Recall that for any two partitions $\lambda,\mu$, we have $\langle p_\lambda(\x),p_\mu(\x)\rangle=z_\lambda\delta_{\lambda\mu}$.

Then we recover $\Omega(q)$ from $\H_\lambdahat^{rss}(q)$ by the formula 

\begin{equation}\Omega(q)=\Exp\left(\sum_{n\geq 1}\sum_{\lambdahat\in(\calP_n)^k}\frac{(-1)^{r(\lambdahat)}\H_{\lambdahat}^{rss}(q)}{(q-1)z_\lambdahat}p_\lambdahat T^n\right).\label{power}\end{equation}

\subsubsection{Multiplicities}\label{multi695}

Let $(\mathcal{X}_1,\dots,\mathcal{X}_k)$ be a generic tuple of irreducible characters of $G^F$ of type $\omhat=(\omega_1,\dots,\omega_k)\in (\mathbf{T}_n)^k$.

\begin{theorem} We have $$\left\langle\Lambda\otimes \calX_1\otimes\cdots\otimes\calX_k,1\right\rangle_{G^F}=\H_\omhat(q).$$
\label{compmulti}\end{theorem}

If the irreducible characters $\calX_1,\dots,\calX_k$ are all split semisimple with corresponding multipartition $\muhat\in(\calP_n)^k$,  then $\H_\omhat(q)=(q-1)\langle h_\muhat,\Log(\Omega(q))\rangle$ by \S\ref{Omegass}. Hence in the split semisimple case, this theorem is exactly \cite[Theorem 6.1.1]{hausel-letellier-villegas}. 

Since the main ingredient \cite[Theorem 4.3.1(2)]{hausel-letellier-villegas} in the proof of \cite[Theorem 6.1.1]{hausel-letellier-villegas} is available for any type $\omega\in{\bf T}_n$, we may follow line by line the proof of \cite[Theorem 7.1.1]{hausel-letellier-villegas} for arbitrary types (not necessarily split semisimple) to obtain the formula of  Theorem \ref{compmulti}.

\begin{remark}The theorem shows that the multiplicities of generic irreducible characters depend only on the types and not on the choices of irreducible characters of a given type. Note that $\H_\omhat(q)$ is clearly a rational function in $q$ with rational coefficients. On the other hand by Theorem \ref{compmulti}, it is also an integer for infinitely many values of $q$. Hence $\H_\omhat(q)$ is a polynomial in $q$ with rational coefficients.\label{dependence}\end{remark}

\section{Poincar\'e polynomials of quiver varieties and multiplicities}

Unless specified $\K$ is an arbitrary algebraically closed field.

For $i=1,\dots,k$ let $L_i,P_i,\sigma_i,C_i,\Sigma_i,\calO_i$ be as in \S \ref{intervar}.  Put $M_i:=C_{\GL_n}(\sigma_i)$ and ${\bf M}:=M_1\times\cdots\times M_k$.

We assume that $(\calO_1,\dots,\calO_k)$ is generic.

\subsection{Decomposition theorem and Weyl group action}\label{decompquiver}

Let $\rho:\mathbb{V}_{\bf L,P,\Sigma}\rightarrow\calV_\bfO$ and ${\rm p}:\mathbb{O}_{\bf L,P,\Sigma}\rightarrow \bfO$ be  the canonical projective maps (see  Diagram (\ref{mainpicture})). For an irreducible character $\chi=\chi_1\otimes\cdots\otimes\chi_k$ of the Weyl group $W_{\bf M}=W_{M_1}\times\cdots\times W_{M_k}$ we put $\bfO_\chi=(\gl_n)^{2g}\times\overline{\calO}_{\chi_1}\times\cdots\times\overline{\calO}_{\chi_k}$ where for each $i=1,\dots,k$, $\calO_{\chi_i}$ is the unique adjoint orbit contained in $\overline{\calO}_i$ corresponding to the character $\chi_i$ via the Springer correspondence $\mathfrak{C}$. 

By  Proposition \ref{Springer}, we have \begin{equation}{\rm p}_*\left(\pIC {\mathbb{O}_{\bf L,P,\Sigma}}\right)\simeq \pIC {\bfO}\oplus\left(\bigoplus_{\chi\in({\rm Irr}\,W_{\bf M})^*} A_\chi\otimes \pIC {\bfO_\chi}\right)\label{A1ic}\end{equation}where $({\rm Irr}\,W_{\bf M})^*:=({\rm Irr}\,W_{\bf M})-\{\chi_o\}$ and $$A_\chi={\rm Hom}_{W_{\bf M}}\left({\rm Ind}_{W_{\bf L}}^{W_{\bf M}}(V_{\bf C}),V_\chi\right)$$with $V_{\bf C}:=\bigotimes_iV_{C_i}$.

\begin{proposition} We have 
\begin{equation}(\rho/_{\PGL_n})_*\left(\pIC {\mathbb{Q}_{\bf L,P,\Sigma}}\right)\simeq \pIC {\calQ_\bfO}\oplus\left(\bigoplus_{\chi\in({\rm Irr}\,W_{\bf M})^*} A_\chi\otimes \pIC {\calQ_{\bfO_\chi}}\right). \label{A3ic}\end{equation}
\label{A3theo}\end{proposition}

The action of $W_{\bf M}({\bf L,C})$ on the $A_\chi$'s (see \S \ref{actionW}) induces thus an action of $W_{\bf M}({\bf L,C})$ on the complex $(\rho/_{\PGL_n})_*\left(\pIC {\mathbb{Q}_{\bf L,P,\Sigma}}\right)$ and so  on the hypercohomology $\mathbb{H}_c^i\left(\mathbb{Q}_{\bf L,P,\Sigma},\IC {\mathbb{Q}_{\bf L,P,\Sigma}}\right)=IH_c^i(\mathbb{Q}_{\bf L,P,\Sigma},\kappa)$. For $v\in W_{\bf M}({\bf L,C})$, we denote by $\theta_v:(\rho/_{\PGL_n})_*\left(\pIC {\mathbb{Q}_{\bf L,P,\Sigma}}\right)\simeq (\rho/_{\PGL_n})_*\left(\pIC {\mathbb{Q}_{\bf L,P,\Sigma}}\right)$ the corresponding automorphism.

\begin{proof}[Proof of Proposition \ref{A3theo}] By applying the proper base change to the top right square of the diagram (\ref{mainpicture}), it follows from the isomorphism (\ref{A1ic}) and Theorem \ref{restriction} that \begin{equation}\rho_*\left(\pIC {\mathbb{V}_{\bf L,P,\Sigma}}\right)\simeq \pIC {\calV_\bfO}\oplus\left(\bigoplus_{\chi\in({\rm Irr}\,W_{\bf M})^*} A_\chi\otimes \pIC {\calV_{\bfO_\chi}}\right). \label{A2ic}\end{equation}

Since the quotient maps $p_1:\mathbb{V}_{\bf L,P,\Sigma}\rightarrow \mathbb{Q}_{\bf L,P,\Sigma}$ and $p_2:\calV_\bfO\rightarrow\calQ_\bfO$ are principal $\PGL_n$-bundles they are smooth and so we have $(p_2)^*\left(\IC {\calQ_\bfO}\right)\simeq \IC {\calV_\bfO}$ and $(p_1) ^*\left(\IC {\mathbb{Q}_{\bf L,P,\Sigma}}\right)\simeq \IC {\mathbb{V}_{\bf L,P,\Sigma}}$. Applying the decomposition theorem to $\rho/_{\PGL_n}$ (Theorem \ref{BBDG}) and the base change theorem we see that if $\IC {Z,\zeta}[r]$ is a direct summand of $(\rho/_{\PGL_n})_*\left(\pIC {\mathbb{Q}_{\bf L,P,\Sigma}}\right)$ then $(p_2)^*\left(\IC {Z,\zeta}\right)=\IC {p_2^{-1}(Z),(p_2)^*(\zeta)}$ is (up to a shift) a direct summand of $\rho_*\left(\pIC {\mathbb{V}_{\bf L,P,\Sigma}}\right)$ and so we must have $Z=\calQ_{\bf O_\chi}$ for some $\chi$ and $\zeta=\kappa$. It is also clear that $\pIC {\calQ_{\bfO_\chi}}$ appears in  $(\rho/_{\PGL_n})_*\left(\pIC {\mathbb{Q}_{\bf L,P,\Sigma}}\right)$ with the same multiplicity as  $\pIC {\calV_{\bfO_\chi}}$ in $\rho_*\left(\pIC {\mathbb{V}_{\bf L,P,\Sigma}}\right)$.\end{proof}

Recall that $d_\bfO$ denotes the dimension of $\calQ_\bfO$. Put $r_\chi=(d_{\bfO_\chi}-d_\bfO)/2$. 

When $({\bf L,P,\Sigma})$ is defined over $\F_q$ the Proposition \ref{A3theo} can be made more precise as follows.

\begin{proposition} If $\K=\overline{\F}_q$ and if $({\bf L,P,\Sigma})$ is defined over $\F_q$ , then the  isomorphism \begin{equation*}(\rho/_{\PGL_n})_*\left(\pIC {\mathbb{Q}_{\bf L,P,\Sigma}}\right)\simeq \pIC {\calQ_\bfO}\oplus\left(\bigoplus_{\chi\in({\rm Irr}\,W_{\bf M})^*} A_\chi\otimes \pIC {\calQ_{\bfO_\chi}}(r_\chi)\right).\end{equation*}is defined over $\F_q$. In particular for $v\in W_{\bf M}({\bf L,C})$, we have

\begin{equation}{\bf X}_{(\rho/_{\PGL_n})_*\left(\pIC {\mathbb{Q}_{\bf L,P,\Sigma}}\right),\theta_v\circ\tilde{\varphi}}={\bf X}_{\pIC {\calQ_\bfO}}+\sum_{\chi\in({\rm Irr}\,W_{\bf M})^*}{\rm Tr}\,(v\,|\, A_\chi)\,q^{-r_\ttauhat}{\bf X}_{\pIC {\calQ_{\bfO_\chi}}}\label{decompquiverq}\end{equation}where $\tilde{\varphi}:F^*\left(\pi_*\left(\pIC {\mathbb{V}_{\bf L,P,\Sigma}}\right)\right)\simeq\pi_*\left(\pIC {\mathbb{V}_{\bf L,P,\Sigma}}\right)$ is the canonical isomorphism induced by the unique isomorphism $\varphi: F^*\left(\pIC {\mathbb{Q}_{\bf L,P,\Sigma}}\right)\simeq \pIC {\mathbb{Q}_{\bf L,P,\Sigma}}$ which induces the identity on $\mathcal{H}_x^{-d_\bfO}\left(\pIC {\mathbb{Q}_{\bf L,P,\Sigma}}\right)$ when $x\in \mathbb{Q}_{\bf L,P,\Sigma}^o(\F_q)$.

\end{proposition}

\begin{proof}It follows from the last assertion of Proposition \ref{quotientgen} and the discussion at the end of \S \ref{actionW}.

\end{proof}

We can proceed as in G\"ottsche and Soergel \cite{Gottsche-Soergel} to prove the following proposition from the mixed Hodge module version of the isomorphism (\ref{A3ic}).

\begin{proposition} Assume $\K=\C$. Then 

\begin{equation} IH_c^i\big(\mathbb{Q}_{\bf L,P,\Sigma},\Q\big)\simeq IH_c^i\big(\calQ_\bfO,\Q\big)\oplus\left(\bigoplus_{\chi\in({\rm Irr}\,W_{\bf M})^*}A_\chi\otimes \left(IH_c^{i+2r_\chi}\big(\calQ_{\bfO_\chi},\Q\big)\otimes\Q(r_\chi)\right)\right)
 \label{decompMHS}\end{equation}is an isomorphism of mixed Hodge structures. 
 \end{proposition}

\subsection{A lemma}\label{A}

Assume that $\K=\overline{\F}_q$. Recall that $F:\gl_n\rightarrow\gl_n$ denotes the standard Frobenius endomorphism so that $(\gl_n)^F=\gl_n(\F_q)$.

Assume that $(\calO_1,\dots,\calO_k)$ is $F$-stable. We do not assume that the eigenvalues of the adjoint orbits $\calO_i$'s are in $\F_q$. 

\begin{lemma} We have $$|\PGL_n(\F_q)|\cdot\sum_{x\in\calQ_\bfO(\F_q)}\mathbf{X}_{\IC {\calQ_\bfO}}(x)=\sum_{x\in\calV_\bfO(\F_q)}\mathbf{X}_{\IC {\calV_\bfO}}(x)=\left\langle\Theta\otimes\mathcal{F}^{\gl_n}\big(\mathbf{X}_{\IC {\overline{\calO}_1}}\big)\otimes\cdots \otimes\mathcal{F}^{\gl_n}\big(\mathbf{X}_{\IC {\overline{\calO}_k}}\big),1\right\rangle_{\gl_n(\F_q)}$$where $\Theta:\gl_n(\F_q)\rightarrow\kappa$, $x\mapsto q^{gn^2+g\,{\rm dim}\, C_{\GL_n}(x)}$.
 \label{lemma}\end{lemma}

\begin{proof}  We continue to denote by $F$ the induced Frobenius endomorphism on $\calV_\bfO$. We will use the notation $\calV_\bfO^F$ instead of $\calV_\bfO(\F_q)$. Let $q:\calV_\bfO\rightarrow\calQ_\bfO$ be the quotient map.  Since $\PGL_n(\F_q)$ acts freely on $\calV_\bfO$ it induces an injective map $\calV_\bfO^F/\PGL_n(\F_q)\rightarrow\calQ_\bfO^F$. Since $\PGL_n(\overline{\F}_q)$ is connected, any $F$-stable orbit of $\calV_\bfO$ has a rational point. Hence the above map is also surjective. As $q$ is a principal $\PGL_n$-bundle we have $q^*(\IC {\calQ_\bfO})\simeq \IC {\calV_\bfO}$ and so $\mathbf{X}_{\IC {\calV_\bfO}}(x)=\mathbf{X}_{\IC {\calQ_\bfO}}(y)$ whenever $q(x)=y$. We thus deduce the first equality. 

If $i:\calV_\bfO\hookrightarrow\bfO$ denotes the inclusion, then by Proposition \ref{restrictionVw} we have $\IC {\calV_\bfO}=i^*\left(\IC {\bfO}\right)=i^*\left(\kappa^{\boxtimes\, 2g}\boxtimes\IC {\overline{\calO}_1}\boxtimes\cdots \boxtimes\IC {\overline{\calO}_k}\right)$ where $\kappa$ is the constant sheaf on $\GL_n$ and $\kappa^{\boxtimes\, 2g}:=\kappa\boxtimes\cdots\boxtimes\kappa$ ($2g$ times). Hence for $x=(a_1,b_1,\dots,a_g,b_g,x_1,\dots,x_k)\in\calV_\bfO^F$, we have $$\mathbf{X}_{\IC {\calV_\bfO}}(x)=\mathbf{X}_{\IC {\overline{\calO}_1}}(x_1)\cdots \mathbf{X}_{\IC {\overline{\calO}_k}}(x_k).$$For $z\in\gl_n^F$, put $$\Xi(z):=\sharp\left\{\left. (a_1,b_1,\dots,a_g,b_g)\in(\gl_n^F)^{2g}\right|\,\sum_i[a_i,b_i]=z\right\}.$$

Hence \begin{align*}\sum_{x\in\calV_\bfO^F}\mathbf{X}_{\IC {\calV_\bfO}}(x)&=\sum_{(x_1,\dots,x_k)\in\,\left(\overline{\calO}_1\times\cdots\times\overline{\calO}_k\right)^F}\Xi(-(x_1+\cdots+x_k))\,\mathbf{X}_{\IC {\overline{\calO}_1}}(x_1)\cdots \mathbf{X}_{\IC {\overline{\calO}_k}}(x_k)\\
&=\big(\Xi*\mathbf{X}_{\IC {\overline{\calO}_1}}*\cdots *\mathbf{X}_{\IC {\overline{\calO}_k}}\big)(0).\end{align*}By Formula (\ref{intfor}) we have $$|\gl_n^F|\cdot f(0)=\sum_{x\in\gl_n^F}\mathcal{F}^{\gl_n}(f)(x)$$for any $f\in\text{Fun}(\gl_n^F)$. We deduce that \begin{align*}\sum_{x\in\calV_\bfO^F}\mathbf{X}_{\IC {\calV_\bfO}}(x)&=|\gl_n^F|^{-1}\sum_{x\in\gl_n^F}\calF^{\gl_n}(\Xi)(x)\,\mathcal{F}^{\gl_n}\big(\mathbf{X}_{\IC {\overline{\calO}_1}}\big)(x)\cdots \mathcal{F}^{\gl_n}\big(\mathbf{X}_{\IC {\overline{\calO}_k}}\big)(x).\\&=\left\langle\calF^{\gl_n}(\Xi)\otimes\mathcal{F}^{\gl_n}\big(\mathbf{X}_{\IC {\overline{\calO}_1}}\big)\otimes\cdots \otimes\mathcal{F}^{\gl_n}\big(\mathbf{X}_{\IC {\overline{\calO}_k}}\big),1\right\rangle_{\gl_n^F}\end{align*}It remains to see that $\calF^{\gl_n}(\Xi)=\Theta$.

For $x\in\gl_n^F$, we have 

\begin{align*}\calF^{\gl_n}(\Xi)(x)&=\sum_y\Psi\left(\mu(x,y)\right)\Xi(y)\\
&=\sum_{(a_1,b_1,\dots,a_g,b_g)\in\,(\gl_n^F)^{2g}}\Psi\left(\mu\left(x,\sum_{i=1}^g[a_i,b_i]\right)\right)\\
&=\sum_{(a_1,b_1,\dots,a_g,b_g)\in\,(\gl_n^F)^{2g}}\prod_{i=1}^g\Psi\left(\mu(x,[a_i,b_i])\right)\\
&=\sum_{(a_1,b_1,\dots,a_g,b_g)\in\,(\gl_n^F)^{2g}}\prod_{i=1}^g\Psi\left(\mu(x,[a_i,b_i])\right)\\
&=\left(\sum_{a,b\in\gl_n^F}\Psi\left(\mu(x,[a,b])\right)\right)^g\\
&=\left(\sum_{a\in\gl_n^F}\sum_{b\in\gl_n^F}\Psi\left(\mu([x,a],b)\right)\right)^g\\
&=\left(|C_{\gl_n}(x)^F|\cdot |\gl_n^F|\right)^g=\Theta(x).
\end{align*}

 \end{proof}

\begin{proposition} Assume that ${\bf \Sigma}$ is a reduced to a point and that $({\bf L,P,\Sigma})$ is defined over $\F_q$. The varieties $\mathbb{V}_{\bf L,P,\Sigma}$ and $\mathbb{Q}_{\bf L,P,\Sigma}$ are polynomial count. Moreover, $$|\mathbb{Q}_{\bf L,P,\Sigma}(\F_q)|=\frac{|\mathbb{V}_{\bf L,P,\Sigma}(\F_q)|}{|\PGL_n(\F_q)|}.$$
\label{polycountresol}\end{proposition}

\begin{proof} The second assertion follows from the fact that $\PGL_n$ is connected and acts freely on $\mathbb{V}_{\bf L,P,\Sigma}$, see beginning of the proof of Lemma \ref{lemma}. 

We only prove the first assertion for $\mathbb{Q}_{\bf L,P,\Sigma}$ as the proof for $\mathbb{V}_{\bf L,P,\Sigma}$ will be similar.

Since ${\bf \Sigma}$ is a point we have $\mathbb{Q}_{\bf L,P,\Sigma}=\mathbb{Q}_{\bf L,P,\Sigma}^o$ and so the variety $\mathbb{Q}_{\bf L,P,\Sigma}$ is nonsingular by Corollary \ref{strat2}. Hence $\IC {\mathbb{Q}_{\bf L,P,\Sigma}}$ is the constant sheaf $\kappa$ concentrated in degree $0$. By Formula (\ref{decompquiverq}) applied with $v=1$, we thus have \begin{equation}\mathbf{X}_{(\rho/_{\PGL_n})_*(\kappa)}=\mathbf{X}_{\IC {\calQ_\bfO}}+\sum_{\chi\in({\rm Irr}\,W_{\bf M})^*}\left({\rm dim}\, A_\chi\right) q^{-r_\chi}\mathbf{X}_{\IC {\calQ_{\bf O_\chi}}}\label{sum0}.\end{equation}By Grothendieck trace formula we have $$\sum_{x\in \calQ_\bfO^F}\mathbf{X}_{(\rho/_{\PGL_n})_*(\kappa)}(x)=|\mathbb{Q}_{\bf L,P,\Sigma}(\F_q)|.$$By Lemma \ref{lemma}, Theorem \ref{multicomp1} and Theorem  \ref{compmulti}, we see that there exists a rational function $Q\in \Q(T)$ such that for any $r\in\Z_{>0}$ $$\sum_{x\in\calQ_\bfO^{F^r}}\mathbf{X}_{\IC {\calQ_\bfO}}(x)=Q(q^r).$$By integrating Formula (\ref{sum0}) over $\calQ_\bfO^F$, we deduce that $$|\mathbb{Q}_{\bf L,P,\Sigma}(\F_{q^r})|=P(q^r)$$for some $P\in\Q(T)$. Since $P(q^r)$ is an integer for all $r\in\Z_{>0}$, the rational function $P$ must be a polynomial with rational coefficients. 
\end{proof}

\subsection{The split case}\label{pure}

In order to use Theorem \ref{HLRpure} we assume that $\K=\C$.  As in \cite[Appendix 7.1]{hausel-letellier-villegas}, we may define a finitely generated ring extension $R$ of $\Z$ and a $k$-tuple of $R$-schemes  $(\mathfrak{O}_1,\dots,\mathfrak{O}_k)$ such that $\mathfrak{O}_i$ is a spreading out of $\calO_i$ and such that for any ring homomorphism $\varphi:R\rightarrow \F_q$ into a finite field $\F_q$, the tuple $\left(\mathfrak{O}_1^{\,\varphi}(\overline{\F}_q),\dots,\mathfrak{O}_k^{\,\varphi}(\overline{\F}_q)\right)$ is a generic tuple of  adjoint orbits of $\gl_n(\overline{\F}_q)$ of same type as $(\calO_1,\dots,\calO_k)$. Denote by $\mathfrak{V}_\bfO$ the $R$-scheme defined from  $(\mathfrak{O}_1,\dots,\mathfrak{O}_k)$ as $\calV_\bfO$ was defined from $(\calO_1,\dots,\calO_k)$ (in the semisimple case this is written in details in \cite[Appendix A]{hausel-letellier-villegas}), and let $\mathfrak{Q}_\bfO$ be the affine quotient $\mathfrak{B}_\bfO/\!/\PGL_n$. Then $\mathfrak{V}_\bfO$  is a spreading out of $\calV_\bfO$. Recall (see for instance Crawley-Boevey and van den Bergh \cite[Appendix B]{crawley-boevey-etal}) that the standard constructions of GIT quotients are compatible with base change for $R$ sufficiently ``large'', namely in our case we have $\mathfrak{Q}_\bfO^\varphi=\mathfrak{V}_\bfO^\varphi/\!/\PGL_n$ for any ring homomorphism $\varphi:R\rightarrow k$ into a field $k$.

\begin{theorem} The cohomology group $IH_c^i\big(\calQ_\bfO,\C\big)$ vanishes if $i$ is odd. For any ring homomorphism $\varphi:R\rightarrow \F_q$ we have $$P_c\left(\calQ_\bfO,q\right)=\sum_{x\in \mathfrak{Q}_\bfO^{\,\varphi}(\F_q)}\mathbf{X}_{\IC {\mathfrak{Q}_\bfO^{\,\varphi}(\overline{\F}_q)}}(x)$$where $P_c(X,q):=\sum_i{\rm dim}\left(IH_c^{2i}(X,\C)\right)q^i$.\label{polycount}
 \end{theorem}

\begin{theorem} If not empty, the variety $\calQ_\bfO$ is pure.\label{purequiv}\end{theorem}

\begin{proof}  Let $\thetahat$ be generic with respect to $\v_\bfO$. Since $\calQ_\bfO\neq\emptyset$, by Theorem \ref{nonemptymu}, we have $\calQ_\bfO^o\neq\emptyset$ and so $\mathfrak{M}_{\xihat_\bfO}^s(\v_\bfO)\simeq\calQ_\bfO^o$ is also not empty. The canonical projective map $\mathfrak{M}_{\xihat_\bfO, \thetahat}(\v_\bfO)\rightarrow \calQ_\bfO$ is then a resolution of singularities by Theorem \ref{resolquiver} and so the group $IH_c^i(\calQ_\bfO,\C)$ is a direct summand of $H_c^i(\mathfrak{M}_{\xihat_\bfO, \thetahat}(\v_\bfO),\C)$ as a mixed Hodge structure. By Theorem \ref{HLRpure}, the variety $\mathfrak{M}_{\xihat_\bfO, \thetahat}(\v_\bfO)$ is pure, hence so is $\calQ_\bfO$. \end{proof}

\begin{proof}[Proof of Theorem \ref{polycount}] By \S \ref{decompquiver} and Proposition \ref{polycountresol}, the variety $\calQ_\bfO$ satisfies the condition of Theorem \ref{Katz2}. Hence the theorem follows from Proposition \ref{rempure} and Theorem \ref{purequiv}. \end{proof}

Let $m:\tilde{{\bf T}}_n\rightarrow {\bf T}_n$ be the map that sends  $\tomega=\omega^1\cdots\omega^r\in\tilde{{\bf T}}_n$ to $(1,\omega^1)\cdots(1,\omega^r)\in{\bf T}_n$, and denote by $m^k$ the map $(m,\dots,m):(\tilde{{\bf T}}_n)^k\rightarrow ({\bf T}_n)^k$. 

Recall (see \S \ref{gen}) that a generic tuple of irreducible characters of $\GL_n(\F_q)$ of a given type $\omhat\in({\bf T}_n)^k$ always exists assuming that ${\rm char}(\F_q)$ and $q$ are large enough.

 We have the following relation between multiplicities and Poincar\'e polynomials of quiver varieties.

\begin{theorem} Let $\tomhat$ be the type of $(\calO_1,\dots,\calO_k)$ and let $\F_q$ be a finite field such that there exists a ring homomorphism $R\rightarrow\F_q$. Then for any generic tuple   $(\calX_1,\dots,\calX_k)$ of irreducible characters of $\GL_n(\F_q)$ of type $m^k(\tomhat)$ we have

$$P_c(\calQ_\bfO,q)=q^{d_\bfO/2}\left\langle \Lambda\otimes \calX_1\otimes\cdots\otimes\calX_k,1\right\rangle.$$
 \label{multiPP}

\end{theorem}

\begin{remark}In the above theorem the existence of a ring homomorphism $R\rightarrow\F_q$ guaranty the existence of a generic tuple of irreducible characters of $\GL_n(\F_q)$.\end{remark}

\begin{proof}[Proof of Theorem \ref{multiPP}] Fix a ring homomorphism $\varphi:R\rightarrow \F_q$. To alleviate the notation we use $\mathfrak{O}_i$ instead of $\mathfrak{O}_i^\varphi(\overline{\F}_q)$. From Theorem \ref{polycount} and Lemma \ref{lemma}, we have 

$$P_c(\calQ_\bfO,q)=\frac{1}{|\PGL_n(\F_q)|}\left\langle \Theta\otimes \calF^{\gl_n}\left(\mathbf{X}_{\IC {\overline{\mathfrak{O}}_1}}\right)\otimes\cdots\otimes\calF^{\gl_n}\left(\mathbf{X}_{\IC {\overline{\mathfrak{O}}_k}}\right), 1\right\rangle.$$Hence Theorem \ref{multiPP} follows from Theorem \ref{multicomp1}
 \end{proof}

From the above theorem and  Theorem \ref{compmulti} we deduce the following result.

\begin{corollary}  Assume that $(\calO_1,\dots,\calO_k)$ is of type $\tomhat\in(\tilde{\bf T}_n)^k$. Then $$P_c\left(\calQ_\bfO,q\right)=q^{d_\bfO/2}\mathbb{H}_{m^k(\tomhat)}(q).$$\end{corollary}
 
\subsection{The general case}\label{generalcase}

Here $\K=\C$. Fix $\w\in W_{\bf M}({\bf L,C})$  and put 

$$P_c^\w\left(\mathbb{Q}_{\bf L,P,\Sigma};q\right):=\sum_i{\rm Tr}\left(\w\,\left|\, IH_c^{2i}\left(\mathbb{Q}_{\bf L,P,\Sigma},\C\right)\right.\right)q^i.$$

We now explain how to associate a multitype $\omhat=(\omega_1,\dots,\omega_k)\in({\bf T}_n)^k$ from the triple $({\bf L,C},\w)$.

Let $w_i$ be the coordinate of $\w$ in $W_{M_i}(L_i,C_i)$. In \S \ref{adjoint} we showed how to associate to $(L_i,C_i)$ a type $\tomega_i\in\tilde{{\bf T}}_n$. Write 

$$\tomega_i=\underbrace{\omega_i^1\cdots\omega_i^1}_{d_{i,1}}\underbrace{\omega_i^2\cdots\omega_i^2}_{d_{i,2}}\cdots\underbrace{\omega_i^{r_i}\cdots\omega_i^{r_i}}_{d_{i,r_i}}$$with $\omega_i^j\neq\omega_i^s$ if $j\neq s$. The group $W_{\GL_n}(L_i,C_i)$ is then isomorphic to $W_{\tomega_i}=\prod_{j=1}^{r_i}S_{d_{i,j}}$ and so the conjugacy classes of $W_{\GL_n}(L_i,C_i)$ are in bijection with $\mathfrak{H}^{-1}(\tomega_i)\subset {\bf T}_n$, see \S \ref{partype}. Hence to $w_i\in W_{M_i}(L_i,C_i)\subset W_{\GL_n}(L_i,C_i)$ corresponds a unique element in $\mathfrak{H}^{-1}(\tomega_i)$ which we denote by $\omega_i$.

\subsubsection{The main theorem}

Let $R$ be the finitely generated ring extension of $\Z$ considered in \S \ref{pure}.
 The main theorem of the paper is the following one.

\begin{theorem} Let $\F_q$ be a finite field such that there exists a ring homomorphim $R\rightarrow\F_q$. Let $(\calX_1,\dots,\calX_k)$ be a generic tuple of irreducible characters of $\GL_n(\F_q)$ of type $\omhat$. Then

$$P_c^\w\left(\mathbb{Q}_{\bf L,P,\Sigma};q\right)=q^{d_\bfO/2}\left\langle\Lambda\otimes \calX_1\otimes\cdots\otimes\calX_k,1\right\rangle.$$
\label{multiPP2}\end{theorem}

\begin{remark} Assume that $\w=1$, i.e., the degree of the types $\omega_i$ are all equal to $1$. By Theorem \ref{multiPP}, we have 
 
$$P_c(\calQ_{\mathbf{S}};q)=q^{d_{\mathbf{S}}/2}\left\langle\Lambda\otimes \calX_1\otimes\cdots\otimes\calX_k,1\right\rangle.$$where $\mathbf{S}=(\gl_n)^{2g}\times\overline{S}_1\times\cdots\overline{S}_k$ with $(S_1,\dots,S_k)$ a generic tuple of adjoint orbits of $\gl_n$ of type $\tomhat$. Hence by Theorem \ref{multiPP2} we have $$P_c\left(\mathbb{Q}_{\bf L,P,\Sigma};q\right)=P_c(\calQ_{\mathbf{S}};q).$$

%This is should be compared with the fact that if $\Sigma=\sigma+C$ then $$P_c(\mathbb{X}_{L,P,\Sigma};q)=P_c(\calC;q)$$where $\calC$ is an adjoint orbit of type $(L,C)$. The later fact is well-known at least when $\Sigma=\{0\}$ as in this case the fibres of the canonical map $\GL_n/L\rightarrow \GL_n/P$ are affine spaces isomorphic to the unipotent radical of $P$.

%These two situations are particular cases of a more general situation with quiver varieties. Indeed it is well-known by the work of Crawley-Boevey and van den Bergh \cite{crawley-boevey-etal} that for any  $\thetahat\in\Z^I$ the two quiver varieties $\mathfrak{M}_{0,\thetahat}(\v)$ and $\mathfrak{M}_{0,\thetahat}(\v)$ are fibres of the same locally trivial fibration (for the complex topology) and therefore have isomorphic intersection cohomology.

\end{remark}

From Theorem \ref{compmulti} we deduce the following identity.

\begin{corollary}$$P_c^\w\left(\mathbb{Q}_{\bf L,P,\Sigma};q\right)=q^{d_\bfO/2}\mathbb{H}_\omhat(q).$$
\end{corollary}

\subsubsection{Proof of Theorem \ref{multiPP2}}\label{lastpre}

By (\ref{decompMHS}) we have 

\begin{equation}P_c^\w(\mathbb{Q}_{\bf L,P,\Sigma};q)=P_c(\calQ_\bfO;q)+\sum_{\chi\in({\rm Irr}\, W_{\mathbf{M}})^*}{\rm Tr}\, (\w\,|\, A_\chi)q^{-r(\chi)} P_c(\calQ_{\bfO_\chi};q).\label{PPeq}\end{equation}To alleviate the notation, for each $\tauhat\in({\bf T}_n)^k$ we choose a generic tuple $(\calX_1,\dots,\calX_k)$ of irreducible characters of type $\tauhat$ and we put $R_\tauhat:=\calX_1\otimes\cdots\otimes\calX_k$. For $\ttauhat\in(\tilde{{\bf T}}_n)^k$ we denote $R_\ttauhat$ instead of $R_{m^k(\ttauhat)}$.

Now for each irreducible character $\chi$ of $W_{\bf M}$ we denote by $\ttauhat_\chi$ the type of $\bfO_\chi$ and we denote simply by $\ttauhat$ the type of $\bfO$. By Theorem \ref{multiPP} we have 

$$P_c(\calQ_{\bfO_\chi};q)=q^{d_{\bfO_\chi}/2}\left\langle \Lambda\otimes R_{\ttauhat_\chi},1\right\rangle.$$Hence we are thus reduced to prove the following identity

\begin{equation*}\left\langle \Lambda\otimes R_\omhat\right\rangle=\left\langle \Lambda\otimes R_\ttauhat\right\rangle+\sum_{\chi\in({\rm Irr}\, W_{\mathbf{M}})^*}{\rm Tr}\, (\w\,|\, A_\chi)\left\langle \Lambda\otimes R_{\ttauhat_\chi},1\right\rangle.\end{equation*}

By Theorem \ref{compmulti} we need to see that 

\begin{equation}\mathbb{H}_\omhat(q)=\mathbb{H}_\ttauhat(q)+\sum_{\chi\in({\rm Irr}\, W_{\mathbf{M}})^*}{\rm Tr}\, (\w\,|\, A_\chi)\,\mathbb{H}_{\ttauhat_\chi}(q)\label{Heq}\end{equation}where $\mathbb{H}_\ttauhat(q):=\mathbb{H}_{m^k(\ttauhat)}(q)$.

From the definition of $\mathbb{H}_\omega(q)$ (cf. Formula (\ref{Hom})) we are reduced to the following problem on Schur functions $\{s_\omega(\x)\}_{\omega\in{\bf T}_n}$:

Let $L,C,M,\calO, A_\chi$ be as in \S \ref{actionW}. For $\chi\in{\rm Irr}\, W_M$, denote by $\ttau_\chi\in\tilde{{\bf T}}_n$ the type of $\calO_\chi$ (with the convention that $\ttau_1=\ttau$). Let $\tomega\in\tilde{\bf T}_n$ be the type associated to $(L,C)$. Fix $w\in W_M(L,C)$ and let $\omega\in\mathfrak{H}^{-1}(\tomega)\in{\bf T}_n$ be the type corresponding to $(L,C,w)$. To prove Formula (\ref{Heq}) it is enough to prove the following identity

\begin{equation}(-1)^{r(\omega)}s_{\omega'}(\x)=s_{\ttau'}(\x)+\sum_{\chi\in({\rm Irr}\, W_M)^*}{\rm Tr}\, (w\,|\, A_\chi)s_{\ttau_\chi'}(\x)\label{Sfor}\end{equation}where for $\tnu=\nu^1\nu^2\cdots\nu^r\in\tilde{\bf T}_n$, $s_\tnu(\x):=s_{\nu^1}(\x)s_{\nu^2}(\x)\cdots s_{\nu^r}(\x)$ and where $r(\omega)=n+\sum_i|\omega^i|$.

We now explain how to get Formula (\ref{Sfor}) from Proposition \ref{twistedLRTR}.

We may assume that $L=\prod_{j=1}^r\left(\GL_{n_{j,1}}\times\cdots\times\GL_{n_{j,s_j}}\right)$ so that $M=\prod_{i=1}^r\GL_{m_i}$ and  $\GL_{n_{j,1}}\times\cdots\times\GL_{n_{j,s_j}}\subset\GL_{m_j}$.
Then the nilpotent orbit $C$ may be written as $$C=\prod_{j=1}^r \left(C_{j,1}\times\cdots\times C_{j,s_j}\right)$$with $C_{j,l}$ a nilpotent orbit of $\gl_{n_{j,l}}$. Let $\omega^{j,l}$ be the partition of $n_{j,l}$ given by the size of the Jordan blocks of $C_{j,l}$, and for each $j=1,2,\dots,r$, let $\tomega_j \in{\bf \tilde{T}}_{m_j}$ be the type given by the collection $\{\omega^{j,l}\}_{l=1,\dots,s_j}$. 

Then \begin{equation}W_M(L,C)\simeq\prod_{j=1}^rW_{\tomega_j}.\label{Weyldecomp}\end{equation}

Consider the map $\tilde{\mathfrak{F}}_r:\tilde{{\bf T}}_{m_1}\times\cdots\times\tilde{{\bf T}}_{m_r}\rightarrow\tilde{\bf T}_n$ where $\tilde{\mathfrak{F}}_r(\tmu_1,\dots,\tmu_r)$ is defined by re-ordering the partitions in the concatenation of the types $\tmu_1,\dots,\tmu_r$. 

\begin{example} Consider the lexicographic ordering on partitions. Then the image of $\left((3,2,1)(2,1),(3,1)\right)$ by  $\tilde{\mathfrak{F}}_2:\tilde{\bf T}_9\times\tilde{\bf T}_4\rightarrow\tilde{ \bf T}_{13}$ is $(3,2,1)(3,1)(2,1)$.
\end{example}

Similarly we define $\mathfrak{F}_r:{\bf T}_{m_1}\times\cdots\times{\bf T}_{m_r}\rightarrow{\bf T}_n$.

We denote by $S:\tilde{{\bf T}}\rightarrow\calP$ the map which assigns to a type  $\lambda^1\cdots\lambda^r\in\tilde{\bf T}$ the partition   $\sum_{i=1}^r\lambda^i$.

Consider the following commutative diagram

$$\xymatrix{{\bf T}_{m_1}\times\cdots\times{\bf T}_{m_r}\ar[r]^{\mathfrak{H}^r}\ar[d]^{\mathfrak{F}_r}&\tilde{{\bf T}}_{m_1}\times\cdots\times\tilde{{\bf T}}_{m_r}\ar[r]^{S^r}\ar[d]^{\tilde{\mathfrak{F}}_r}&\calP_{m_1}\times\cdots\times\calP_{m_r}\\
{\bf T}_n\ar[r]^{\mathfrak{H}}&\tilde{\bf T}_n&}.$$

Note that $\tomega=\tilde{\mathfrak{F}}_r(\tomega_1,\dots,\tomega_r)$. Let $w_j$ be the coordinate of $w\in W_M(L,C)$  in $W_{\tomega_j}$. The element $w_j$ defines a unique element $\omega_j\in\mathfrak{H}^{-1}(\tnu_j)\subset{\bf T}_{m_i}$. Then $\omega=\mathfrak{F}_r(\omega_1,\dots,\omega_r)$ and so

\begin{equation} s_\omega(\x)=s_{\omega_1}(\x)\cdots s_{\omega_k}(\x).\end{equation}

For each $i=1,2,\dots,r$, put $\tau^i=S(\tomega_i)\in\calP_{m_i}$. Note that the collection of the partitions $\tau^1,\dots,\tau^r$ gives the type $\ttau$ of $\calO$.

Now for each $i=1,2,\dots,r$, we have $$s_{\omega_i}(\x)=\sum_{\lambda\unlhd \tau^i}c_{\omega_i}^\lambda s_\lambda(\x)$$and so $$s_\omega(\x)=\sum_{(\lambda^1,\dots,\lambda^r)\unlhd(\tau^1,\dots\tau^r)}\left(\prod_ic_{\omega_i}^{\lambda^i}\right)s_{\lambda^1}(\x)\cdots s_{\lambda^r}(\x)$$where $(\lambda^1,\dots,\lambda^r)\unlhd(\tau^1,\dots\tau^r)$ means that $\lambda^i\unlhd\tau^i$ for all $i=1,\dots,r$. Note that the set of sequences $(\lambda^1,\dots\lambda^r)$ such that $(\lambda^1,\dots,\lambda^r)\unlhd(\tau^1,\dots\tau^r)$ is in bijection with the set $\{\ttau_\chi\,|\, \chi\in{\rm Irr}\, W_M(L,C)\}$. The bijection associates to a sequence $(\lambda^1,\dots\lambda^r)$ the unique type given by the collection of partitions $\lambda^1,\dots,\lambda^r$. Moreover if $(\lambda^1,\dots,\lambda^r)$ corresponds to $\chi$, we have $\prod_ic_{\omega_i}^{\lambda^i}={\rm Tr}\,(w\,|\, A_\chi)$ by Proposition \ref{twistedLRTR}, hence

\begin{align*}s_\omega(\x)&=\sum_{\chi\in{\rm Irr}\, W_M} {\rm Tr}\,(w\,|\, A_\chi)s_{\ttau_\chi}(\x)\\
 &=s_\ttau(\x)+\sum_{\chi\in({\rm Irr}\, W_M)^*}{\rm Tr}\,(w\,|\, A_\chi)s_{\ttau_\chi}(\x)
\end{align*}from which we deduce our Formula (\ref{Sfor}).

\subsubsection{Application to multiplicities in tensor products}

Assume that $(\calX_1,\dots,\calX_k)$ is a generic tuple of irreducible characters of type $\omhat$. Theorem \ref{multiPP2} has the following consequences.

\begin{theorem} We have:

\noindent (a)  The multiplicity $\langle\Lambda\otimes \calX_1\otimes\cdots\otimes\calX_k,1\rangle$ is a polynomial in $q$ of degree $d_\bfO/2$ with integer coefficients (with the convention that $d_\bfO=-\infty$ if $\calQ_\bfO=\emptyset$). If moreover the degrees of the characters $\calX_1,\dots,\calX_k$ are all split, then the coefficients of that polynomial are positive.

\noindent (b) The coefficient of $q^{d_\bfO/2}$ in $\langle\Lambda\otimes\calX_1\otimes\cdots\otimes\calX_k ,1\rangle$ equals $1$.

\noindent (c)  We have $\langle\Lambda\otimes \calX_1\otimes\cdots\otimes\calX_k,1\rangle\neq 0$ if and only if $\v_\bfO\in\Phi(\Gamma_\bfO)$. If $g=0$, then $\v_\bfO$ is a real root if and only if $\left\langle\calX_1\otimes\cdots\otimes\calX_k ,1\right\rangle=1$.

\noindent (d) If $g\geq 1$, we always have $\langle\Lambda\otimes \calX_1\otimes\cdots\otimes\calX_k,1\rangle\neq 0$.
\end{theorem}

\begin{proof} Let us first see that if $\calQ_\bfO\neq\emptyset$ then ${\rm dim}\, IH_c^{2d_\bfO}(\calQ_\bfO,\C)=1$. Consider a resolution $\mathbb{Q}_{\bf \hat{L},\hat{P},\{\sigma\}}\rightarrow\calQ_\bfO$. It is clear from Formula (\ref{PPeq}) applied to ${\bf \hat{L},\hat{P},\{\sigma\}}$ instead of  ${\bf L,P,\Sigma}$  that ${\rm dim}\, H_c^{2d_\bfO}(\mathbb{Q}_{\bf \hat{L},\hat{P},\{\sigma\}},\C)={\rm dim}\, IH_c^{2d_\bfO}(\calQ_\bfO,\C)$. But $\mathbb{Q}_{\bf \hat{L},\hat{P},\{\sigma\}}$ is irreducible by Theorem \ref{strat} and so ${\rm dim}\, H_c^{2d_\bfO}(\mathbb{Q}_{\bf \hat{L},\hat{P},\{\sigma\}},\C)=1$. 

It is thus clear from Formula (\ref{PPeq}) that $P_c^\w\left(\mathbb{Q}_{\bf L,P,\Sigma};q\right)$ is a polynomial in $q$ of degree $d_\bfO$ with integer coefficients and that the coefficient of $q^{d_\bfO}$ is equal to $1$. It is also clear that if $\w=1$, then the coefficients are positive. Hence  $q^{-d_\bfO/2}P_c^\w\left(\mathbb{Q}_{\bf L,P,\Sigma};q\right)=\langle\Lambda\otimes \calX_1\otimes\cdots\otimes\calX_k,1\rangle$ satisfies the assertions (a) and (b) of the theorem.

From what we just said it is clear that $\langle\Lambda\otimes \calX_1\otimes\cdots\otimes\calX_k,1\rangle\neq 0$ if and only if $\calQ_\bfO\neq\emptyset$. Hence the assertion (c) follows from Theorem \ref{nonemptymu} and Proposition \ref{real}.

Finally the assertion (d) follows from the assertion (c) and Proposition \ref{imaginary}.

\end{proof}

\end{document}